\documentclass{article}
\usepackage{amssymb}
\usepackage{amsmath}
\usepackage{amsfonts}
\usepackage{cite}
\usepackage{amsthm}
\usepackage{pdfsync}
\usepackage[english]{babel}
\usepackage{ulem}

\usepackage[letterpaper,top=2cm,bottom=2cm,left=3cm,right=3cm,marginparwidth=1.75cm]{geometry}

\usepackage{amsmath}
\usepackage{graphicx}
\usepackage{amsmath,amssymb,amsthm,amsfonts}
\usepackage{paralist}
\usepackage{graphics} 
\usepackage{epsfig} 
\usepackage{graphicx}
\usepackage{caption}
\usepackage{subcaption}
\usepackage{epstopdf}
\usepackage[colorlinks=true]{hyperref}
\usepackage{ulem}
\usepackage[title]{appendix}

\usepackage{txfonts}
\usepackage{enumerate}
 \usepackage{verbatim}

\date{}

\newtheorem{Theorem}{Theorem}[section]
\newtheorem{Lemma}{Lemma}[section]
\newtheorem{Remark}{Remark}[section]

\numberwithin{equation}{section} \theoremstyle{plain}

\def\R{{\textbf{R}}}

\def\p{\partial}

\def\R{\mathbb R}

\def\d{\mathrm{d}}

\def\va{\varepsilon}
\def\ds{\displaystyle}

\title {Multi-spots Steady States in Two-species Keller-Segel Models with Logistic Growth: Large Chemotactic Attraction Regime
 }

\author{ Fanze Kong \thanks{Department of Applied Mathematics, University of Washington, Seattle, WA 98195, USA; fzkong@uw.edu }, 
Juncheng Wei \thanks{Department of Mathematics, Chinese University of Hong Kong
 Shatin, NT, Hong Kong; wei@math.cuhk.edu.hk}
and 
Liangshun Xu \thanks{School of Mathematics and Information Sciences, Guangxi University, Nanning, P. R. China; lsxu@gxu.edu.cn }
      }

\begin{document}
        \date{}
\maketitle
\vspace{-0.4in}

\begin{abstract}
One of the most important findings in the study of chemotactic process is self-organized cellular aggregation, and a high volume of results are devoted to the analysis of a concentration of single species.  Whereas, the multi-species case is not understood as well as the single species one.  In this paper, we consider two-species chemotaxis systems with logistic source in a bounded domain $\Omega\subset \R^2.$  Under the large chemo-attractive coefficients and one certain type of chemical production coefficient matrices, we employ the inner-outer gluing approach to construct multi-spots steady states, in which the profiles of cellular densities have strong connections with the entire solutions to Liouville systems and their locations are determined in terms of reduced-wave Green's functions.  In particular, some numerical simulations and formal analysis are performed to support our rigorous studies.  
\end{abstract}

\bigskip

\section{Introduction and Main Results}\label{introductionsect}
In this paper, we consider the following two-species chemotaxis system with logistic growth in 2D:
\begin{align}\label{timedependent}
\left\{\begin{array}{ll}
u_{1t}=\Delta u_1-\chi_1\nabla\cdot(u_1\nabla v_1)+ \lambda_1 u_1(\bar u_1-u_1),&x\in\Omega,t>0,\\
u_{2t}=\Delta u_2-\chi_2\nabla\cdot(u_2\nabla v_2)+ \lambda_2 u_2(\bar u_2-u_2),&x\in\Omega,t>0,\\
v_{1t}=\Delta v_1-v_1+a_{11}u_1+a_{12}u_2,&x\in\Omega,t>0,\\
v_{2t}=\Delta v_2-v_2+a_{21}u_1+a_{22}u_2,&x\in\Omega,t>0,\\
\partial_{\textbf{n}}u_1=\partial_{\textbf{n}}u_2=\partial_{\textbf{n}}v_1=\partial_{\textbf{n}}v_2=0,&x\in\partial\Omega,t>0,\\
u_{1}(x,0)=u_{10}(x),~u_{2}(x,0)=u_{20}(x),~v_{1}(x,0)=v_{10}(x),~v_{2}(x,0)=v_{20}(x),&x\in\Omega,
\end{array}
\right.
\end{align}
where $\Omega$ is a bounded domain with the smooth boundary $\partial\Omega$ in $\mathbb R^2$, $\textbf{n}$ denotes the unit outer normal vector, $\chi_1>0$ and $\chi_2>0$ are chemo-attractive coefficients, $a_{ij}>0$ with $i,j=1,2$ are chemical production coefficients, and initial data $(u_{10},u_{20},v_{10},v_{20})$ is assumed to be smooth enough, non-negative and not identically equal to zero.  Here $u_1$ and $u_2$ are cellular densities of two species; $v_1$ and $v_2$ are chemical concentrations; $\lambda_1$ and $\lambda_2$ represent intrinsic cellular growth and $\bar u_1$ and $\bar u_2$ interpret the levels of carrying capacities.  Our goal in this paper is to construct multi-spots stationary solutions in (\ref{timedependent}) rigorously under an asymptotical limit $\chi_1, \chi_2\rightarrow +\infty$ with $\frac{\chi_1}{\chi_2}=O(1).$ 

\subsection{Chemotaxis and Keller-Segel Models}
Chemotaxis is a process in which uni-cellular or multi-cellular organisms direct their movements in response to chemical stimulus gradients.  This mechanism is ubiquitous in pathological and physiological processes such as morphogenesis, wound healing, tumor growth, cell differentiation, etc.  To describe the chemotactic phenomenon, E. Keller and L. Segel proposed a class of strongly coupled parabolic PDEs and some typical form reads as
 \begin{align}\label{proto}
\left\{\begin{array}{ll}
u_t= \overbrace{\Delta u}^{\text{cellular diffusivity}}-\overbrace{\chi \nabla\cdot (u\nabla v)}^{\text{chemotactic movement}}+\overbrace{f(u)}^{\text{ source}},&x\in\Omega,t>0,\\
v_t=\overbrace{\Delta v}^{\text{chemical diffusivity}}+\overbrace{g(u,v)}^{\text{chemical production/consumption}},&x\in\Omega,t>0,
\end{array}
\right.
 \end{align}
 where $\Omega$ is assumed to be a bounded domain in $\R^N$, $N\geq 1$ or the whole space.  Numerous studies are devoted to the analysis of spatial-temporal dynamics in one-population model (\ref{proto}) and we refer the readers to the survey papers \cite{hillen2009,horstmann2004,wang2013mathematics,painter2019mathematical}.  One of the most famous research results in the study of classical chemotaxis models is the so-called ``chemotactic collapse" \cite{childress1981,nanjundiah1973,horstmann2005,herrero1996,senba2000,senba2002,del2006}.  In particular, with regard to concentrated stationary solutions, it has been shown that cellular density $u$  asymptotically converges to the linear combination of several $\delta$-functions with some regular parts in 2D; meanwhile, the chemical concentration $v$ converges to the finite sum of Neumann Green's functions, where the coefficients are $8\pi$ or $4\pi$ depending on the locations of $\delta$-functions.  We mention that Del pino and Wei \cite{del2006} utilized the entire solution of the Liouville equation to approximate the cellular density $u$ asymptotically.

Compared to the comprehensive study of spontaneous concentration in the single-species Keller-Segel system, the localized patterns in the multi-species system are not well-understood and a few results devoted to the pattern formation in the multi-species counterpart, see \cite{wolansky2002multi,wolansky1997critical,wang2015time}.  We point out that Wang et al. in \cite{wang2015time} performed the bifurcation analysis around the constant steady state in the two-species competition Keller-Segel model and obtained the large chemotactic coefficient $\chi$ triggers \textit{Turing's instability} \cite{turing1990chemical}. 
Motivated by their results, we shall consider the ``far from" Turing's regime and assume chemotactic coefficients in (\ref{timedependent}) are large enough, then construct the multi-spots steady states asymptotically under the singular limits of $\chi_1\rightarrow +\infty$ and $\chi_2\rightarrow +\infty$  via the gluing method. 

To study the existence of non-constant steady states, we are concerned with the stationary problem of (\ref{timedependent}), which is
\begin{align}\label{ss}
\left\{\begin{array}{ll}
0=\Delta u_1-\chi_1\nabla\cdot(u_1\nabla v_1)+ \lambda_1 u_1(\bar u_1-u_1),&x\in\Omega,\\
0=\Delta u_2-\chi_2\nabla\cdot(u_2\nabla v_2)+ \lambda_2 u_2(\bar u_2-u_2),&x\in\Omega,\\
0=\Delta v_1-v_1+a_{11}u_1+a_{12}u_2,&x\in\Omega,\\
0=\Delta v_2-v_2+a_{21}u_1+a_{22}u_2,&x\in\Omega,\\
\partial_{\textbf{n}}u_1=\partial_{\textbf{n}}u_2=\partial_{\textbf{n}}v_1=\partial_{\textbf{n}}v_2=0,&x\in\partial\Omega,
\end{array}
\right.
\end{align}
 Before stating our main result, we introduce some notations and assumptions. For convenience, we set $\chi_1 = \chi$, $\chi_2 = \gamma \chi$ and $d = \frac{a_{21}}{a_{12}}\gamma^2$. Here and in the sequel, we impose the following assumptions on matrix  $A=(a_{ij})_{2\times 2}$: 
\begin{itemize}
    \item[(H1).] $(a_{ij})_{2 \times 2}$ is a real, irreducible and positive matrix;
\item[(H2).] $a_{11}a_{22}-a_{12}a_{21}\gamma^2\not=0;$
\item[(H3).] $(a_{ij})_{2\times 2}$ is a positive definite matrix.
\end{itemize}
Under the assumption (H1) and (H2), as shown in \cite{lin2010profile}, we have the existence of the entire solution denote $(\Gamma_{1,\mu_1}, \Gamma_{2,\mu_2})(y-\xi)$ with any $\xi\in\R^2$ to the following Liouville system 
\begin{equation}\label{gamma}
\left\{
\begin{aligned}
\Delta \Gamma_1 + b_{11}e^{\Gamma_1} + b_{12}e^{ \Gamma_2} = 0, \ \ \ \  &y\in\mathbb R^2,\\
\Delta \Gamma_2 + b_{21}e^{\Gamma_1} + b_{22}e^{ \Gamma_2} = 0, \ \ \ \  &y\in\mathbb R^2,
\end{aligned}
\right.
\end{equation}
where 
  \begin{equation}\label{bmatrix2025}
           (b_{ij})_{2 \times 2} =  
        \begin{pmatrix} 
             a_{11} & a_{21}\gamma  \\
             a_{21}\gamma   & a_{22}d 
             \end{pmatrix}
        \end{equation}
and one further defines 
  \begin{equation}\label{U0}
  (U_1, U_2) = (e^{\Gamma_{1,\mu_1}}, e^{\Gamma_2,\mu_2}).
  \end{equation}
 To capture the global behavior of $(v_1,v_2)$. we introduce the Neumann Green's function $G(x;\xi)$, which satisfies 
\begin{align}\label{Greenintro}
\left\{\begin{array}{ll}
\Delta_x G-G=- \delta_{\xi}(x),&x\in\Omega,\\
\partial_{\textbf{n}}G=0,&x\in\partial\Omega.
\end{array}
\right.
\end{align}
In addition, we define the regular part of Neumann Green's function $G(x;\xi)$ as $H(x;\xi)$, which solves
   \begin{align}\label{regular}
      \left\{ 
         \begin{array}{ll}
     - \Delta H + H  = - \frac{1}{2\pi} \log\frac{1}{|x - \xi|},&x\in\Omega, \\
        \frac{\partial H}{\partial \textbf{n}} = \frac{1}{2\pi} \frac{(x- \xi)\cdot \textbf{n}}{|x - \xi|^2},&x\in\partial\Omega.
        \end{array}
        \right.
   \end{align}
By employing the inner-outer gluing method, we extend the results shown in \cite{kong2022existence} and obtain the existence of multi-spots, which are
\begin{Theorem}\label{thm11}
Assume that $k$, $o$ are non-negative integers with $k+o\geq 1$ and $\lambda_j \bar{u}_j<\bar C_{j,\Omega}$, $j=1,2.$  Then for sufficiently large $\chi_1:=\chi:=\frac{1}{\va^2}$ with $\chi_2=\gamma\chi_1$ and given positive constant $\gamma$, there exists a solution $(u_{1,\chi},u_{2,\chi},v_{1,\chi}, v_{2,\chi})$ to (\ref{ss}) satisfying the following form:
\begin{align}\label{mainu}
 u_{j,\chi}(x)=\sum_{k =1}^mc_{jk}U_{jk}\bigg(\frac{x - \xi^{\va}_k}{\va};\mu_k\bigg) + o(1); 
\end{align}
 \begin{align}\label{mainv}
  {v_{j,\chi}}(x) =\va^2 \sum_{k =1}^{m}\bigg[-m_j\log\va+\Gamma_{j,\mu_{jk}}\bigg(\frac{x-\xi^{\va}_k}{\va}\bigg) +  \hat c_{jk} H(x, \xi^{\va}_k)-\mu_{jk}\bigg] + o(1),
 \end{align}
 where $m_j = \sum_{l=1}^2b_{jl}\sigma_l$, $H$ is defined as the regular part of Neumann Green's function satisfying (\ref{regular}), $U_{jk}$ and $\Gamma_{j,\mu_{jk}}$ are given by (\ref{U0}) and (\ref{gamma}), respectively.  Moreover, $\xi^{\va}_k\in \Omega$ and $\hat c_{jk}=2\pi m_j$ for $k\leq o$; $\xi^{\va}_k\in \partial\Omega$ and $\hat c_{jk}=\pi m_j$ for $o<k\leq m$, where $\sigma_{l}:=\frac{1}{2\pi}\int_{\mathbb R^N} e^{\Gamma_{l,\mu_{lk}}} dy$ and $b_{jl}$ are defined in \eqref{bmatrix2025}.  In addition, the $m$ -tuple $(\xi_1^{\va},\cdots,\xi_m^{\va})$ converges to a critical point of $\mathcal J_m$ as $\va\rightarrow 0,$ where $\mathcal J_m$ is defined by
 \begin{align}\label{jm}
 \mathcal J_m=\sum\limits_{k=1}^m \bar c_k^2H(x_k,x_k)+\sum_{k\not=l} \bar c_k\bar c_lG(x_k,x_l).
 \end{align}
Here $\bar c_k=2$ for $k\leq o$ and $\bar c_k=1$ for $o<k\leq m.$  In particular, the critical points of $\mathcal J_m$ are assumed to be non-degenerate and $\bar C_{j,\Omega}:=\sum\limits_{k=1}^m \hat c_{jk} C_{\Omega}$, where $C_{\Omega}$ is the positive lower bound of Green's function $G(x,y)$; $c_{jk}:=\frac{2\pi\sigma_j}{\int_{\mathbb R^2}e^{2 \Gamma_{j,\mu_{jk}}} \,dy }\bar u_j+O\Big(\frac{1}{\sqrt{\chi}}\Big)$ and $\mu_{jk}$ is determined by 
 \begin{equation*}
\mu_{jk} = \hat c_{jk} H(\xi^{\va}_k, \xi^{\va}_k) + \sum_{l \neq k}\hat c_{jl} G(\xi^{\va}_k, \xi^{\va}_l).
 \end{equation*} 
\end{Theorem}
Theorem \ref{thm11} demonstrates that when the intrinsic growth rate is small, (\ref{ss}) admits infinitely many interior and boundary multi-spots under the singular limits of large $\chi_1$ and $\chi_2$.  In Appendix \ref{appena}, we develop the formal construction of single interior spot, which supports our rigorous analysis for the proof of Theorem \ref{thm11}.   
Our main theoretical tool is the inner-outer gluing method, which has been used to study singularity formations within energy critical heat equations \cite{cortazar2020green}, harmonic map flows \cite{davila2020singularity}, Keller-Segel systems \cite{davila2024existence}, etc. successfully.  

\section{The Choice of Ansatz and Error Computations}\label{choiceansatz}

In this section, we shall discuss the choice of approximate solutions to (\ref{ss}).  In light of the $u_j$-equation with $j=1,2$, one regards logistic source terms as perturbations and obtains  
\begin{align}\label{relationuiandviintro}
u_j=C_je^{\chi_j v_j},
\end{align}
where constant $C_j>0$ will be determined later on.  Upon substituting (\ref{relationuiandviintro}) into the $v_j$-equation in (\ref{ss}), we define $\bar v_j=\chi_jv_j$ and arrive at
\begin{equation}\label{rescalebeforesingleeqintro}
\left\{
\begin{aligned}
0=\Delta \bar v_1-\bar v_1+a_{11}C_1\chi_1 e^{\bar v_1}+a_{12}C_2\chi_1e^{\bar v_2}, \ \ &x\in\Omega,\\
0=\Delta \bar v_2-\bar v_2+a_{21}C_1\chi_2e^{\bar v_1}+a_{22}C_2\chi_2e^{\bar v_2}, \ \  &x\in\Omega.
\end{aligned}
\right.
\end{equation}
Moreover, define $C_1\chi_1=\tilde \varepsilon^{m_1-2}$, $C_2\chi_2=d\tilde \varepsilon^{m_2-2}$ with $m_1$, $m_2$ and $d$ determined later on, then we have
\begin{equation}\label{rescalebeforesingleeqintro1}
\left\{
\begin{aligned}
0&=\Delta \bar v_1-\bar v_1+a_{11}\tilde \varepsilon^{m_1-2} e^{\bar v_1}+ \frac{d a_{12} }{\gamma}\tilde \varepsilon^{m_2-2}e^{\bar v_2}, \ \ \  &&x\in\Omega,\\
0&=\Delta \bar v_2-\bar v_2+  a_{21} \gamma\tilde \varepsilon^{m_1-2}e^{\bar v_1}+ a_{22} d \tilde \varepsilon^{m_2-2}e^{\bar v_2}, \ \ \ 
 &&x\in\Omega,
\end{aligned}
\right.
\end{equation}
where $\gamma:=\frac{\chi_2}{\chi_1}.$  
Without loss of generality, we assume that there is only one center $\xi\in \mathbb R^2$, then define $x-\xi=\tilde\varepsilon y$ and $\bar v_j=-m_j\log\tilde\varepsilon+\tilde V_{j}(y)$ with $m_j$ determined later on to obtain 
 from (\ref{rescalebeforesingleeqintro1}) that 
\begin{equation}\label{rescaledsingleeq}
\left\{
\begin{aligned}
0&=\Delta_y \tilde V_1-\tilde\varepsilon^2\tilde V_1+a_{11} e^{\tilde  V_1}+ a_{21}\gamma e^{\tilde  V_2}, \ \ \ 
  &&y\in\Omega_{\tilde\varepsilon},\\
0&=\Delta_y \tilde V_2-\tilde \varepsilon^2\tilde  V_2+  a_{21} \gamma e^{\tilde  V_1}+a_{22}d e^{\tilde  V_2}, \ \ \ 
    &&y\in\Omega_{\tilde\varepsilon},
\end{aligned}
\right.
\end{equation}
where $\Omega_{\tilde\varepsilon}:=(\Omega-\xi)/\tilde\varepsilon.$  By choosing $d$ such that the coefficient matrix in (\ref{rescaledsingleeq}) is symmetric, i.e.  
\begin{align}
\frac{a_{21}}{a_{12}} \gamma^2 = d,
\end{align}
we let $\tilde\varepsilon\rightarrow 0$ to get the limiting problem is
\begin{equation}\label{liouvillezhanglin}
\left\{
\begin{aligned}
0&=\Delta_y \Gamma_1 + b_{11}e^{\Gamma_1} + b_{12}e^{ \Gamma_2}, \ \  &&y\in\mathbb R^2,\\
0&=\Delta_y \Gamma_2 + b_{21}e^{\Gamma_1} + b_{22}e^{ \Gamma_2}, \ \  &&y\in\mathbb R^2,
\end{aligned}
\right.
\end{equation}
where $(\Gamma_1,\Gamma_2)$ is the leading approximation of $(\tilde V_1,\tilde V_2).$  Here
\begin{align}\label{bvaluenew2024104}
b_{11}:=a_{11},~b_{22}:=a_{22}d,~b_{21}=b_{12} = a_{21}\gamma.
\end{align}
Noting that
 $B=(b_{ij})_{2\times 2}$ is a real symmetric, irreducible, positive and invertible matrix,  one utilizes the results shown in \cite{lin2010profile} to obtain there exist a family of classical solutions $( \Gamma_{1,\tilde\mu_1},\Gamma_{2,\tilde\mu_2})$ to (\ref{liouvillezhanglin}) such that 
\begin{equation}\label{farfieldtildeVi0}
  \Gamma_{j,\tilde\mu_j}(y) = \Gamma_{j}(y, \tilde \mu_j) + \tilde\mu_j 
\end{equation}
 where $\tilde \mu_j$ are constants, $\hat{m} := \min\{m_1, m_2\} > 2$ and 
 \begin{equation*}
    \Gamma_{j}(y, \tilde\mu_j) = -m_j\log|y| + O(|y|^{2 - \hat{m}}) \ \ \ \   |y| \gg 1.
 \end{equation*}
   Denote $\sigma_j:=\frac{1}{2\pi}\int_{\mathbb R^2} e^{\Gamma_{j,\tilde\mu_j}} dy$, then we have $m_1$ and $m_2$ are determined by
\begin{equation}\label{m1m2decayrate}
m_1=\sigma_1b_{11}+\sigma_2b_{12}, \ \ \ \  m_2=\sigma_1b_{21}+\sigma_2b_{22}.
\end{equation}
Thus, for $j=1,2,$
\begin{align*}
\bar v_j=-m_j\log\tilde\varepsilon+\tilde V_{j}(y),\ \ \ \  \tilde V_{j}(y)=\Gamma_{j,\tilde\mu_j}(y) + o(1).
\end{align*}
We mention that by the blow-up analysis \cite{lin2010profile}, $m_j>2$ for all $j=1,2$.  Moreover, for $j=1,2$, the leading order term of $u_j$ is 
\begin{align*}
 u_j(x)=C_j\tilde\varepsilon^{-m_j}e^{\Gamma_{j,\tilde\mu_j}}(1 + o(1)).
\end{align*}
Noting that $C_1=\frac{\tilde\varepsilon^{m_1-2}}{\chi_1}$ and $C_2=d\frac{\tilde\varepsilon^{m_2-2}}{\chi_2}$, we have 
\begin{align}\label{Ui0leadingofspike}
u_1 =\frac{1}{\chi_1\tilde\varepsilon^2}e^{\Gamma_{1,\tilde \mu_1}}(1 + o(1)) := c_1 e^{\Gamma_{1,\tilde \mu_1}}   \ \ \text{and} \ \   u_2 =\frac{d}{\chi_2\tilde\varepsilon^2}e^{\Gamma_{2,\tilde\mu_2}}(1 + o(1)) := c_2e^{\Gamma_{2,\tilde\mu_2}}.
\end{align}
The leading part of $c_j$, $j=1,2$ is determined by the global balancing condition $\int_{\Omega} u_j(\bar u_j-u_j)dx=0$, which implies
\begin{align}\label{leadingci0}
c_{j} + o(1) =\frac{\int_{\mathbb R^2}e^{\Gamma_{j,\tilde\mu_j}}\,dy}{\int_{\mathbb R^2}e^{2 \Gamma_{j,\tilde\mu_j}} \,dy}\bar u_j=\frac{2\pi\sigma_j}{\int_{\mathbb R^2}e^{2 \Gamma_{j,\tilde\mu_j}} \,dy }\bar u_j.
\end{align}
In addition, thanks to Pohozaev identity shown in \cite{lin2010profile}, we find
\begin{align}\label{combine1sect2}
4(\sigma_1+\sigma_2)= b_{11}\sigma^2_1 + 2b_{12}\sigma_1\sigma_2+ b_{22}\sigma_2^2.
\end{align}
Combining (\ref{combine1sect2}) with (\ref{leadingci0}), one gets $(\sigma_1,\sigma_2)$ solves
\begin{align}\label{algebraicsystemforsigma}
\left\{
\begin{aligned}
\frac{\bar u_1}{\bar u_2}\int_{\mathbb R^2}e^{2 \Gamma_{2,\tilde\mu_2}}\,dy\sigma_1&=\frac{a_{12}}{a_{21}}\frac{\chi_1}{\chi_2}{\int_{\mathbb R^2}e^{2 \Gamma_{1,\tilde\mu_1}}\,dy}\sigma_2,\\
4(\sigma_1+\sigma_2) &=b_{11}\sigma_1^2+2b_{12}\sigma_1\sigma_2+b_{22}\sigma_2^2,
\end{aligned}
\right.
\end{align}
where $b_{ij}$, $i,j=1,2$ are given in \eqref{bvaluenew2024104}.  Noting that $(a_{ij})_{2\times 2}$ is a positive definite matrix by assumption (H3), one has the second constraint in \eqref{algebraicsystemforsigma} is an ellipse passing through $(0,0)$.  On the other hand, the first equality in (\ref{algebraicsystemforsigma}) can not be a closed curve since it cannot cross the coordinate axes by using the fact that all points must lie in the first quadrant. 
Therefore, the system \eqref{algebraicsystemforsigma} admits at least a positive solution $(\sigma_1, \sigma_2)$, where the schematic diagram is shown in Figure \ref{fig:ellipse}.
\begin{figure}
    \centering
    \includegraphics[width=0.9\linewidth,height=0.45\linewidth]{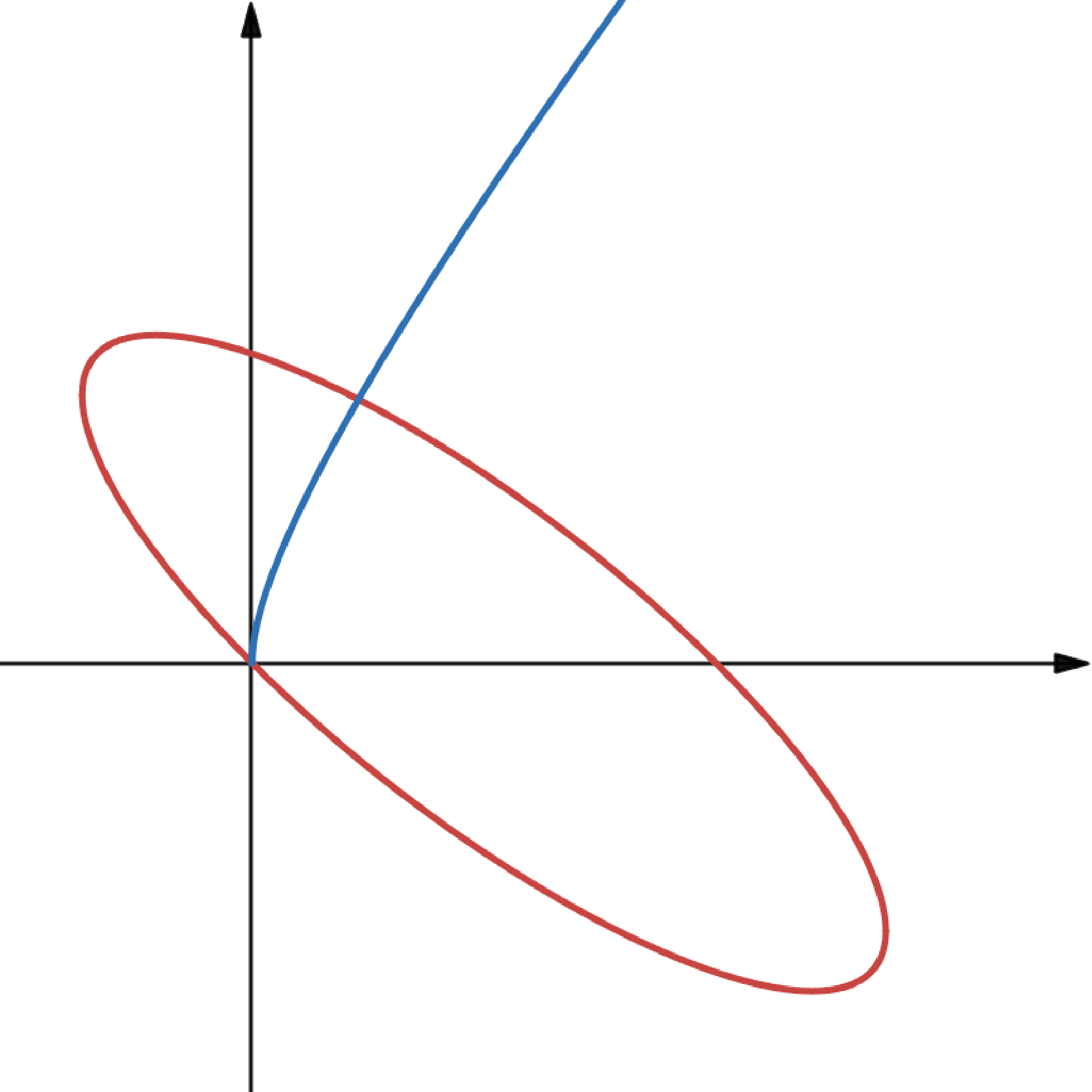}
    \caption{Schematic Diagram of (\ref{algebraicsystemforsigma})}
    \label{fig:ellipse}
\end{figure}

We further define the correction term as $H^{\tilde \va}_j(x;\xi)$, $j=1,2$, which satisfies 
\begin{align}
\left\{
\begin{aligned}
  \Delta_x H^{\tilde \va}_j-H^{\tilde \va}_j &= -m_j\log\tilde\va+\Gamma_{j,\tilde\mu_j}\Big(\frac{x-\xi}{\tilde\va}\Big),&x\in\Omega,\\
\partial_{\textbf{n}}H^{\tilde \va}_j &=-\partial_{\textbf{n}} \Gamma_{j,\tilde\mu_j},    &x \in\partial\Omega.
\end{aligned}
\right.
\end{align}
In summary, we set the rough approximation of the single spot in (\ref{ss}) as
\begin{equation}\label{uj0vj0before}
\left\{
  \begin{aligned}
u_{j} &= U_{j}(y, \tilde{\mu}_j):= c_je^{\Gamma_{j,\tilde\mu_j}}, \\
\bar v_{j}& =-m_j\log\tilde \varepsilon + \Gamma_{j} (y, \tilde{\mu}_j)+H^{\tilde \va}_j(x;\xi),
\end{aligned}
\right.
\end{equation}
where $y=\frac{x-\xi}{\tilde \varepsilon}$ and $\tilde \varepsilon:=\frac{1}{\sqrt{c_1\chi_1}}\ll 1$.  Here $c_j$ is determined by (\ref{leadingci0}), $(\sigma_1,\sigma_2)$ is the solution to (\ref{algebraicsystemforsigma}) and $(\Gamma_{1,\tilde\mu_1},\Gamma_{2,\tilde\mu_2})$ solves Liouville system (\ref{liouvillezhanglin}). 

It remains to determine the parameter $\tilde \mu_j$, $j=1,2.$ In light of (\ref{Greenintro}), similarly as shown in Lemma 2.1 of \cite{del2006}, we have for any $\alpha\in(0,1)$,
 \begin{align}\label{2162025117}
 H_{j}^{\tilde \va}(x;\xi)=\hat c_{j}H(x;\xi)- \tilde \mu_{j} + O(\tilde\va^{\alpha}),
 \end{align}
 where $\hat c_{j}=2\pi m_j$ if $\xi\in\Omega$ and $\hat c_{j}=\pi m_j$ if $\xi\in\partial\Omega.$  To guarantee the error is small, we choose
 \begin{align}\label{tildemujvaH}
\tilde \mu_j =\hat c_jH(\xi,\xi),
\end{align}
where $\hat c_j=2\pi m_j$ if $\xi\in\Omega$ and $\hat c_j=\pi m_j$ if $\xi\in\partial\Omega.$ 
 We remark that $\tilde\varepsilon$ depends on $c_j$ and $c_j$ is determined globally, which may cause the ambiguity for our subsequent analysis.  To solve this issue, we define $\varepsilon=\frac{1}{\sqrt{\chi}}\ll 1$ and rewrite (\ref{uj0vj0before}) as 
\begin{equation}\label{uj0vj0after}
\left\{
\begin{aligned}
u_{j} &= U_{j}(y; \mu_j), \\
\bar v_{j}& =-m_j\log\varepsilon + \Gamma_{j} (y, \mu_j)+ H^{\va}_j(x;\xi),
\end{aligned}
\right.
\end{equation}
where $y:=\frac{x-\xi}{\va}$, $ \mu_j=\hat c_j H  (\xi,\xi)$ and
$H^{\va}_j$ solves
\begin{equation}\label{Hvajsolves2024}
\left\{
\begin{aligned}
\Delta_x H^{\va}_j-H^{\va}_j &= -m_j\log\va+\Gamma_{j, \mu_j}\Big(\frac{x-\xi}{\va}\Big),  & x \in \Omega,\\
\partial_{\textbf{n}}H^{ \va}_j &=-\partial_{\textbf{n}} \Gamma_{j, \mu_j},  &x \in \partial\Omega.
\end{aligned}
\right.
\end{equation}

Proceeding with the similar argument shown above, we set the first approximation of multi-spots with $o$, $0\leq o\leq m,$ interior bubbles to (\ref{ss}) as 
  \begin{equation}\label{generate1}
\left\{\begin{aligned}
u_j &=  \sum_{k = 1}^m  U_{jk}(y-\xi_k', \mu_{jk}):= \sum_{k=1}^mc_{jk}e^{\Gamma_{j,\mu_{jk}}(y-\xi'_k)}, \\
\bar v_{j}& =\sum_{k=1}^m \Big(-m_j\log\varepsilon + \Gamma_{jk} (y-\xi_k', \mu_{jk})+H^{\va}_{jk}(x;\xi_k) \Big),
\end{aligned}
\right.
\end{equation} 
where $m\geq 1$, $\va:=\frac{1}{\sqrt{\chi}}$, $y:=\frac{x}{\va}$, $\xi_k':=\frac{\xi_k}{\va}$. In particular, $\mu_{jk}$, $j=1,2$ are determined by
\begin{equation}
      \mu_{jk} =  \hat c_{jk}H(\xi_k; \xi_k) + \sum_{l \neq k}\hat c_{jl} G(\xi_k;\xi_l),
   \end{equation}
where $G(x;\xi)$ is given by (\ref{Greenintro}) and $\hat c_{jk}=2\pi m_j$ if $k\leq o;$ $\hat c_{jk}=\pi m_j$ if $o<k\leq m$.  For the simplicity of our notation, we define 
 \begin{equation}
   U_{jk}(y - \xi'_k) = U_{jk}(y - \xi'_k, \mu_{jk})  \ \ \text{and} \ \     \Gamma_{jk}(y - \xi'_k)  =  \Gamma_{jk}(y - \xi'_k, \mu_{jk}),  . 
\end{equation}
then rewrite the first approximation solution as
 \begin{equation}\label{20252p20new}
\left\{\begin{aligned}
u_j &=  \sum_{k = 1}^m U_{jk}(y-\xi_k'), \\
\bar v_{j}& =\sum_{k=1}^m \Big(-m_j\log\varepsilon + \Gamma_{jk} (y-\xi_k') + H^{\va}_{jk}(x;\xi_k) \Big),
\end{aligned}
\right.
\end{equation} 
where $H_{jk}^{\va}$ solves (\ref{Hvajsolves2024}) with $\xi$ replaced by $\xi_k.$

Next, we compute the error generated by (\ref{20252p20new}).  
Noting that (\ref{ss}) can be reduced as the following two-coupled equations:
   \begin{equation}\label{224expression2024}
   \left\{
   \begin{aligned}
    S_1(u_1,u_2) &= \Delta u_1 + \nabla\cdot((\Delta_x-1)^{-1}(a_{11}u_1+a_{12}u_2)) + \lambda_1 u_1 (\bar u_1 - u_1)=0,\\
     S_2(u_1,u_2) &= \Delta u_2 + \nabla\cdot((\Delta_x-1)^{-1}(a_{21}u_1+a_{22}u_2)) + \lambda_2 u_2 (\bar u_2 - u_2)=0.  
        \end{aligned}
        \right.
      \end{equation}
Then in the region $|x - \xi_k| < \delta \varepsilon$ with constant $\delta>0$, we calculate to get
   \begin{equation}
     \left \{ 
       \begin{aligned}
          &\nabla_x u_{j} = \frac{1}{\va}\nabla U_{jk}(y-\xi_k') + \sum_{l\not \neq k} \frac{1}{\va}\nabla U_{jl}(y-\xi_l')+o(1),  \\
          &\nabla_x \bar v_{j}  = \frac{1}{\va} \nabla \Gamma_{jk}(y-\xi_k') + \nabla \tilde H^{\va}_{jk}(x, \xi_k)+o(1),
       \end{aligned}
       \right.
   \end{equation}
  where 
  \begin{equation}\label{eq2.26}
     \tilde H^{\va}_{jk}(x, \xi_k) = H_{jk}^\va(x; \xi_k) + \sum_{l \neq k}\Big(\Gamma_{jk}(y - \xi_l) + H^{\va}_{jl}(x; \xi_l) \Big). 
  \end{equation}
Moreover, one has 
  \begin{equation}
     \left \{ 
       \begin{aligned}
          &\Delta_x u_{j} = \frac{1}{\va^2}\Delta U_{jk}(y-\xi_k') + \sum_{l \neq k} \frac{1}{\va^2}\Delta  U_{jl}(y-\xi'_l)+o(1),  \\
          &\Delta_x \bar v_{j}  = \frac{1}{\va^2} \Delta \Gamma_{jk}(y-\xi_k') + \Delta \tilde H^{\va}_{jk}(x, \xi_k)+o(1),
       \end{aligned}
       \right.
   \end{equation}
  \begin{equation}
  \begin{split}
      \nabla u_{j} \cdot \nabla \bar v_{j} =& \frac{1}{\va^2} \nabla U_{jk}(y-\xi_k')\cdot \nabla \Gamma_{jk}(y - \xi_k')  + \frac{1}{\va}\nabla U_{jk}(y-\xi_k') \cdot \nabla \tilde{H}^\va_{jk}(x, \xi_k)  \\
           & + \frac{1}{\va^2}\nabla \Gamma_{jk}(y-\xi_k') \cdot \sum_{l \neq k} \nabla U_{jl}(y-\xi'_l) 
             + \frac{1}{\va}\sum_{l \neq k}\nabla U_{jl}(y-\xi_l') \cdot \nabla \tilde{H}^\va_{jk}(x, \xi_k),
             \end{split}
  \end{equation}
and
\begin{equation}
  \begin{split}
    u_{j} \Delta \bar v_{j} =& \frac{1}{\va^2}U_{jk}(y-\xi_k') \cdot \Delta \Gamma_{jk}(y - \xi_k') 
        + \frac{1}{\va^2}\sum_{l\neq k}U_{jl}(y-\xi'_l) \Delta \Gamma_{jk}(y-\xi_k')  \\
        & + U_{jk}(y-\xi_k')\cdot \Delta \tilde{H}^\va_{jk}(x, \xi_k) 
            + \sum_{l \neq k} U_{jl}(y-\xi_l') \cdot \Delta \tilde{H}^\va_{jk}(x, \xi_k).
   \end{split}
\end{equation}
 Then, one finds for $j=1,2$,
 \begin{equation}\label{20251114eq1}
  \begin{split}
    S_j(\textbf{u}) =& \frac{1}{\va^2}[ \Delta U_{jk}(y-\xi_k') - \nabla U_{jk}(y-\xi_k')\cdot \nabla \Gamma_{jk}(y - \xi'_k) - U_{jk}(y-\xi_k') \cdot \Delta \Gamma_{jk}(y-\xi_k')  ] \\
     & - \Big[\frac{1}{\va}\sum_{l \neq k} \nabla U_{jl}(y-\xi_l') \cdot \nabla \tilde{H}^\va_{jk}(x, \xi_k) +  \sum_{l \neq k} U_{jl}(y-\xi_l') \cdot \Delta \tilde{H}^\va_{jk}(x, \xi_k) \Big]  \\
      & - \Big[\frac{1}{\va}\nabla U_{jk}(y-\xi_k') \cdot \nabla \tilde{H}^\va_{jk}(x, \xi_k) + \frac{1}{\va^2}\nabla \Gamma_{jk}(y - \xi'_k) \cdot \sum_{l \neq k}\nabla U_{jl}(y-\xi'_{l})  \\
      &  - \frac{1}{\va^2}\sum_{l\neq k}U_{jl}(y-\xi'_l) \Delta \Gamma_{j, {\mu}_{jk}}(y-\xi_k') + U_{jk}(y-\xi_k')\cdot \Delta \tilde{H}^{\va}_{jk}(x, \xi_k) \Big]  \\
       & + \sum_{l \neq k} \frac{1}{\va^2}\Delta  U_{jl}(y-\xi'_l) + \lambda_j U_{jk}(y-\xi'_k)\Big(\bar u - U_{jk}(y-\xi_k') \Big) \\ 
         & + \lambda_j\sum_{l \neq k} U_{jl}(y-\xi_l')\Big(\bar u_j -U_{jk}(y-\xi_k') -\sum_{l \neq k} U_{jl}(y-\xi'_l) \Big) \\
          & +\lambda_j U_{jk}(y-\xi_k') \sum_{l \neq k} U_{jl}(y-\xi_l')+o(1),
  \end{split}
  \end{equation}
  where $\textbf{u}=(u_1,u_2)^T.$  Similarly as shown in \cite{kong2022existence}, we observe from (\ref{20251114eq1}) that the main contribution term in the error is
\begin{align}\label{barIj12025117}
\bar I_{j1}:=\frac{1}{\va}\nabla U_{jk}(y-\xi_k') \cdot \nabla \tilde{H}^\va_{jk}(x, \xi_k)+U_{jk}(y-\xi_k')\cdot \Delta \tilde{H}^{\va}_{jk}(x, \xi_k),
\end{align}
and $\bar I_{j1}=O\big(\frac{1}{\va}\big)$. 
Then it follows that the leading error of $\va^2 S_j(\textbf{u})$ is $O(\va)$.  To further study $\bar I_{j1}$ given in (\ref{barIj12025117}), we use a single interior spot as an example to illustrate our idea.  In fact, we expand
\begin{align}\label{nablatildeHjva2025117}
\nabla \tilde H_{j}^{\va}(x;\xi)=\hat c_{j}\nabla H(\xi;\xi)+\hat c_{j}\nabla^2H(\xi;\xi)(x-\xi)+O(\va^{\alpha}),
\end{align}
where \eqref{2162025117} has been used.  Since $x-\xi=\va y,$ we have that $\va\nabla\cdot (U_{j}\nabla H(\xi))$ dominates.  To balance this, we adjust location $\xi$ such that $\nabla H(\xi_0)=0$, which implies the principal term $\xi_0$ of $\xi$ is a critical point of $\mathcal J_m$ defined by (\ref{jm}) with $m=1$.  Similarly, for the multi-spots, we adjust $(\xi_1,\cdots,\xi_m)$ such that the leading term is governed by the critical point of (\ref{jm}).

To eliminate the error generated from the logistic growth, we  
define the second approximation of $(u_j,\bar v_j)$ as
  \begin{equation}\label{secondansatz}
\left\{
\begin{aligned}
u_{j} &=  \sum_{k = 1}^m  \Big(U_{jk}(y-\xi_k') + \va^2 \phi_{jk}(y - \xi'_k) \Big), \\
\bar v_{j}& =\sum_{k=1}^m \Big(-m_j\log\va + \Gamma_{jk} (y-\xi_k')+H^{\va}_{jk}(x;\xi_k) \Big) + \va^2\Big( \psi_{jk}(y - \xi'_k) + \mathcal{H}^{\va}_{jk}(x, \xi_k)\Big),
\end{aligned}
\right.
\end{equation} 
where $(\phi_{jk},\psi_{jk})$, $j=1,2$, $k=1,\cdots,m$ are the next order term and $\mathcal{H}^{\va}_{jk}$ are correction terms of $\psi_{jk}$ satisfying 
\begin{equation}
\left\{
 \begin{aligned}
   & \Delta_x  {\mathcal{H}}^{\va}_{jk}(x, \xi_k) -  {\mathcal{H}}^{\va}_{jk}(x, \xi_k)  = -  \psi_{jk},&x\in\Omega,\\
   & \partial_{\textbf{n}}{\mathcal H}_{jk}^{\va}=-\p_{\textbf{n}} \psi_{jk},&x\in\p \Omega.
\end{aligned}
  \right.
 \end{equation}
For simplicity of notation, we drop ``k" and use $ \psi_j$, $\phi_j$ and $U_{j}$ to replace $ \psi_{jk}$, $\phi_{jk}$ and $U_{jk}(y-\xi_k')$.  To balance the error generated by logistic source, we choose $(\phi_1, \phi_2, \psi_1, \psi_2)$ as a solution to
\begin{equation}\label{nextorderansatz}
  \left\{ 
\begin{aligned}
   &\nabla (U_{1}\nabla g_1)+ \lambda_1 U_{1}(\bar u_1 - U_{1}) = 0,  \\
    &\nabla (U_{2} \nabla g_2) + \lambda_2 U_{2}(\bar u_2 - U_{2}) = 0,  \\
   & g_1= \frac{\phi_1}{U_{1}} -  \psi_1; \ \   g_2 = \frac{\phi_2}{U_{2}} -  \psi_2,  \\
    &\Delta \psi_1 + a_{11} \chi_1\phi_1 + a_{12}\chi_1 \phi_2 = 0,\\
    &\Delta \psi_2 + a_{21} \chi_2\phi_1 + a_{22} \chi_2 \phi_2 = 0.
\end{aligned}
\right.
\end{equation}
Next, we solve (\ref{nextorderansatz}) and first define $h_j:=-\lambda_jU_j(\bar u_j-U_j)$. Then by applying the first integral method, one gets
\begin{align}
g_{j}=\int_{r}^{\infty}\frac{1}{\rho U_{j}(\rho)}\int_0^{\rho}  h_{j}(s)s\,ds\,d\rho,
\end{align}
where we have used
$$\int_{\R^2} h_j\,dy=0.$$
Thus, one has for $j=1,2$, 
\begin{align}
 g_j\sim \langle r\rangle^{2-\delta_j},
\end{align}
where $\delta_j>0$ is small enough.  In light of $\tilde g_j:=U_jg_j\sim \langle r\rangle^{-(m_j-2+\delta_j)}$, we find from the variation-of-parameters formula that there exists $(\psi_1, \psi_2)$ such that 
\begin{align}
 \psi_j=O(\log |y|),\text{ for }|y|\gg 1.
\end{align}
Invoking $\phi_j=U_{j}g_j+U_j \psi_j$, we further obtain that there exists $\phi_j\sim \langle r\rangle^{-(m_j-2+\delta_j)}$, where $m_j>2$ and $\delta_j>0$ is small enough.

By using $(u_{j1},\bar v_{j1})$ defined in (\ref{secondansatz}) as the ansatz, we shall perform the error computation and establish the inner and outer systems satisfied by the remainder term $(\varphi_j,w_j)$.  To this end, we write the solution $(u_j,\bar v_j)$ to (\ref{ss}) as 
\begin{equation}\label{collect1ansatz}
\left\{
  \begin{aligned}
u_j&=\sum_{k =1}^m  \bigg[U_{jk}\bigg(\frac{x-\xi_k}{\varepsilon}\bigg)+ \varepsilon^2\phi_{jk}\bigg(\frac{x-\xi_k}{\varepsilon}\bigg)\bigg] + \varphi_j\bigg(\frac{x}{\va}\bigg), 
 \\
  \bar v_j &=\sum_{k =1}^m \bigg[ \bigg(-m_j\log\varepsilon +  \Gamma_{jk}\bigg(\frac{x - \xi_k}{\varepsilon}\bigg)+ H^{\va}_{jk}(x;\xi_{k}) \bigg)+ \varepsilon^2\bigg(\psi_{jk}\bigg(\frac{x - \xi_k}{\varepsilon}\bigg) +  \mathcal{H}^{\va}_{jk}(x, \xi_k)\bigg) \bigg]+  w_j\bigg(\frac{x}{\va}\bigg),
\end{aligned}
\right.
\end{equation}
 where and in the sequel we rewrite $\xi_k^{\va}$ as $\xi_k$ for the simplicity of notations.  Then, we compute the error term to get 
 
 \begin{equation}
   \nabla u_j = \sum_{k =1}^m \Big[ \frac{1}{\varepsilon}\nabla U_{jk}(y - \xi'_k) + \varepsilon \nabla \phi_{jk}(y - \xi'_k)\Big] + \frac{1}{\varepsilon}\nabla_y \varphi_j(y),
 \end{equation}
 \begin{equation}
     \nabla \bar v_j = \sum_{k =1}^m \Big[\frac{1}{\varepsilon} \nabla \Gamma_{jk}(y - \xi'_k) +  \nabla H^{\varepsilon}_{jk}(\varepsilon y, \xi'_k) + \varepsilon \nabla \psi_{jk}(y - \xi'_k) + \varepsilon^2\nabla \mathcal{H}^{\varepsilon}_{jk}(\varepsilon y, \xi'_k) \Big] + \frac{1}{\varepsilon}\nabla_y w_j(y),
 \end{equation}

  \begin{equation}
    \Delta u_j = \sum_{k =1}^m \Big[\frac{1}{\varepsilon^2}\Delta U_{jk}(y - \xi'_k) + \Delta \phi_{jk}(y - \xi'_k) \Big] + \frac{1}{\varepsilon^2}\Delta_y \varphi_j(y),
  \end{equation}
 and
  \begin{equation}\label{collect2ansatz}
    \Delta \bar v_j = \sum_{k =1}^m \Big[\frac{1}{\varepsilon^2} \Delta \Gamma_{jk}(y - \xi'_k) + \Delta H^{\varepsilon}_{jk}(\varepsilon y, \xi'_k) + \Delta \psi_{jk}(y - \xi'_k) + \varepsilon^2 \Delta \mathcal{H}^{\varepsilon}_{jk}(\varepsilon y, \xi'_k) \Big] + \frac{1}{\varepsilon^2}\Delta_y w_j(y). 
  \end{equation}
Upon substituting (\ref{collect1ansatz})--(\ref{collect2ansatz}) into (\ref{224expression2024}), one finds
  \begin{equation}
  \begin{split}
      0  = \Delta u_j - \nabla u_j \cdot \nabla \bar v_j - u_j \cdot \Delta \bar v_j + \lambda_j u_j (\bar u_j  - u_j) = L_j[\varphi_1, \varphi_2] + \sum_{l=1}^7 I_{jl},
     \end{split}
  \end{equation}
where 
   \begin{equation}\label{2442025116}
    L_j[\varphi_1, \varphi_2] = -\Delta \varphi_j + \nabla \cdot(P_j \nabla w_j) + \nabla \cdot (\varphi_j \nabla Q_j),
   \end{equation}

     \begin{equation}
   P_j = \sum_{k=1}^m U_{jk}(y - \xi'_k) \ \ \text{and} \ \  Q_j = \sum_{k= 1}^m \Gamma_{jk}(y - \xi'_k) 
     \end{equation}
and $I_{jl}$, $l=1,\cdots,7$ are defined as
  
\begin{equation}
   I_{j1} = - \frac{1}{\varepsilon^2} \sum_{k =1}^m \sum_{l \neq k}  U_{jk}(y - \xi'_k)\Delta \Gamma_{jl}(y - \xi'_l),
\end{equation}

\begin{equation} 
\begin{split}
 I_{j2}  =& -\frac{1}{\varepsilon^2} \sum_{k = 1}^m \sum_{l \neq k} \nabla U_{jk}(y - \xi'_k) \cdot \nabla \Big(\Gamma_{jl}(y - \xi'_l) +\va
       H^{\varepsilon}_{jk}(\varepsilon y,\xi_k) \Big)\\
     &  - \frac{1}{\varepsilon}\sum_{k =1}^m   \nabla U_{jk}(y - \xi'_k) \cdot \nabla H^{\varepsilon}_{jk}(\varepsilon y, \xi_k) + \sum_{k =1}^m \sum_{l=1}^m U_{jk}(y - \xi'_k) \cdot \Delta H^{\varepsilon}_{jl}(\varepsilon y, \xi_l),
 \end{split}
\end{equation}

\begin{equation}
\begin{split}
   I_{j3} = - \sum_{k =1}^m \sum_{l \neq k} & \Big( \nabla U_{jk}(y - \xi'_k) \cdot \nabla \psi_{jl}(y - \xi'_l)   + \nabla \phi_{jk}(y - \xi'_k) \cdot \nabla \Gamma_{jl}(y - \xi'_l)  \\
      & + U_{jk}(y - \xi'_k) \cdot\Delta \psi_{jl}(y - \xi'_l) + \phi_{jk}(y - \xi'_k) \cdot \Delta \Gamma_{jl}(y - \xi'_l)\Big),
   \end{split}
\end{equation}

\begin{equation}
\begin{split}
   I_{j4} =- \sum_{k =1}^m \sum_{l =1}^m &\Big(\varepsilon \nabla U_{jk}(y - \xi'_k) \cdot \nabla \mathcal{H}^{\va}_{jl}(\varepsilon y, \xi_k) +  
      \varepsilon \nabla \phi_{jk}(y - \xi'_k) \cdot \nabla H^{\varepsilon}_{jl}(\varepsilon y, \xi_l)   \\
      & + \varepsilon^2  U_{jk}(y - \xi'_k) \cdot \Delta \mathcal{H}^{\va}_{jl}(\varepsilon y, \xi_l) + \varepsilon^2  \phi_{jk}(y - \xi'_k) \cdot \Delta H^{\varepsilon}_{jl}(\varepsilon y, \xi_l)\Big),
       \end{split}
\end{equation}

\begin{equation}
\begin{split}
   I_{j5} = \sum_{k =1}^m \sum_{l =1}^m &\Big(\varepsilon^2 \nabla \phi_{jk}(y - \xi'_k) \cdot (\nabla  \psi_{jl}(y - \xi'_l) + 
   \va \nabla \mathcal{H}^{\va}_{jl}(\varepsilon y, \xi_l))    \\
      & + \varepsilon^2 \phi_{jk}( y - \xi'_k) \cdot \big(\Delta \psi_{jl}(y - \xi'_l) + \varepsilon^2  \Delta \mathcal{H}^{\varepsilon}_{jl}(\varepsilon y, \xi_l)\big)\Big), 
       \end{split}
\end{equation}

\begin{equation}
\begin{split}
  I_{j6} = &- \frac{1}{\va}\nabla \varphi_j \cdot \sum_{k =1}^m\nabla H^{\va}_{jk}(\va y, \xi_k) - \varphi_j \sum_{k =1}^m \Delta H^{\va}_{jk}(\varepsilon y, \xi_k)   \\
     & -\frac{1}{\varepsilon}\nabla \varphi_j \cdot \Big(\varepsilon \sum_{k =1}^m \nabla \psi_{jk}(y - \xi'_k) + \varepsilon^2 \nabla \mathcal{H}^{\varepsilon}_{jk}(\varepsilon y, \xi_k) \Big) \\
      & - \varphi_j \Big(\sum_{k =1}^m \Delta  \psi_{jk}(y - \xi'_k) + \varepsilon^2 \Delta \mathcal{H}^{\varepsilon}_{jk}(\varepsilon y, \xi_k) \Big) \\
      & - \nabla w_j \cdot \sum_{k=1}^m \nabla \phi_{jk}(y - \xi'_k)  - \Delta w_j \cdot \sum_{k =1}^m\phi_{jk} (y - \xi'_k)  \\
      & - \frac{1}{\va^2}(\nabla \varphi_j \cdot \nabla w_j + \varphi_j \Delta w_j)
  \end{split}
\end{equation}
 and
\begin{equation}
\begin{split}
    I_{j7}&= \lambda_j \bigg( \varepsilon^2\sum_{k =1}^m\phi_{jk}\bigg(\frac{x-\xi_k}{\varepsilon}\bigg) + \varphi_j\bigg)\bigg(\bar u_j-\sum_{k =1}^m\bigg[U_{jk}\bigg(\frac{x-\xi_k}{\varepsilon}\bigg)+\varepsilon^2\phi_{jk}\bigg(\frac{x-\xi_k}{\varepsilon}\bigg)\bigg] - \varphi_j\bigg)\\
    &- \lambda_j \sum_{k =1}^m U_{jk}\bigg(\frac{x-\xi_k}{\varepsilon}\bigg) \bigg(\sum_{k =1}^m\varepsilon^2\phi_{jk}\bigg(\frac{x-\xi_k}{\varepsilon}\bigg)+ \varphi_j\bigg)+\lambda_j\sum_{k=1}^m U_{jk}\bigg(\frac{x-\xi_k}{\va}\bigg)  \sum_{l \not= k}^m\bigg(\bar u_j-U_{jl}\bigg(\frac{x-\xi_l}{\varepsilon}\bigg)\bigg).
    \end{split}
    \end{equation}
We summarize the computation above to obtain $ \boldsymbol{\varphi}:=(\varphi_1,\varphi_2)$ solves
  \begin{equation}\label{inn-out-before}
   \left\{
  \begin{aligned}
      &L_j[\boldsymbol{\varphi}]= \va^2\sum_{l =1}^7 I_{jl}(\boldsymbol{\varphi}, \mathbf{P}), &&\text{~in~}\Omega_{\va},  \\
       & 0 = \Delta w_1 + a_{11} \chi_1 \varphi_1 + a_{12}\chi_1\varphi_2,   &&\text{~in~}\Omega_{\va}, \\ 
        & 0 = \Delta w_2 + a_{21} \chi_2\varphi_1 + a_{22}\chi_2\varphi_2,   &&\text{~in~}\Omega_{\va}, \\
         &\frac{\partial w_1}{\partial \boldsymbol{\nu}}=\frac{\partial w_2}{\partial \boldsymbol{\nu}} = \frac{\partial \varphi_1}{\partial \boldsymbol{\nu}}=\frac{\partial \varphi_2}{\partial \boldsymbol{\nu}},  &&\text{~on~}\partial\Omega_{\va},
         \end{aligned}
         \right.
        \end{equation}
where $j=1,2$ and
  \begin{equation}
   \mathbf{P} = (c_{11}, \cdots, c_{1m}, c_{21}, \cdots, c_{2m}, \xi_1, \cdots, \xi_m).
  \end{equation}
  
The subsequent sections are devoted to the solvability of (\ref{inn-out-before}) and the existence of solution $(\varphi_1,\varphi_2,w_1,w_2)$.  To this end, we decompose the domain $\Omega$ as inner and outer regions.  Correspondingly, the solution $(\varphi_1,\varphi_2)$ is decomposed as the combination of inner and outer solutions.  Firstly, we establish the linear theory in the inner region, which is shown in Section \ref{sect3}.


\section{Inner Linear Theory}\label{sect3} 
In this section, we consider the inner region $| x-\xi_k|<\delta$ with constant $\delta>0$, where $\xi_k$ denotes the location of the $k$-th spot.  Noting that $\va$ is small enough, one has the inner region $\{y\in \mathbb R^2:\vert y\vert<\frac{\delta}{\va}\}$ approximates the whole space $\mathbb R^2$.  Then, we define the stretched variable $y=\frac{x-\xi_k}{\va}$ and the $k$-th inner operator $L^{\text{inn}}_{k}[\varphi_1, \varphi_2] := (L^{\text{inn}}_{1k}[\varphi_1], L^{\text{inn}}_{2k}[\varphi_2])^T$ as
 \begin{equation}\label{innerlphi}
 L^{\text{inn}}_{jk}[\varphi_j] := -\Delta_y \varphi_j + \nabla \cdot(U_{j}\nabla_y w_j) + \nabla \cdot (\varphi_j\nabla_y \log U_{j}),~y\in \mathbb R^2,
 \end{equation}
where $w_j = (- \Delta_y)^{-1}(b_{j1}\varphi_1+b_{j2}\varphi_2)$ and $j=1,2.$  Here and in the subsequent analysis, we drop ``$k$" and $c_{jk}$ in $U_{jk}$ given by (\ref{generate1}) since in each inner region, the form of $U_{jk}$ is the same and $c_{jk}$ is a constant.  We remark that $c_{jk}\varphi_j$ here is equal to the original one, where $c_{jk}$ is incorporated in (\ref{innerlphi}).  Similarly as shown in \cite{kong2022existence}, when the location of the spot is in the interior of domain $\Omega,$ the inner problem is formulated as 
  \begin{equation}\label{eq2.1}
    L^{\text{inn}}_{jk}[\varphi_j] + h_j=0 \ \ \text{in} \ \ \R^2,
\end{equation}
where $\textbf{h}:=(h_1,h_2)^T$ denotes the error.  We introduce the intermediate variables $g_1$ and $g_2$ to simplify the inner problem (\ref{eq2.1}) as 
\begin{equation}\label{innerproblem}
\left\{
\begin{aligned}
&\nabla_y\cdot(U_{1}\nabla g_1)=h_1, \ \ \ \  g_1=\frac{\varphi_1}{U_{1}}- w_1,     &&y\in\mathbb R^2,\\
&\nabla_y\cdot(U_{2}\nabla g_2)=h_2, \ \ \ \  g_2=\frac{\varphi_2}{U_{2}}- w_2,     &&y\in\mathbb R^2,\\
&\Delta_y w_1+b_{11}U_{1} w_1+b_{11}U_{1}g_1+b_{12}U_{2}w_2+b_{12}U_{2}g_2=0,       &&y\in\mathbb R^2\\
&\Delta_y w_2+b_{21}U_{1} w_1+b_{21}U_{1}g_1+b_{22}U_{2}w_2+b_{22}U_{2}g_2=0,       &&y\in\mathbb R^2,
\end{aligned}
\right.
\end{equation}
where $b_{ij}$, $i=1,2$, $j=1,2$ are given in (\ref{bvaluenew2024104}).  System (\ref{innerproblem}) can be regarded as the coupling of divergence form equations and the linearized Liouville systems.  We first consider the non-degeneracy of operator $\nabla_y\cdot(U_{j}\nabla g_j)$ to $U_j$ and obtain 
\begin{Lemma}\label{nondegeneracy1}
The bounded solution space to the following problem
\begin{equation}\label{limitingoperator}
\left\{
\begin{aligned}
 &\nabla\cdot\Big(U_{j}\nabla \Big(\frac{\tilde g_j}{U_{j}}\Big)\Big)=0, \ \  &y\in \mathbb R^2,\\
&\tilde g_j\in H^2_0(\mathbb R^2), \ \   &\vert \tilde g_j\vert=O(1)(1+r)^{-\sigma-2}
\end{aligned}
\right.
\end{equation}
is one-dimensional and spanned by the nontrivial kernel $U_{j}$, where ${\tilde g_{j}}=U_{j}g_j$, $j =1, 2$ and $\sigma>0$ is a small constant.
\end{Lemma}

\begin{proof}
Thanks to the definition of $\tilde g_j$ shown in (\ref{innerproblem}), we have $ g_j$ satisfies
\begin{align}\label{211}
\nabla\cdot(U_{j}\nabla g_j)=0,~~~y\in \mathbb R^2.
\end{align}
Denote $g_{+}$ and $\eta$ as
\begin{equation*}
 g_{j+}=
   \left\{
\begin{aligned}
& g_j,  &&g_j>0,\\
& 0,   &&g_j\leq 0,
\end{aligned}
\right.
\text{~~~and~~~}
\eta=
  \left\{
\begin{aligned}
&1,    &&y\in B_{R}(0),    \\
&0,    &&y\in \mathbb R^2\backslash B_{2 R}(0),
\end{aligned}
\right.
\end{equation*}
where $R>0$ is a constant.  Upon multiplying (\ref{211}) by $ g_{j+}^{N}\eta^2$ with $N$ being determined later on, we integrate it over $\mathbb R^2$ to get
\begin{align*}
0=&\int_{\mathbb R^2}\nabla\cdot(U_{j}\nabla g_j)g_{j+}^{N}\eta^2dy=-\int_{\mathbb R^2}(U_{j}\nabla g_j)\cdot\nabla(g_{j+}^{N}\eta^2)dy,
\end{align*}
which implies
\begin{align*}
0=\int_{\mathbb R^2}U_{j}\nabla g_j\cdot \nabla (g_{j+}^N)\eta^2dy  + \int_{\mathbb R^2}(U_{j}\nabla g_j)\cdot g_{j+}^{N}\nabla(\eta^2)dy.
\end{align*}
Noting that the support of $\eta$ is $B_{2R}(0)$, one further has
\begin{align*}
\int_{B_{2R}(0)}U_{j}\nabla g_j \cdot \nabla (g_{j+}^N)\eta^2dy=-\int_{\mathbb R^2}(U_{j}\nabla g_j)\cdot g_{j+}^{N}\nabla(\eta^2)dy.
\end{align*}
Then we utilize the integration by parts to get
\begin{align*}
\frac{4N}{(N+1)^2}\int_{B_{2R}(0)}U_{j}\big\vert\nabla g_{j+}^{\frac{N+1}{2}}\big\vert^2dy=\frac{1}{N+1}\int_{B_{2R}(0)}g_{j+}^{N+1}\nabla\cdot(U_{j}\nabla\eta^2)dy.
\end{align*}
Therefore,
\begin{align}\label{212}
\int_{B_{2R}(0)}U_{j}\big\vert\nabla g_{j+}^{\frac{N+1}{2}}\big\vert^2dy\leq C\int_{B_{2R}(0)}\vert g_{j+}\vert^{N+1}\cdot \vert\nabla\cdot(U_{j}\nabla\eta^2)\vert dy,
\end{align}
where $C>0$ is some large constant. Since $g_j$ satisfies $\vert g_{j}\vert\leq C_1 \frac{1}{U_{j}}(1+r)^{-\sigma-2}$ for small $\sigma>0$ and some constant $C_1>0$, one has from (\ref{212}) that 
\begin{align}\label{referee21}
\int_{B_{2R}(0)}U_{j}\big\vert\nabla_{y} g_{j+}^{\frac{N+1}{2}}\big\vert^2dy\leq C_2(2R)^2\Big(\frac{1}{U^{N}_{j}}\big\vert_{\vert y\vert=2R}\Big)\frac{1}{(1+2R)^{(N+1)\sigma+2N+2}},
\end{align}
where $C_2>0$ is a constant. 
 Then we obtain for some positive constant $C_3$,
\begin{align*}
4R^2\Big(\frac{1}{U^{N}_{j}}\big\vert_{\vert y\vert=2R}\Big)\leq C_3(1+2R)^{m_jN+2}.
\end{align*}
By choosing $N=\frac{\sigma}{2(m_j-\sigma-2)}$, one uses \eqref{referee21} to get
\begin{align*}
\int_{B_{2R}(0)}U_{j}\big\vert\nabla_{y} \big(g_{j+}^{\frac{N+1}{2}}\big)\vert^2dy\leq C_2C_3(1+2R)^{m_jN+2}\frac{1}{(1+2R)^{(N+1)\sigma+2N+2}}\rightarrow 0\text{~as~}R\rightarrow\infty,
\end{align*}
where $\sigma$ can be chosen small enough.  Thus, we obtain $g_{j+}\equiv C_4$ for $y_m\in\mathbb R^2$, where $C_4>0$ is a constant. Proceeding with the similar argument for $g^N_{j-}\eta^2$, we obtain $g_{j-}\equiv C_5$ for some constant $C_5>0$.  This completes the proof of the lemma.
\end{proof}

Next, we analyze the non-degeneracy of the linearized operators of Liouville systems.  Define
\begin{align*}
\mathcal L[w_1, w_2]=\bigg(\Delta w_1+\sum_{j=1, 2}b_{1j}U_{j} w_j,\Delta w_1+\sum_{j=1, 2}b_{2j}U_{j} w_j\bigg)^T.
\end{align*}
We have its adjoint operator is 
\begin{align*}
\mathcal L^*[w_1,w_2]=\bigg(\Delta w_1+U_{1}\sum_{j=1, 2}b_{1j}w_j,\Delta w_1+U_{2}\sum_{j=1, 2}b_{2j}w_j\bigg)^T.
\end{align*}
For convenience, we denote the following transform 
      \begin{equation}\label{2024114logtransUup}
         (\log U^1, \log U^2) = \Big(\sum_{j=1, 2}b^{1j}\log U_j, \sum_{j=1, 2}b^{2j}\log U_j \Big),
      \end{equation}      
where $(b^{ij})_{2 \times 2}$ is the inverse matrix of $ (b_{ij})_{2 \times 2}$.   For the kernel of $\mathcal L(w_1,w_2)$, we have the following lemma.
\begin{Lemma}
Assume that $(w_1,w_2)^T$ satisfies
\begin{equation}
\left\{
\begin{aligned}
&\mathcal L[w_1,w_2]=0    \ \  \text{ in }  \ \    \mathbb R^2,\\
&|w_j(y)|\leq C(1+|y|^{\tau}),  \ \  \text{for some }\tau\in[0,1)\text{ with }j=1,2,
\end{aligned}
\right.
\end{equation}
 then we have  
\begin{align}
(w_1, w_2)^T \in \text{span}\Big\{(\partial_1 \log U_{1},\partial_1\log U_{2})^T,(\partial_2 \log U_{1},\partial_2\log U_{2})^T,(\partial_3 \log U_{1},\partial_3\log U_{2})^T \Big\},
\end{align}
where $\partial_1\log U_{j}=\partial_{x_1}\log U_{j}$, $\partial_{2}\log U_{j}=\partial_{x_2}\log U_{j}$ and $\partial_3\log U_{j}=r(\log U_{j})'(r)+{2}$ for $j=1,2.$
\end{Lemma}
\begin{proof}
See the proof of Theorem 2.1 in \cite{lin2013liouville1} and Lemma 3.1 in \cite{lin2010profile}.
\end{proof}
  
In addition, for the kernel of adjoint operator $\mathcal L^*,$ we have the following result.
\begin{Lemma}\label{lemm13adjoint}
Assume $(w_1, w_2)^T$ is a solution to
\begin{equation}
\left\{
\begin{aligned}
&\mathcal L^*[w_1, w_2]=0 \ \ \text{ in } \ \ \mathbb R^2,\\
&|w_j(z)|\leq C(1+|z|^{\tau}),\ \  \text{ with } \ \ j=1,2,
\end{aligned}
\right.
\end{equation}
for some $\tau \in [0, 1)$,  then we have
\begin{align}
(w_1, w_2)^T\in \text{span}\Big\{(\partial_1 \log U^{1},\partial_1\log U^{2})^T,(\partial_2 \log U^{1},\partial_2\log U^{2})^T,(\partial_3 \log U^{1},\partial_3\log U^{2})^T \Big\},
\end{align}
where $U^1$ and $U^2$ are defined in (\ref{2024114logtransUup}). 
\end{Lemma}
\begin{proof}
See the proof of Corollary 5.2 in \cite{huang2019existence}.
\end{proof}
With the aid of Lemma \ref{nondegeneracy1} and Lemma \ref{lemm13adjoint}, one finds
\begin{Lemma}\label{lemma31}
  Suppose that $h_j$ satisfy
  \begin{equation}\label{innercon}
      \int_{\R^2}h_j(y)dy = 0,~\int_{\R^2}h_j(y)y_ldy=0 \ \   \text{for} \ \ j =1,2,\, l=1, 2,
  \end{equation}
  then we have for any $\Vert h_j\Vert_{4+\sigma}<\infty$ with $\sigma\in(0,1)$, there exist a solution $\boldsymbol{\varphi}:=(\varphi_1,\varphi_2)^T=\boldsymbol{\mathcal{T}_{in}}[h_1,h_2]$ to \eqref{innerproblem} such that
    \begin{equation}\label{innerprori}
      \|\varphi_j\|_{2 + \sigma} \le C_j \|{h}_j\|_{4 + \sigma}, 
    \end{equation}
    where $\boldsymbol{\mathcal{T}_{in}}[h_1,h_2]$ is a continuous linear operator from the Banach space $\mathcal C^*\times \mathcal C^*$ of all functions $(h_1,h_2)^T$ in $L^\infty\times L^\infty$ for which $\Vert h_1\Vert_{4+\sigma}+\Vert h_2\Vert_{4+\sigma}<+\infty$ into $L^\infty\times L^\infty$.
\end{Lemma}

\begin{proof}
Similarly as shown in \cite{kong2022existence}, we perform Fourier projection and obtain the $k$-th mode of (\ref{innerproblem}) as follows
\begin{equation}\label{fourierexpansion}
\left\{
\begin{aligned}
&\frac{d^2\tilde g_{1k}}{dr^2}+\frac{1}{r}\frac{d\tilde g_{1k}}{dr}-\frac{k^2}{r^2}\tilde g_{1k}+(\log U_{1})_r\frac{d\tilde g_{1k}}{dr}+U_{1}\tilde g_{1k} = h_{1k},\\
 &\frac{d^2\tilde g_{2k}}{dr^2}+\frac{1}{r}\frac{d\tilde g_{2k}}{dr}-\frac{k^2}{r^2}\tilde g_{2k}+(\log U_2)_r\frac{d\tilde g_{2k}}{dr}+U_2\tilde g_{2k}= h_{2k}, \\
&\frac{d^2 w_{1k}}{dr^2}+\frac{1}{r}\frac{d w_{1k}}{dr}-\frac{k^2}{r^2} w_{1k}+b_{11}U_{1} w_{1k}+b_{11}\tilde g_{1k}+b_{12}U_{2} w_{2k}+b_{12}\tilde g_{2k}=0,\\
&\frac{d^2 w_{2k}}{dr^2}+\frac{1}{r}\frac{d w_{2k}}{dr}-\frac{k^2}{r^2} w_{2k}+b_{11}U_{1} w_{2k}+b_{11}\tilde g_{1k}+b_{12}U_{2} w_{2k}+b_{12}\tilde g_{2k}=0,
\end{aligned}
\right.
\end{equation}
where we define $U_j g_j=\tilde g_j$ and
\begin{align}
\tilde g_j=\sum_{k=0}^\infty \tilde g_{jk}(r)e^{ik\theta}, \ \  h_j=\sum_{k=0}^\infty h_j(r)e^{ik\theta},  \ \  w_j=\sum_{k=0}^\infty w_{jk}(r)e^{ik\theta}, \ \   j=1,2.
\end{align}
First of all, we consider the $0$-th mode in (\ref{fourierexpansion}), which is 
\begin{equation}\label{originalmode0}
\left\{
\begin{aligned}
&\frac{d^2\tilde g_{10}}{dr^2}+\frac{1}{r}\frac{d\tilde g_{10}}{dr} +(\log U_{1})_r\frac{d\tilde g_{10}}{dr}+U_{1}\tilde g_{10}=h_{10},\\
&\frac{d^2\tilde g_{20}}{dr^2}+\frac{1}{r}\frac{d\tilde g_{20}}{dr} +(\log U_2)_r\frac{d\tilde g_{20}}{dr}+U_{2}\tilde g_{20}=h_{20},\\
&\frac{d^2 w_{10}}{dr^2}+\frac{1}{r}\frac{d w_{10}}{dr}+b_{11}U_{1} w_{10}+b_{12}U_{2} w_{20}=f_{10},\\
&\frac{d^2 w_{20}}{dr^2}+\frac{1}{r}\frac{d w_{20}}{dr}+b_{21}U_{1} w_{10}+b_{22}U_{2} w_{20}=f_{20},
\end{aligned}
\right.
\end{equation}
where 
\begin{equation}
f_{j0}:=-b_{j1}\tilde g_{10}-b_{j2}\tilde g_{20},~j=1,2.
\end{equation}
We choose the solution to the $\tilde g_j$-equation in (\ref{originalmode0}) as
\begin{align}
\tilde g_j =U_{j}g_{j0},~~g_{j0}=\int_{r}^{\infty}\frac{1}{\rho U_{j}(\rho)}\int_0^{\rho}h_{j0}(s)s\,ds\,d\rho.
\end{align}
By using the mass condition in  (\ref{innercon}), one further has
\begin{align*}
0=\int_{\mathbb R^2}h_j\,dy=\int_0^{2\pi}\int_0^{\infty}
    \sum_{k=0}^\infty h_{jk}e^{ik\theta}r\,dr\,d\theta=2\pi \int_0^{\infty} h_{j0}(r)r\,dr.
\end{align*}
Then it follows 
\begin{equation}\label{decaygj0tildegjo}
g_{j0}\sim \langle r\rangle^{m_j-2-\sigma},~\tilde g_{j0}\sim \langle r\rangle^{-2-\sigma},\text{~for}~\sigma>0\text{ small enough}.
\end{equation}
  Next, we shall solve $(w_{10},w_{20})$ in (\ref{originalmode0}) via the variation-of-parameter method.  To begin with, we focus on the following homogeneous problem 
 \begin{equation}\label{homomode0}
\left\{
\begin{aligned}
\frac{d^2 w_{10}}{dr^2}+\frac{1}{r}\frac{d w_{10}}{dr}+b_{11}U_{1} w_{10}+ b_{12}U_{2} w_{20}=0,\\
\frac{d^2 w_{20}}{dr^2}+\frac{1}{r}\frac{d w_{20}}{dr}+b_{11}U_{1} w_{10}+ b_{12}U_{2} w_{20}=0.
\end{aligned}
\right.
\end{equation}
By using Lemma 2.1 of \cite{lin2010profile}, we have there exist two linearly independent solution pairs $\textbf{Z}_j=(Z_{j1},Z_{j2})^T$ with $j=1,2$ of (\ref{homomode0}), which satisfy
\begin{align}\label{fundamental2025115}
Z_{j1}=O(\log(1+r)),\ \ \ \   Z_{j2}=O(\log(1+r)).
\end{align}
We further rewrite the equation (\ref{homomode0}) as  
\begin{equation}\label{mathcalL0eqf0}
\mathcal L_0[\boldsymbol{W}_0]=\boldsymbol{f}_0,
\end{equation}
where $\boldsymbol{W}_0=(w_{10}, w_{20})^T$, $\boldsymbol{f}_0=(f_{10},f_{20})^T $,
\begin{align}\label{diagfirst}
\mathcal L_0=\text{diag}{(\Delta_r,\Delta_r)}+\boldsymbol{ A},~~\mathcal L^*_0=\text{diag}(\Delta_r,\Delta_r)+\boldsymbol{A^T},
\end{align}
and
\begin{align}\label{boldsymbolA2025}
\boldsymbol{A}=\left( \begin{array}{cc}
b_{11}U_{1} & b_{12}U_{2}  \\
  b_{21}U_{1} & b_{22}U_{2} \\
    \end{array}
  \right).
  \end{align}
Next, we are concerned with the following homogeneous adjoint problem
\begin{align}\label{adjointbecomes}
  \mathcal L^{*}_0 [\textbf{Z}_j^*]=0,
  \end{align}
 where $\mathcal L_0^*$ is defined in \eqref{diagfirst}. We claim that (\ref{adjointbecomes}) admits two linear independent solution pairs $\textbf{Z}^*_j=(Z^*_{j1},Z^*_{j2})$ with $j=1,2$ and they satisfy 
 \begin{align}\label{Zhat1220251}
 Z^*_{j1}=O(\log(1+r)) \text{~and~} Z^*_{j2}=O(\log(1+r)).
 \end{align}
  To show this, we rewrite (\ref{adjointbecomes}) as   
  \begin{equation}\label{adjointproblem}
  \left\{
  \begin{aligned}
  \frac{d^2 w_{1}}{dr^2}+\frac{1}{r}\frac{d w_1}{dr}+b_{11}U_{1} w_{1}+b_{21}U_{1} w_{2}=0,\\
  \frac{d^2 w_2}{dr^2}+\frac{1}{r}\frac{d w_2}{dr}+b_{12}U_{2} w_{1}+b_{22}U_{2} w_{2}=0.
  \end{aligned}
  \right.
  \end{equation}
 By using the transform $\hat w_{j}(r)= w_j(e^r)$ with $j=1,2$, one finds
\begin{align}
\frac{d^2\hat w_j}{dr^2}+\bigg(\sum_{k=1, 2}a_{kj}\hat w_{k}\bigg)U_{j}(e^r)e^{2r}=0.
\end{align}
Noting that $U_j(r)\sim \langle r\rangle^{-m_j}$ with $m_j>2$, we follow the similar argument shown in Lemma 2.1 of \cite{lin2010profile} and obtain that there exists a solution $(w_1,w_2)$ satisfying
$$w_j = O(\log(1+r)),~j=1,2,$$
which proves our claim.  
Now, we are ready to solve the inhomogeneous problem (\ref{mathcalL0eqf0}).  In fact, by applying the integration by parts, one gets 
\begin{align}
&\int_{B_r(0)}\mathcal L_0 (\boldsymbol{W}_0)\cdot \textbf{Z}^*_j\,dx-\int_{B_r(0)}\mathcal L^*_0 (\textbf{Z}^*_j)\cdot \boldsymbol{W}_0\,dx\nonumber\\
=&\int_{\partial B_r(0)}\frac{\partial\boldsymbol{W}_0}{\partial\boldsymbol{\nu}}\cdot\textbf{Z}^*_j\,dS-\int_{\partial B_r(0)}\frac{\partial \textbf{Z}_j^*}{\partial\boldsymbol{\nu}}\cdot \boldsymbol{W}_0\,dS,  \ \  j = 1, 2,
\end{align}
where $\textbf{Z}^*_j:=(Z^*_{j1},Z^*_{j2})^T$ are given in \eqref{Zhat1220251}.  It follows that $ \boldsymbol{W}_0$ satisfies the following first order ODEs  
\begin{equation}\label{vectorequation}
\boldsymbol{W}'_0=H(r)\boldsymbol{W}_0+\hat {\boldsymbol{f}_0},
\end{equation}
where 
\begin{equation*}
H:=\left( \begin{array}{cc}
  Z_{11}^* & Z_{12}^*  \\
  Z_{21}^* & Z_{22}^* \\
    \end{array}
  \right)^{-1}\left( \begin{array}{cc}
  \frac{d Z_{11}^*}{dr} & \frac{dZ_{12}^*}{dr}  \\
 \frac{d Z_{21}^*}{dr} & \frac{dZ_{22}^*}{dr} \\
    \end{array}
  \right),\text{~~}\hat {\boldsymbol{f}_0}:=\frac{1}{2\pi r}\left( \begin{array}{cc}
Z_{11}^* & Z_{12}^*  \\
  Z_{21}^* & Z_{22}^* \\
    \end{array}
  \right)^{-1}\left( \begin{array}{c}
\int_{B_r}(f_{10}Z_{11}^*+f_{20}Z_{12}^*)\,dx  \\
\int_{B_r}(f_{10}Z_{21}^*+f_{20}Z_{22}^*)\,dx  \\
    \end{array}
  \right).
  \end{equation*}
In light of (\ref{decaygj0tildegjo}), one finds
  \begin{align}
    \int_{B_r(0)}\big(f_{10}Z_{j1}^*+f_{20}Z_{j2}^* \big)\,dx\sim \langle r\rangle^{-\sigma},
  \end{align}
  where $\sigma>0$ is small enough.  Thus, we have 
  $$\hat {\boldsymbol{f}}_0\sim o\bigg(\frac{1}{r}\bigg).$$
Therefore, we choose a solution to (\ref{vectorequation}) as
$$\boldsymbol{W}_0 = t_1\textbf{Z}_1 + t_2\textbf{Z}_2,$$
  where $t_j=O(\log(1+r))$, $\textbf{Z}_1=(Z_{11},Z_{12})^T$ and $\textbf{Z}_2=(Z_{21},Z_{22})^T$ given by (\ref{fundamental2025115}) are the fundamental solutions of \eqref{vectorequation}.  Moreover, we have the mode 0 of the $\varphi_j-$component in (\ref{innerproblem}) exists and satisfies
$${\varphi}_{j0}=U_{j}{w}_{j0}+\tilde g_{j0}\sim \langle r\rangle^{-2-\sigma},~j=1,2,$$
where $\sigma>0$ is small enough.  

We next focus on the mode $k\geq 1$ in (\ref{fourierexpansion}).  For the $\tilde g_{j}$-equation in (\ref{fourierexpansion}), we define 
\begin{align}\label{ljkfirsttwoeqs}
\mathcal L_{jk}[z]:=-\frac{d^2 z}{dr^2}-\frac{1}{r}\frac{d z}{dr}+\frac{k^2}{r^2}z - (\log U_{j})_r\frac{d z}{dr}-U_{j}z,
\end{align}
and construct the barrier function
\begin{align}\label{wjkbarrier}
z_{jk}:=C_{jk}\Vert h_j\Vert_{4+\sigma}{(1+r)^{-2-\sigma}},
\end{align}
where $\sigma$ is small enough.  Then, we compute to get
\begin{equation}
\begin{aligned}
\mathcal L_{jk}[ z_{jk}]+h_{jk}
=&\frac{C_{jk}\Vert h_j\Vert_{4+\sigma}}{(1+r)^{4+\sigma}}\bigg[-(2+\sigma)(3+\sigma)+\frac{1+r}{r}(2+\sigma)+\frac{k^2}{r^2}(1+r)^2\nonumber\\
&-U_{j0}(1+r)^2-(\log U_{j})_r(1+r)(2+\sigma)\bigg] +h_{jk}.
\end{aligned}
\end{equation}
Since 
\begin{align}
(\log U_{j})_r\sim- \frac{m_j}{r},~\text{for~}r\gg 1,~m_j>2,
\end{align}
we choose positive constant  $R$ large enough such that 
$$\mathcal L_{jk}[z_{jk}] + h_k > 0 \ \ \text{for} \ \  r>R.$$
 In addition, with the fixed $R>0$, we set a large constant $C_{jk}>0$ to obtain
  \begin{equation}\label{311}
      \frac{C_{jk}\|h_j\|_{4 + \sigma}}{(1 + R)^{2 + \sigma}} - \max_{y \in B_R(0)}\tilde g_{jk} > 0,
   \end{equation}
   where $\tilde g_{jk}$ is bounded in $B_R(0)$.  Then, we apply the maximum principle to get
 \begin{equation}\label{eq2.10}
\tilde g_{jk} \le \frac{C_{jk}\|h_j\|_{4 + \sigma}}{(1 + r)^{2 + \sigma}}.
 \end{equation}
Similarly, we have
\begin{equation*}
\mathcal L_{jk}[- z_{jk}] + h_{jk} < 0 \ \   \text{for} \ \  r>R,
\end{equation*}
which implies
 \begin{equation}\label{eq2.1011}
-\tilde g_{jk} \le \frac{C_{jk}\|h_j\|_{4 + \sigma}}{(1 + r)^{2 + \sigma}},
 \end{equation}
 where the maximum principle was used.
 
 One collects (\ref{eq2.10}) and (\ref{eq2.1011}) to get
\begin{align}\label{tildegjkdecay}
\Vert \tilde g_{jk}\Vert_{L^\infty(\mathbb R^2)}\leq \frac{C_{jk}\Vert h_j\Vert_{4+\sigma}}{(1+r)^{2+\sigma}} \ \  \text{ in } \ \ 
 \mathbb R^2.
\end{align}
On the other hand, the existence of $\tilde g_{jk}$ follows from Fredholm alternative theorem since Lemma \ref{nondegeneracy1} implies mode $k\geq 1$ of $\tilde g_{jk}$-equation in (\ref{fourierexpansion}) does not admit any nontrivial kernel.

For the $w_{jk}$-equation with $k\geq 1$ in (\ref{fourierexpansion}), we first consider mode $1$ and  focus on the following equation
\begin{equation}\label{psimode1}
\left\{
\begin{aligned}
\frac{d^2 w_{11}}{dr^2}+\frac{1}{r}\frac{d w_{11}}{dr}-\frac{1}{r^2} w_{11}+b_{11}U_{1} w_{11}+b_{12}U_{2} w_{21}=f_{11},\\
\frac{d^2 w_{21}}{dr^2}+\frac{1}{r}\frac{d w_{21}}{dr}-\frac{1}{r^2} w_{21}+b_{21}U_{1} w_{11}+b_{22}U_{2} w_{21}=f_{21},
\end{aligned} 
\right.
\end{equation}
where 
\begin{align}\label{175fj1def}
f_{j1}:=-b_{j1}\tilde g_{11}-b_{j2}\tilde g_{21},~j=1,2.
\end{align}
With the aid of Lemma \ref{lemm13adjoint}, we find (\ref{psimode1}) only admits one bounded kernel $(\partial_r\log U_{1},\partial_r\log U_{2})^T$. Next, we shall show the existence of $(w_{11},w_{21})$ to (\ref{psimode1}).  To this end, we follow the argument shown in the proof of Lemma 2.3 in \cite{ao2016non} and define
\begin{align}
X_{\alpha}:=\bigg\{\boldsymbol{u}\in L^2_{\text{loc}}(\mathbb R^2)\times L^2_{\text{loc}}(\mathbb R^2)\big|\int_{\mathbb R^2}(1+r^{2+\alpha})|\boldsymbol{u}|^2\,dx<+\infty\bigg\},
\end{align}
\begin{align}
Y_{\alpha}:=\bigg\{\boldsymbol{u}\in W^{2,2}_{\text{loc}}(\mathbb R^2)\times W^{2,2}_{\text{loc}}(\mathbb R^2)\bigg|\int_{\mathbb R^2}(1+r^{2+\alpha})|\Delta \boldsymbol{u}|^2+\frac{|\boldsymbol{u}|^2}{(1+r^{2+\alpha})}\,dx<+\infty\bigg\}
\end{align}
for some $\alpha>0$.  Moreover, we denote
\begin{align}\label{diagsecond}
\mathcal L_1=\text{diag}{\bigg(\Delta_r-\frac{1}{r^2},\Delta_r-\frac{1}{r^2}\bigg)}+\boldsymbol{A},
\end{align}
with $\boldsymbol{A}$ defined by (\ref{boldsymbolA2025}) and consider 
\begin{align}
\mathcal L_1: Y_{\alpha}\rightarrow X_{\alpha}.
\end{align}
As shown in \cite{ao2016non}, one finds $\mathcal L_1$ is a bounded linear operator and has a closed range in $X_{\alpha}$ for $\alpha\in(0,\frac{1}{2})$.  It follows that $X_{\alpha}$ can be decomposed as
\begin{align}
X_{\alpha}=\text{Im}\mathcal L_1\oplus (\text{Im} \mathcal L_1)^{\perp}.
\end{align}
As a consequence, let $\boldsymbol{\phi}\in (\text{Im}\mathcal L_1)^{\perp}$,  then we have $(\mathcal L_1[{\boldsymbol{W}}],\boldsymbol{\phi})_{X_{\alpha}}=0$, $\forall \boldsymbol{W}\in Y_{\alpha}$, or equivalently, 
\begin{align}
(\mathcal L_1[\boldsymbol{W}],\boldsymbol{\varphi})_{L^2(\mathbb R^2)}=0,
\end{align}
where $\boldsymbol{\varphi}=(1+r^{2+\alpha})\boldsymbol{W}$.  Thus, $\mathcal L_1^*\boldsymbol{\varphi}=0$ in $\mathbb R^2$ with
$$\int_{\mathbb R^2}\frac{|\boldsymbol{\varphi}|^2(x)}{1+|x|^{2+\alpha}}dx<+\infty.$$
Then we apply Green's formula to get
\begin{align*}
|\boldsymbol{\varphi}(z)|=O(1+\log|z|).
\end{align*}
With this, by using Lemma \ref{lemm13adjoint}, we have $\boldsymbol{\varphi}\in\text{span}\{(\partial_r\log U^{1},\partial_r\log U^{2})^T\}.$   Thus, $$(\text{Im}\mathcal L_1)^{\perp}\subseteq \text{span}\{(\partial_r\log U_{1},\partial_r\log U_{2})^T\}.$$  In addition, we use the integration by parts and the fact that $\partial_r\log U_{j}\rightarrow 0$ as $r\rightarrow+\infty$ to get 
$$\text{span}\{(\partial_r\log U^{1},\partial_r\log U^{2})^T\}\subseteq (\text{Im}\mathcal L_1)^{\perp}.$$  Therefore, we obtain
\begin{align}
(\text{Im}\mathcal L_1)=\text{span}\{(\partial_r\log U^{1},\partial_r\log U^{2})^T\}^{\perp}.
\end{align}
Next, we claim that 
\begin{align}\label{weclaimf1dotzstar}
\int_{\mathbb R^2} ({\boldsymbol{f}}_{1}\cdot \boldsymbol{Z^*})dy=0,
\end{align}
where $\boldsymbol{Z^*}:=(\partial_r\log U^{1},\partial_r\log U^{2})^T$ given in (\ref{2024114logtransUup}) and ${\boldsymbol{f}}_{1}:=(f_{11},f_{21})^T$ defined in (\ref{175fj1def}). Indeed, by testing $y_i$, $i=1,2$ against the $g_j$-equation in (\ref{innerproblem}), one has
 \begin{equation}\label{lefthandsidebefore2024104}
 \int_{\R^2}\nabla \cdot (U_j \nabla g_j)y_i d y = \int_{\R^2}h_jy_id y = 0.
\end{equation}
Moreover, the left hand side of (\ref{lefthandsidebefore2024104}) can be written as
 \begin{equation}\label{3177}
    \int_{\R^2} \nabla \cdot(U_{j} \nabla g_j)y_i d y = - \int_{\R^2}U_{j}\nabla g_j\cdot e_i\d y  = \int_{\R^2}g_jU_{j}\nabla \log U_{j}\cdot  e_i\d y,
 \end{equation}
 where $e_1=(1,0)$ and $e_2=(0,1)$.  For $i =1$, we further calculate to get 
  \begin{equation}\label{eq2.13}
  \begin{split}
   \int_{\R^2}g_jU_{j}\nabla \log U_{j}\cdot e_1 dy &= \int_{0}^{2\pi}\int_0^{\infty}\sum_{j = 1}^{\infty} g_{jk}(r)e^{ik\theta} U_{j}(r) 
      (\p_r \log U_j)(r) r\cos \theta d r d\theta \\
      & = \pi\int_{0}^{\infty}g_{j1}(r)U_{j}(r)(\p_r \log U_j)(r)rd r.
   \end{split}
  \end{equation}
Collecting (\ref{lefthandsidebefore2024104}), (\ref{3177}) and (\ref{eq2.13}), one obtains 
  $$\int_{0}^{\infty}g_{j1}(r)U_{j}(r)(\p_r \log U_j)rd r=0.$$
Then we invoking the definition of ${\textbf{Z}}^*$ in Lemma \ref{lemm13adjoint} and \eqref{175fj1def} to finish the proof of our claim.
Thanks to (\ref{weclaimf1dotzstar}), we apply Fredholm alternative theorem to get that there exists a solution $\boldsymbol{W}_1\in Y_{\alpha}$ to problem (\ref{psimode1}).

Next, we derive the estimate of $\boldsymbol{W}_1.$ 
 Define 
\begin{align*}
 z_{j1}=\bar C_{j1}\Vert h_j\Vert_{4+\sigma} (1+r)^{\sigma}  +\delta r, ~j=1,2,
\end{align*}
and
\begin{align}
 \bar{\mathcal L_{11}}[z]=-\frac{d^2}{dr^2}z-\frac{1}{r}\frac{d}{dr}z +  \frac{1}{r^2}z - b_{11}U_{1}z- b_{12}U_{2} w_{21},\\
  \bar{\mathcal L_{21}}[z]=-\frac{d^2}{dr^2}z-\frac{1}{r}\frac{d}{dr}z + \frac{1}{r^2}z - b_{21}U_{1}w_{11} - b_{22}U_{2}z,
\end{align}
where $\bar C_{j1}>0$ are some large constants and $\delta>0$ is a small constant.
Then
we compute to get
\begin{align}
 \bar{\mathcal L_{j1}}[z_{j1}]+f_{j1}=&\bar C_{j1}\Vert h_j\Vert_{4+\sigma}(1+r)^{\sigma-2}\bigg[\sigma(1-\sigma)-\frac{1}{r}\sigma(1+r)+\frac{1}{r^2}(1+r)^2-b_{jj}U_{j}(1+r)^2\bigg]\nonumber\\
 &-2b_{ij,i\not=j}U_{i} w_{i1}  + f_{j1} - b_{11}\delta U_{1}r.
\end{align}
Noting that $(w_{11},w_{21})^T\in X_{\alpha}$ and $U_j \sim -\frac{m_j}{ r}$ for $r$ large and $m_j>2$, we choose $\bar C_{j1}>0$ and $R_j>0$ large enough to obtain
\begin{equation*}
\bar{\mathcal L_{j1}}[z_{j1}]+f_{j1}>0,\text{ for }r>R_j.
\end{equation*}
 With the fixed $R_{j}>0$, we further set a large constant $\bar C_{j1}>0$ to get
  \begin{equation}\label{3111}
      {\bar C_{j1}\|h_j\|_{4 + \sigma}}{(1 + R)^{ \sigma}} - \max_{y \in B_R(0)}w_{j1} > 0,
   \end{equation}
   where $w_{j1}$ is bounded in $B_R(0)$.  By using the maximum principle on annulus $B_{\tilde R_j}(0)\backslash B_{R_j}(0)$, we have
 \begin{equation}\label{eq2.101} 
 w_{j1} \le {\bar C_{j1}\|h_j\|_{4 + \sigma}}{(1 + r)^{ \sigma}},
 \end{equation}
where we let $\tilde R_j\rightarrow +\infty$ then $\delta\rightarrow 0$.
Similarly, we apply the maximum principle into $-w_{j1}$ and compute to get
 \begin{equation}\label{eq2.1012} 
 |w_{j1} |\le {\bar C_{j1}\|h_j\|_{4 + \sigma}}{(1 + r)^{ \sigma}},
 \end{equation}
 where \eqref{eq2.101} was used.
 
  It remains to analyze mode $k\geq 2$ of the linearization of the $w_j$-equations with $j=1,2$ shown in \eqref{innerproblem}.  As stated in Lemma \ref{lemm13adjoint}, there is not any nontrivial kernel to the mode-$k$ equations with $k\geq 2$, which satisfies $w_j\leq C(1+|z|^{\tau})$ for some $\tau\in[0,1).$  Similarly as above, we consider 
  \begin{align}
  \mathcal L_k: Y_{\alpha}\rightarrow X_{\alpha},
  \end{align}
  where 
  \begin{align}\label{diagkth}
\mathcal L_k:=\text{diag}{\bigg(\Delta_r-\frac{k^2}{r^2},\Delta_r-\frac{k^2}{r^2}\bigg)}+\boldsymbol{ A}
\end{align}
with $\boldsymbol{A}$ defined by (\ref{boldsymbolA2025}).
It can be shown that $\mathcal L_k$ is a bounded linear operator and has a closed range in $X_{\alpha}$ for $\alpha\in(0,\frac{1}{2}).$  Thus, $X_{\alpha}$ can be decomposed as
  \begin{align}
  X_{\alpha}=\text{Im} \mathcal L_k\oplus (\text{Im}\mathcal L_k)^{\perp}.
  \end{align}
  Similarly, we prove that $(\text{Im}\mathcal L_k)^{\perp}= \emptyset$ by using Lemma \ref{lemm13adjoint}.  Then we apply the Fredholm alternative theorem to obtain there exists a solution ${\boldsymbol{W}}_k:=(w_{1k}, w_{2k})^T$ satisfying
  \begin{equation}\label{psimodek}
\left\{ 
\begin{aligned}
\frac{d^2 w_{1k}}{dr^2}+\frac{1}{r}\frac{d w_{1k}}{dr}-\frac{k^2}{r^2} w_{1k}+b_{11}U_{1} w_{1k}+b_{12}U_{2} w_{2k}=f_{1k},\\
\frac{d^2 w_{2k}}{dr^2}+\frac{1}{r}\frac{d w_{2k}}{dr}-\frac{k^2}{r^2} w_{2k}+b_{21}U_{1} w_{1k}+b_{22}U_{2} w_{2k}=f_{2k}.
\end{aligned}
\right.
\end{equation}
We next establish the estimate of ${\boldsymbol{W}}_k\in Y_{\alpha}.$  To this end, we define
  \begin{equation}
  \bar {\mathcal L}_{jk}[z]=-\frac{d^2}{dr^2}z-\frac{1}{r}\frac{d}{dr} + \frac{k^2}{r^2}z - b_{jj}U_{j}z - b_{i\not=j}U_{i}w_{ik},~j=1,2,
  \end{equation}
  and
  \begin{align}
  z_{jk}:=\bar C_{jk}\Vert h_j\Vert_{4+\sigma}(1+r)^{\sigma}  + \delta_k r,
  \end{align}
  where constant $\sigma>0$, $\delta_k>0$ are small and $\bar C_{jk}>0$ is large. 
  By the direct computation, we use the maximum principle, choose $\bar C_{jk}>0$ large enough and take $\delta_k\rightarrow 0$ such that 
  \begin{align}
  |w_{jk}|\leq \bar C_{jk}\Vert h_j\Vert_{4+\sigma}(1+r)^{\sigma},
  \end{align}
  where $\sigma>0$ is sufficiently small.  The existence of $(w_{1k},w_{2k})^T$ to (\ref{psimodek}) directly follows from the invertibility of $\mathcal L_{k}$. 

  Recall that for $j=1,2$ and $k\in\mathbb N,$ 
\begin{align}
\phi_{jk}=\tilde g_{jk}+U_{j}\psi_{jk}.
\end{align}
Thus, we have there exist $\phi_{jk}$ satisfying
\begin{align}
\vert \phi_{jk}\vert \leq C_{jk}\frac{1}{(1+r)^{2+\sigma}},
\end{align}
where $\sigma\in(0,1)$ small.  This completes the proof of our lemma.
\end{proof}

In Lemma \ref{lemma31}, we establish the existence and a-priori estimate of $(\varphi_1,\varphi_2)$  to inner problem (\ref{innerproblem}), which corresponds to the linearization around the interior spots to (\ref{ss}).  Whereas, if the center is at the boundary $\partial\Omega,$ we must solve the inner problem (\ref{innerproblem}) in the half space $\R^2_{+} = \{(y_1, y_2) \in \R^2:y_2 \ge 0\}$ rather than $\R^2.$  To this end, we define norm $\Vert \cdot\Vert_{\nu,H}$ as 
$$\Vert h\Vert_{\nu,H}:=\sup_{y\in\R^2_+}|h|(1+|y|)^{\nu},~\nu>0,$$
and develop the following solvability results.
 
\begin{Lemma}\label{lemma32}
Given any function $h_j$, $j=1,2$ and $\beta_j(x)$ satisfying
    \begin{equation}\label{331}
       \int_{\R^2_{+}}h_j dy  - \int_{\p \R^2_{+}} \beta_j dS_y= 0,  \ \   \int_{\R^2_+}h_jy_1dy  - \int_{\p\R^2_+}\beta_j y_1dS_y  = 0,
    \end{equation}
    and  $\|h_j\|_{4 + \sigma, H} < \infty$ with $\sigma\in(0,1)$, we have the problem
    \begin{equation}\label{332}
  \begin{cases}
    \ds L^{\text{inn}}_k[\varphi_1,\varphi_2] =(h_1,h_2)^T\ \ &\text{in} \ \ \R^2_{+},\\
    \ds U_{j} \frac{\p g_j}{\p \boldsymbol{\nu}} = \beta_j(x)\ \ &\text{on} \ \  \p \R^2_{+},~j=1,2
  \end{cases}
\end{equation}
    admits a solution $(\varphi_1,\varphi_2)$ satisfying the following estimate: 
      \begin{equation}\label{333}
        \|\varphi_j\|_{2 + \sigma, H} \le C_j \|h_j\|_{4 + \sigma, H},
      \end{equation}
     where $C_j>0$ is a constant and $g_j=\frac{\varphi_j}{U_{j}}- w_j$.  Moreover, $\boldsymbol{\varphi}$ satisfies $(\varphi_1,\varphi_2)^T =  \hat{\mathcal T_p}[h_1,h_2]$, where $\hat{ \mathcal T_p}[h_1,h_2]$ is defined by a linear operator.   
\end{Lemma}
\begin{proof} For any given $(\beta_1, \beta_2)$ defined on $\p\R^2_{+} \times \p\R^2_{+}$, we have there exists a function pair $(\tilde \varphi_{j}, \tilde w_{j})$ such that 
 \begin{equation*}
  \frac{\p \tilde\varphi_j}{\p \boldsymbol{\nu}} - U_{j}\frac{\p \tilde w_{j}}{\p \boldsymbol{\nu}} = \beta_j \ \  \text{on}  \ \  \p \R^2_{+},
 \end{equation*}
 where $\|\tilde \varphi\|_{2 + \sigma, H} \le C \|h_j\|_{4 + \sigma, H}$, $j=1,2$.  Then, we define $\vartheta_j := \frac{\varphi_{j}}{U_{j}} - w_{j}$ and find $\vartheta_j$ satisfies
\begin{equation}\label{335}
    \int_{\R^2_{+}}U_{j}\nabla \vartheta_{j} \cdot  {\boldsymbol{e}}_1 dy  = 0,
  \end{equation}
  where ${\boldsymbol{e}}_1=(1,0)$.  Now, the problem (\ref{332}) is transformed into the following form
  \begin{equation}\label{336}
  \begin{cases}
    \ds \nabla\cdot (U_{j} \nabla \bar{g}_j) = h_j - \nabla \cdot (U_{j}\nabla \vartheta_j) \ \ &\text{in} \ \ \R^2_{+}\\
    \ds U_{j} \frac{\p \bar{g}_j}{\p \boldsymbol{\nu}} = 0 \ \ &\text{on} \ \  \p \R^2_{+},~j=1,2,
  \end{cases}
\end{equation}
where $\bar{g}_j: = g_j - \vartheta_j$.  Define the solution of (\ref{336}) as  $(\phi_\vartheta, \psi_\vartheta)$ and
\begin{equation*}
   \bar{g}_{j} :=
   \begin{cases}
      \bar{g}_j(y_1, y_2) \ \ &\text{for} \ \  y_2 \ge 0; \\
      \bar{g}_j(y_1, -y_2) \ \ &\text{for} \ \  y_2 < 0,
   \end{cases}
\end{equation*}
then we have the equation in (\ref{336}) is evenly extended into the whole space, which is
 \begin{equation*}
    \nabla\cdot(U_{j} \nabla \bar{g}_{j}) = \tilde{h}_j \ \ \text{in} \ \ \R^2,~j=1,2,
 \end{equation*}
where
 \begin{equation*}
      \tilde{h}_j(y_1, y_2)=
      \left\{
     \begin{aligned}
        & h_j(y_1, y_2) - \nabla\cdot (U_{j}\nabla \vartheta_j)(y_1, y_2) \ \   &\text{for} \ \ y_2 \ge 0, \\
        & h_j(y_1, -y  _2) - \nabla \cdot(U_{j}\nabla \vartheta_j)(y_1, -y_2) \ \  &\text{for} \ \ y_2 < 0.
     \end{aligned}
     \right.
 \end{equation*}
It is easy to check that $\|\tilde{h}_j\|_{4 + \sigma} < \infty$ due to $\|h_j\|_{4 + \sigma, H}<+\infty$.  The key step is the verification of the orthogonality condition.  To finish it , we first obtain from the property of even function that
 \begin{equation*}
    \int_{\R^2_-} h_j(y_1, -y_2) - \nabla\cdot (U_{j}\nabla \vartheta_j)(y_1, -y_2)dy = \int_{\R^2_+} h_j(y_1, y_2) - \nabla\cdot (U_{j}\nabla \vartheta_j)(y_1, y_2)dy,
 \end{equation*}
 and
 \begin{equation*}
   \int_{\R^2_-} \big[h_j(y_1, -y_2) - \nabla\cdot (U_{j}\nabla \vartheta_j)(x_1, -x_2)\big]y_1 dy= \int_{\R^2_+} \big[h_j(y_1, y_2) - \nabla\cdot (U_{j}\nabla \vartheta_j)(y_1, y_2)\big]y_1dy.
 \end{equation*}
 Then, by using condition (\ref{331}), we have from the divergence Theorem that
 \begin{equation*}
 \begin{split}
   \int_{\R^2}\tilde{h}_j dy &=  2\int_{\R^2_+} h_j- \nabla\cdot(U_{j}\nabla \vartheta_j) dy  = 2\int_{\R^2_+} h_jdy- 2\int_{\p\R^2_+} (U_{j} \nabla \vartheta_j)\cdot \boldsymbol{\nu}dS_y  \\
        & = 2\int_{\R^2_+} h_j dy - 2\int_{\p\R^2_+} U_{j} \frac{\p \vartheta_j}{\p \boldsymbol{\nu}}dS_y =  2\int_{\R^2_+} h_j dy-   2\int_{\p\R^2_+} \beta_j dS_y = 0,
   \end{split}
 \end{equation*}
 which implies the mass condition in (\ref{innercon}) holds.  For the first moment condition, we utilize the integration by parts and (\ref{335}) to get
\begin{equation}\label{338}
\begin{split}
   \int_{\R^2_+}\tilde{h}_jy_1dy &= 2 \int_{\R^2_+}[h_j - \nabla\cdot (U_{j} \nabla \vartheta_j)]y_1dy \\
          & = 2\int_{\R^2_+}h_j y_1dy - 2 \int_{\p \R^2_+} y_1U_{j0}\nabla \vartheta_j \cdot \boldsymbol{\nu} dS_y + 2\int_{\R^2_+}U_{j}\nabla \vartheta_j \cdot \boldsymbol{e}_1dy \\
          & = 2\int_{\R^2_+}h_j y_1 dy- 2 \int_{\p \R^2_+}\beta_j y_1dS_y = 0.
    \end{split}
\end{equation}
  Since $\tilde{h}_j$ is even with respect to $y_1$, we can easily obtain from (\ref{338}) that $\int_{\R^2}\tilde h_j y_1dy=0$, which completes the verification of orthogonality condition (\ref{innercon}).  Therefore, we can utilize the results shown in Lemma \ref{lemma31} to find there exists the solution $(\tilde\phi_1,\tilde\phi_2,\tilde\psi_1,\tilde\psi_2)$ to the following system:
  \begin{equation*}
   \left\{
   \begin{aligned}  
   -\Delta \tilde{\psi}_1 &= U_{1} \tilde{\psi}_1 + U_{1} \tilde{g}_1, \ \ &\text{in} \ \ \R^2, \\
   -\Delta \tilde{\psi}_2 &= U_{2} \tilde{\psi}_2 + U_{2} \tilde{g}_2, \ \ &\text{in} \ \ \R^2,\\
      - \Delta \tilde{\psi}_1 &= a_{11}\tilde{\phi}_1 +a_{12}\tilde\phi_2  \ \ &\text{in} \ \ \R^2,\\
         - \Delta \tilde{\psi}_2 &= a_{21}\tilde{\phi}_1 +a_{22}\tilde\phi_2  \ \ &\text{in} \ \ \R^2.
      \end{aligned}
      \right.
  \end{equation*}
In particular, $\tilde \phi_j$, $j=1,2$ satisfies the following estimate
    \begin{equation}\label{340}
        |\tilde{\phi}_j| \le C_j \frac{\|h_j\|_{4 + \sigma, H}}{(1 + r)^{2 + \sigma}}.
    \end{equation}
Since $\tilde \phi_j$ is even, it can be defined as the even extension of $\phi_j$.  By using (\ref{340}), we further show that $(\phi_1,\phi_2)$ is the solution of (\ref{332}) and satisfies
    \begin{equation*} 
       \|\phi_j\|_{2 + \sigma, H} = \|\tilde{\phi}_j + \phi_{j0}\|_{2 + \sigma, H} \le  \|\tilde{\phi}_j\|_{2 + \sigma, H} + \|\phi_{j0}\|_{2 + \sigma, H} \le C_j \|h_j\|_{4 + \sigma, H},
    \end{equation*}
 which completes the proof of (\ref{333}) and this Lemma.
\end{proof}

 In the next section, we focus on the outer problem and establish the outer linear theory.
\section{Outer Linear Theory}
Similarly as shown in Subsection 3.2 of \cite{kong2022existence}, we first formulate the outer operator.  
Concerning the $\varphi_j$-equations of (\ref{ss}), we define
\begin{align}\label{barLphi}
 \bar L_j[\varphi_1,\varphi_2]=-\Delta\varphi_j +\nabla\cdot(P_j\nabla  w_j)+\nabla\cdot(\varphi_j\nabla \bar Q_j),
\end{align}
where
$$P_j(y)=\sum_{k=1}^m  U_{jk}(y - \xi'_k)~\text{and}~  \bar Q_j(y) = \sum_{k =1}^{m}\Big(-m_j\log\varepsilon +  \Gamma_{jk}(y - \xi'_k)+ \hat{c}_{jk}H_j^{\varepsilon}(y,\xi_{k})\Big).$$
Here $c_{jk}$ given in (\ref{generate1}) is included in $U_{jk}$.  Then, we substitute $\Delta \bar Q_j=\va^2\bar Q_j - a_{j1}P_1 - a_{j2}P_2$ into (\ref{barLphi}) and expand it to obtain
\begin{align*}
\bar L_j[\varphi_1,\varphi_2]=&\Delta \varphi_j-\nabla \varphi_j \cdot \nabla \bar Q_j-\va^2\bar Q_j \varphi_j\\
&-\nabla P_j\cdot \nabla w_j - P_j\Delta w_j + (a_{j1}P_1+a_{j2}P_2)\phi_j:=I_{31}+I_{32}.
\end{align*}
Due to the decay property of $P_j$, we have $I_{32}$ is negligible in the outer region and $I_{31}$ dominates.  On the other hand, the leading term in the logistic source is $\va^2\lambda_j \bar u_j\phi_j$.  Combining this term with $I_{31}$, we define the outer operator as
\begin{equation}\label{outeroperator2025}
   L^{o}_j[\varphi_1,\varphi_2] = -\Delta \varphi_j +\nabla \varphi_j \cdot \nabla \bar Q_j - \va^2(\lambda_j \bar{u}_j-\bar Q_j) \phi_j\ \ \text{in}\ \ \Omega_{\va},
\end{equation}
 where $ \va y =x$ and $\va \xi_j' = \xi_j$.  Moreover, the outer norm $\|\cdot\|_{\nu, o}$, $\nu>0$ is defined as
   \begin{equation*}
      \|h_j\|_{\nu, o}: = \sup_{y \in \Omega_{\va}}\frac{ |h_j|}{ \sum\limits_{k =1}^{m}(1 + |y - \xi_k'|)^{-\nu}}.
    \end{equation*}
We next derive a-priori estimate of outer solution $(\varphi_1,\varphi_2)$ then prove its existence.  Our results are summarized as
\begin{Lemma}\label{lemma33}
Assume that $\|h_j\|_{b+2, o} < +\infty$ and $\lambda_j\bar u_j<\bar C_j,$ where $\bar C_j:=\sum\limits_{k=1}^m \hat c_{jk}C_{\Omega}$ and $C_{\Omega}$ is the positive lower bound of Green's function,
   then the problem
     \begin{equation}\label{problem1}
     \begin{cases}
         L_j^{o}[\varphi_1,\varphi_2]+ h_j=0 \ \ \ &\text{in} \ \    \Omega_{\va}, \\
        \ds\frac{\p \varphi_j}{\p \boldsymbol{\nu}} = 0  \ \ \  &\text{on} \ \ \p \Omega_{\va}
        \end{cases}
     \end{equation}
    admits the solution $(\varphi_1,\varphi_2)^T=\mathcal{T}_o[h_1,h_2]$ satisfying 
      \begin{equation}\label{345}
            \|\varphi_j\|_{b, o} \le C\|h_j\|_{b +2, o}, 
      \end{equation}
      where $b>0$ is a constant, $C>0$ is a constant and $\mathcal{T}_o[h_1,h_2]$ is a continuous
      linear mapping.
\end{Lemma}
\begin{proof} 
To show that a-priori estimates (\ref{345}) satisfied by $\varphi_1$ and $\varphi_2$ hold, we can follow the argument shown in the proof of Lemma 3.3 in \cite{kong2022existence} with some slight modification and the details are omitted.  The existence of $(\varphi_1,\varphi_2)$ immediately follows from Fredholm alternative theorem.
\end{proof}

\begin{Remark}
As shown in \cite{davila2024existence,kong2022existence}, the principal parts of outer operators in single species Keller-Segel models can be formally regarded as $6$-dimensional Laplacians.  Whereas, the principal term in $L^{\text{o}}_j[\varphi_1,\varphi_2]$, $j=1,2$ defined by \eqref{outeroperator2025} is approximated by the $m_j+2$-dimensional Laplacian with $m_j>2$ since the algebraic decay rate of cellular density $u$ is $m_j.$
\end{Remark}

Lemma \ref{lemma33} demonstrates that the outer problem (\ref{problem1}) admits the decay solution if source $(h_1,h_2)$ decays fast.  Next, we shall first employ Lemma \ref{lemma31} and Lemma to construct the multi-interior spots to (\ref{ss}) via the inner-outer gluing scheme.  The existence of multi-boundary spots can also be shown by invoking Lemma \ref{lemma32} and the detailed discussions are exhibited in Section \ref{section5}.

\section{Inner-outer Gluing Procedures}\label{section5}
This section is devoted to the construction of multiple interior spots and boundary spots.  Noting that when the locations of spots are at the boundary $\partial\Omega$, the inner problem near each spot has to be solved in the half space $\R^2_{+}$ and it is necessary to use the straighten transformation, which may cause some difficulty.  Thus, we divide our following discussions into two cases: construction of multi-interior spots and construction of multi-boundary spots. 


\subsection{Construction of Interior Spots}\label{sub51}
To begin with, we give some preliminary notation and definitions.  Recall that  
the inner and outer norms $\|\cdot\|_{\nu, k}$ and $\|\cdot\|_{b, o}$ are given by 
  \begin{equation}\label{eq7.2}
     \|h\|_{\nu, k} := \sup_{y \in \mathbb{R}^2} |h(y)|(1 + |y - \xi'_k|)^{\nu}  \ \ \text{and} \ \   
     \|h\|_{\nu, o}:= \sup_{x \in \Omega} \frac{|h|}{\sum_{k =1}^m (1 + |x - \xi'_k|)^{-\nu}}.
  \end{equation}
In addition, we denote $\delta'= \inf_{j \neq k}|\xi_j - \xi_k|$ and the cut-off functions as $\eta_j := \eta(|y - \xi'_k|) > 0$, where
  \begin{equation}\label{eq7.3}
      \eta(r) : = 
        \begin{cases}
            1 , \ \   r \le \frac{\delta}{\varepsilon}  ;\\
            0, \ \   r > \frac{2\delta}{\varepsilon}
        \end{cases}
  \end{equation}
and $\delta$ is a fixed number. 

With the aid of \eqref{eq7.2} and \eqref{eq7.3}, we now decompose $(\phi_1, \phi_2)$ and $(w_1, w_2)$ into the following form 
\begin{equation}\label{eq7.4}
\begin{split}
   \varphi_j = \sum_{k =1}^m \varepsilon^{\gamma_{1}}c_{jk}\varphi_{jk}(y)\eta_k +  \varepsilon^{\gamma_{2}}\varphi_j^o, \ \ \ \  &w_{jk} = (-\Delta)^{-1}(b_{j1}\varphi_{1k} + b_{j2} \varphi_{2k})  \\
    w_j^o = \va^{\gamma_{2}}(-\Delta  + \varepsilon^2)^{-1}(a_{j1}\varphi^o_1 + a_{j2}\varphi^o_2), \ \ \ \  &  w'_{jk} = (- \Delta + \varepsilon^2)^{-1}(b_{j1}\varphi_{1k}\eta_k + b_{j2} \varphi_{2k}\eta_k),
    \end{split}
 \end{equation}
where $\gamma_{1} > 0$ and $\gamma_{2} > 0$ will determined later on. In light of (\ref{2442025116}) and (\ref{inn-out-before}), we find for $j=1,2,$  
\begin{equation*}
   L_j(\varphi_j) = \varepsilon^{\gamma_{1}}\sum_{k =1}^m L_j[\varphi_{jk} \eta_k] + \varepsilon^{\gamma_{2}}L_j[\varphi^o_j] = \varepsilon^2 \sum_{k =1}^5 I_{jk}(\mathbf P) + \varepsilon^2 I_6(\boldsymbol{\varphi}, \mathbf P) + \varepsilon^2 I_7(\boldsymbol{\varphi}, \mathbf P).
\end{equation*}
Invoking the definitions of $L_{jk}^{\text{inn}}$ and $L_j^{o}$ given by (\ref{innerlphi}) and (\ref{outeroperator2025}), respectively, we have 
\begin{equation}\label{eq7.5}
\begin{split}
  & \varepsilon^{\gamma_{1}}\sum_{k =1}^m c_{jk} L^{\text{inn}}_{jk}[\varphi_{jk}\eta_k] + \varepsilon^{\gamma_{2}} L^o_j[\varphi_1^o,\varphi_2^o]\\
   = &\varepsilon^2 h_j(\boldsymbol{\varphi}, \mathbf P)-\va^{\gamma_2}\nabla\cdot(P_j\nabla \bar w_j^o)+\va^{\gamma_2}\Delta \bar Q_j\varphi^o_j-\va^{2+\gamma_2}\lambda_j\bar u_j\varphi^o_j\\
  &+\va^{1+\gamma_2}\sum_{k=1}^m\nabla_y\cdot(\varphi_j^o\nabla_xH^{\va}_{jk}(\va y,\xi_{k}))+\va^{\gamma_1+1}\bigg(\sum_{k=1}^m\varphi_{jk}\nabla\eta_{k}\bigg)\cdot\bigg(\sum_{j=1}^m\nabla H_{jk}^{\va}(\va y,\xi_k)\bigg)\\
  &-\va^{\gamma_1}\sum_{k=1}^m\sum_{l\not=k}[\nabla\cdot(U_{jl}\nabla\bar w_{jk})+\nabla\cdot(\varphi_{jk}\nabla\Gamma_{jl})\eta_k]-\va^{\gamma_1}\sum_{k=1}^m\sum_{l\not=k}\varphi_{jk}\nabla\eta_{k}\cdot\nabla\Gamma_{jl},
     \end{split}
\end{equation}
where $h_{j}(\boldsymbol{\varphi},\mathbf {P}):=\sum_{k=1}^7I_{jk}(\varphi_{11},\cdots,\varphi_{1m},\varphi_{21},\cdots,\varphi_{2m},\varphi^o,\textbf{P})$ and $\bar w_j^o(y)=w_j^o(\va y),$  $\bar w'_{jk}(y)=w_{jk}(\va y)$. 
Define 
\begin{align}\label{Ferror2025116}
F_j(\boldsymbol{\varphi},\textbf{P}):=&\varepsilon^2 h_j(\boldsymbol{\varphi}, \mathbf P)-\va^{\gamma_2}\nabla\cdot(P_j\nabla \bar w_j^o)+\va^{\gamma_2}P_j\varphi^o_j-\va^{2+\gamma_2}\lambda_j\bar u_j\varphi^o_j\nonumber\\
  &+\va^{1+\gamma_2}\sum_{k=1}^m\nabla_y\cdot(\varphi_j^o\nabla_xH^{\va}_{jk}(\va y,\xi_{k}))\nonumber\\
  &-\va^{\gamma_1}\sum_{k=1}^m\sum_{l\not=k}[\nabla\cdot(U_{jl}\nabla\bar\psi_{jk})+\nabla\cdot(\varphi_{jk}\nabla\Gamma_{jl})\eta_k].
\end{align}

Moreover, a simple computation shows that 
\begin{equation}\label{eq7.6}
\begin{split}
   L^{\text{inn}}_{jk}[\varphi_{jk}\eta_k] = & \eta_k L^{\text{inn}}_{jk}[\varphi_{jk}] - 2\nabla \varphi_{jk} \cdot \nabla \eta_k - \varphi_{jk}\Delta \eta_k +  \varphi_{jk} \cdot \nabla \Gamma_{jk}(y-\xi'_k) \nabla\eta_k  \\
     &+ \nabla \cdot (U_{jk}\nabla w'_{jk}) - \nabla \cdot (U_{jk}\nabla w_{jk}) \eta_k.
   \end{split}
\end{equation}
 Combining \eqref{Ferror2025116} with \eqref{eq7.6}, one denotes
 \begin{align}\label{FJK2025116}
 F_{jk}(\boldsymbol{\varphi},\mathbf P):=F_j(\boldsymbol{\varphi},\textbf{P})\eta_k - \va^{\gamma_1}\nabla \cdot (U_{jk}\nabla w'_{jk}) + \va^{\gamma_1}\nabla \cdot (U_{jk}\nabla w_{jk}) \eta_k.
 \end{align}
Moreover, we define $J_j$ as  
 \begin{equation}\label{eq7.8}
 \begin{split}
     J_j(\boldsymbol{\varphi}, \mathbf{P}) =& F_j(\boldsymbol{\varphi}, \mathbf{P})(1 - \sum_{k =1}^m \eta_k^2)+\va^{\gamma_1+1}\bigg(\sum_{k=1}^m\varphi_{jk}\nabla\eta_{k}\bigg)\cdot\bigg(\sum_{j=1}^m\nabla H_{jk}^{\va}(\va y,\xi_k)\bigg)  \\
     & +\va^{\gamma_1}\sum_{k=1}^m\big[2\nabla \varphi_{jk} \cdot \nabla \eta_k+\varphi_{jk}\Delta \eta_k -  \varphi_{jk} \cdot \nabla \Gamma_{jk}(y-\xi'_k) \nabla\eta_k \big]\\
     &-\va^{\gamma_1}\sum_{k=1}^m\sum_{l\not=k}\varphi_{jk}\nabla\eta_{k}\cdot\nabla\Gamma_{jl}.
\end{split}
 \end{equation}
Collecting (\ref{eq7.5}), (\ref{FJK2025116}) and \eqref{eq7.8}, we formulate the system satisfied by $\varphi_{jk}$ and $\varphi^o_j$ as
\begin{equation}\label{eq7.9}
\left\{
\begin{aligned}
   L^{\text{inn}}_{jk}[\varphi_{1k},\varphi_{2k}] &= \varepsilon^{-\gamma_{1}}F_{jk}(\boldsymbol{\varphi}, \textbf P), \ \ &&\text{in} \ \  \mathbb{R}^2,  \ \  k =1, \cdots, m, \ \ \\
   L_j^o[\varphi_1^o,\varphi_2^o] &=   \varepsilon^{-\gamma_{2}} J_j(\boldsymbol{\varphi}, \mathbf P),  \ \ &&\text{in} \ \  \Omega_{\varepsilon},  \ \  j =1, 2.
   \end{aligned}
   \right.
\end{equation}
In order to use the solvability results stated in Lemma \ref{lemma31}, we must impose the orthogonality conditions.  To this end, we let compactly supported functions $W_{ik}$ be radial with respect to $\xi'_k$ and satisfy
  \begin{equation*}
     \int_{\mathbb{R}^2}W_{0k}(y - \xi'_k)d y = 1,
  \end{equation*}
and compactly supported radial functions $W_{lk},~l=1,2$ satisfy
  \begin{equation*}
      \int_{\mathbb{R}^2}W_{lk}(|y - \xi'_k|)(y - \xi'_k)_l d y = 1.
  \end{equation*}
With these test functions, we modify \eqref{eq7.9} as the following problem
\begin{equation}\label{eq7.10}
\left\{
\begin{aligned}
   L^{\text{inn}}_{jk}[\varphi_{1k},\varphi_{2k}] &= \varepsilon^{-\gamma_{1}}F_{jk}(\boldsymbol{\varphi}, \textbf P)  - \sum_{l = 0, 1, 2}m_{jlk}[\varepsilon^{- \gamma_{1}}F_{jk}(\boldsymbol{\varphi}, \mathbf P)]W_{lk} \ \ \text{for} \ \  k =1, \cdots, m, \\
   L_j^o[\varphi^o_1,\varphi_2^o] &=   \varepsilon^{-\gamma_{2}} J_j(\boldsymbol{\varphi}, \mathbf P), \ \  j =1, 2,
   \end{aligned}
   \right.
\end{equation}
 where
  \begin{equation}\label{eq7.11}
    m_{j0k}[h_j] = \int_{\mathbb{R}^2} h_j d y  \ \  \text{and} \ \  m_{jlk}[h_j] = \int_{\mathbb{R}^2}h_j(y)(y - \xi'_k)_l d y. 
  \end{equation}

Thanks to Lemma \ref{lemma31} and \ref{lemma33}, we find if right hand sides in (\ref{eq7.10}) are given, then there exists a solution of 
  \begin{equation*}
     (\varphi_{11}, \cdots, \varphi_{1m}, \varphi^o_1,  \varphi_{21}, \cdots , \varphi_{2m}, \varphi^o_2).
  \end{equation*}
 to \eqref{eq7.10} provided with
\begin{equation}\label{eq7.12}
 m_{jlk}[\varepsilon^{- \gamma_{1}}F_{jk}(\boldsymbol{\varphi}, \mathbf P)] =0 \ \ \text{for} \ \ l = 0, 1, 2; \,  k = 1, 2, \cdots , m;\,  j = 1, 2.
\end{equation}
Moreover, 
   \begin{equation}\label{eq7.13}
     \begin{split}
      \varphi_{jk} &= \mathcal{A}_{jk}(\varphi_{11}, \cdots, \varphi_{1m}, \varphi^o_1,  \varphi_{21}, \cdots , \varphi_{2m}, \varphi^o_2, \mathbf{P}_0 + \mathbf{P}_1), \\
      \varphi^o_j &= \mathcal{A}_j^o(\varphi_{11}, \cdots, \varphi_{1m}, \varphi^o_1,  \varphi_{21}, \cdots , \varphi_{2m}, \varphi^o_2, \mathbf{P}_0+ \mathbf{P}_1), \\
      \mathbf{P} &= \mathcal{A}_p(\varphi_{11}, \cdots, \varphi_{1m}, \varphi^o_1,  \varphi_{21}, \cdots , \varphi_{2m}, \varphi^o_2, \mathbf{P}_0 +  \mathbf{P}_1),
      \end{split}
   \end{equation}
where $\mathcal{A}_{jk}$, $\mathcal{A}^o_j$ and $\mathcal{A}_p$ are linear operators and 
\begin{equation}\label{eq7.14}
    \mathbf{P} = \mathbf{P}_0 + \mathbf{P}_1 \ \ \text{with} \ \ \mathbf{P}_0 = (c_{110}, \cdots, c_{1m0}, c_{210}, \cdots , c_{2m0}, \xi_{10}, \cdots, \xi_{m0}).
\end{equation}
Then we use \eqref{eq7.13} and \eqref{eq7.14} to rewrite the solutions and the operators in the form of 
\begin{equation}\label{eq7.15}
   \vec{\varphi} =  (\varphi_{11}, \cdots, \varphi_{1m}, \varphi^o_1,  \varphi_{21}, \cdots , \varphi_{2m}, \varphi^o_2, \mathbf{P}_0 + \mathbf{P}_1)
\end{equation}
and
  \begin{equation}\label{eq7.16}
     \mathcal{A}(\vec{\varphi}) = (\mathcal{A}_{11}(\vec{\varphi}), \cdots, \mathcal{A}_{1m}(\vec{\varphi}), \mathcal{A}^o_1(\vec{\varphi}), \mathcal{A}_{21}(\vec{\varphi}), \cdots, \mathcal{A}_{2m}(\vec{\varphi}), \mathcal{A}^o_2(\vec{\varphi}), \mathcal{A}_p(\vec{\varphi})).
  \end{equation}
We further define 
  \begin{equation*}
  \begin{split}
     X_k= \Big\{\varphi \in L^\infty(\mathbb{R}^2): &\nabla \varphi \in L^\infty(\mathbb{R}^2); \|\varphi\|_{2 + \sigma, k} < \infty,  \\
    & \int_{\mathbb{R}^2}\varphi d y = 0 \ \ \text{and} \ \  \int_{\mathbb{R}^2} \varphi(y) (y - \xi'_k)_l d y = 0, l = 1, 2 \Big\},
     \end{split}
  \end{equation*}
  \begin{equation*}
      X_o = \Big\{\varphi \in L^\infty(\Omega_{\varepsilon}) :  \nabla \varphi \in L^\infty(\Omega_{\varepsilon}), \, \|\varphi\|_{b, o} < +\infty, \, \frac{\partial \varphi}{\partial \boldsymbol{\nu}} = 0\text{ on }\partial\Omega_{\va}\Big\}
  \end{equation*}
and 
 \begin{equation*}
 \begin{split}
   X_p = \Big\{(c_{11}, \cdots , c_{1m},& c_{21}, \cdots, c_{2m}, \xi_1, \cdots, \xi_m) \in \mathbb{R}^m \times \mathbb{R}^m \times (\mathbb{R}^2)^m: \\
       & \|\mathbf{P}\|_p = \sup_k|c_{1k}| + \sup_{k}|c_{2k}| + \sup_k|\xi_k| \Big\}.
        \end{split}
 \end{equation*}
We collect them to define $X$ as 
  \begin{equation}\label{eq7.17}
     X := \Big(\prod_{k =1}^m X_k \times X_o \Big)^2 \times X_p
  \end{equation}
equipped with the following norm
   \begin{equation}\label{eq7.18}
       \|\vec{\varphi}\|_X = \sum_{j =1, 2} \Big(\sum_{k =1}^m \|\varphi_{jk}\|_{2 + \sigma, k} + \|\varphi_j^o\|_{b, o} \Big) + \|\mathbf P\|_p.
   \end{equation}
With (\ref{eq7.17}) and (\ref{eq7.18}), we shall show the existence of solution  $\vec{\varphi}$ given by (\ref{eq7.15}) to problem (\ref{eq7.10}) in $X$ via the fixed point theorem.  First of all, we claim that for $\|\vec{\varphi}\|_X < 1$, 
   \begin{equation}\label{eq7.19}
       \|\mathcal{A}(\vec{\varphi})\|_X < 1,
   \end{equation}
   where $\mathcal A$ is defined by (\ref{eq7.16}). 
For the inner operator, by using Lemma \ref{lemma31}, we find it is sufficient to prove 
  \begin{equation}\label{eq7.20}
  \|\varepsilon^{-\gamma_{1}}CF_{jk}(\boldsymbol{\varphi}, \mathbf{P})\|_{4+\sigma, k} < 1,
  \end{equation}
  which concludes $\|\varepsilon^{-\gamma_{1}}CF_{jk}(\boldsymbol{\varphi}, \mathbf{P})\|_{4+\sigma, k} < 1,$ where $C:=\max\{C_1,C_2\}$ given in \eqref{innerprori}.  To this end, we expand $\varepsilon^{-\gamma_{1}}F_{jk}(\boldsymbol{\varphi}, \mathbf{P})$ as 
 \begin{equation}\label{eq7.21}
 \begin{split}
    \varepsilon^{-\gamma_{1}}F_{jk}(\boldsymbol{\varphi}, \mathbf{P}) = &\varepsilon^{2 - \gamma_{1}}\sum_{l=1}^5I_{jl} \\
       & + \Big[\varepsilon^{2 - \gamma_{1}}I_{j7} - \varepsilon^{2 + \gamma_{2} - \gamma_{1}}\lambda_j \bar{u}_j \phi^o_j - \varepsilon^{\gamma_{2} -\gamma_{1}}\nabla\cdot (P_j \nabla w_j^o) + \varepsilon^{\gamma_{2} - \gamma_{1}} \varphi_j^o \Delta Q_j\Big] \\
       &+\bigg[\va^{2-\gamma_1}I_{j6}+\va^{1+\gamma_2-\gamma_1}\sum_{k=1}^m\nabla_y\cdot(\varphi_j^o\nabla_xH^{\va}_{jk}(\va y,\xi_{k}))\nonumber\\
  &-\sum_{k=1}^m\sum_{l\not=k}[\nabla\cdot(U_{jl}\nabla\bar\psi_{jk})+\nabla\cdot(\varphi_{jk}\nabla\Gamma_{jl})\eta_k]\bigg]\\
  &-\nabla\cdot (U_{jk}(y-\xi_j')\nabla(\bar w_{jk}-\bar w_{jk}'))\\
      : =&  II_1 +II_2 +II_3-II_4. 
    \end{split}
 \end{equation}
 For $II_1$ , we show that
 \begin{align}\label{II1small2025}
 \|II_1\|_{4 +\sigma, k} = o(1),
 \end{align}
and only discuss 
$\nabla_y U\cdot \nabla_y \Gamma$ since the others can be treat in the same way. For $y \in B_{\delta/2\varepsilon}(\xi'_k)$, we have for $r$ large, 
  \begin{equation}\label{eq7.22}
  \begin{split}
      \varepsilon^{- \gamma_{1}}&|(1 + |y - \xi'_k|)^{4 + \sigma} \nabla U_{jl}(y - \xi'_l) \cdot  \nabla \Gamma_{jk}(y - \xi'_k)| \\
       &  \le \varepsilon^{- \gamma_{1}}C \Big|(1 + |y - \xi'_k|)^{4 + \sigma}\frac{1}{(1 + |y - \xi'_l|)^{m_j+1}} \cdot \frac{1}{1 + |y - \xi'_k|} \Big| \le C \varepsilon^{m_j - 2 - \gamma_{1} - \sigma},
   \end{split}
  \end{equation}
  where $C>0$ is a large constant.
Taking $\sigma>0$ and $\gamma_1>0$ small enough such that $m_j - 2 - \gamma_{1} - \sigma > 0$ for $j=1,2$, we then obtain from \eqref{eq7.22} that
  \begin{equation*}
     \|\nabla U_{jl}(y - \xi'_l)\cdot \nabla \Gamma_{jk}(y - \xi'_k)\|_{4 + \sigma, k}  = o(1). 
  \end{equation*}
The leading terms in $II_2$ are generated by the outer solution.  We only analyze $\varepsilon^{\gamma_{2} - \gamma_{1}}\varphi^o_j\Delta Q_j$ as an example and proceed with the same argument on the others.   
Notice that 
   \begin{equation*}
   \begin{split}
      |\Delta Q_j| &\le C \sum_{k = 1}^{m}\big(b_{j1} e^{\Gamma_{1k}}+b_{j2}e^{\Gamma_{2k}}\big)  \le C \sum_{k =1}^m \frac{1}{(1 + |y - \xi'_k|)^{\min\{m_1,m_2\}}},
      \end{split}
   \end{equation*}
   where $C>0$ is a constant.
 Therefore, 
\begin{align}\label{eq7.24}
   \varepsilon^{\gamma_{2} - \gamma_{1}} |\varphi^o_j \Delta Q_j| 
   \le& C \varepsilon^{\gamma_{2} - \gamma_{1}} \|\varphi^o_j\|_{b, o}\sum_{k =1}^m \frac{1}{(1 + |y - \xi'_k|)^{\min\{m_1,m_2\}}}\cdot   \sum_{k =1}^m(1 + |y - \xi'_k|)^{-b},
\end{align}
where $C>0$ is a constant.  Since $\min\{m_1,m_2\}  > 2 + \sigma$, we take $b > 2$ and obtain $m_j + b > 4 + \sigma$ for $j=1,2$.  Assume that $\gamma_{2} > \gamma_{1}$, we then obtain from \eqref{eq7.24} that 
  \begin{equation}\label{eq7.25}
     \|\varepsilon^{\gamma_{1} - \gamma_{2}} \varphi^o_j \Delta Q_j\|_{4 + \sigma, k} = o(1). 
  \end{equation}
Thus, we have
    \begin{equation}\label{eq7.26}
       \|II_2\|_{4 + \sigma, k} = o(1).
    \end{equation}
For $II_4$, since 
  \begin{equation*}
     \bar w_{jk} = C + O\bigg(\frac{1}{|y|}\bigg), \ \ \text{and} \ \  \bar w'_{jk} = O(\varepsilon^\sigma) + C + O\bigg(\frac{1}{|y|}\bigg) \ \ \text{as} \ \  |y| \to \infty, 
  \end{equation*}
we immediately get that 
  \begin{equation}\label{eq7.27}
      \|II_4\|_{4 + \sigma, k} = o(1),
  \end{equation}
  where the algebraic decay properties of $U_{jk}$, $j=1,2$ were used.
 For $II_3$, we focus on the worse term $\va|\nabla\varphi_{jk}\cdot\nabla \tilde H^\varepsilon_{jk}(\varepsilon y, \xi_k')|
,$
where $\tilde H$ is defined by \eqref{eq2.26}.  The other terms in $II_3$ can be treated similarly as shown above.  We find
  \begin{equation*}
    \Big| \varepsilon \nabla\varphi_{jk} \cdot  \nabla \tilde H^\varepsilon_{jk}(\varepsilon y, \xi_k') \Big| \le \frac{C\delta}{(1 + |y - \xi'|)^{4 + \sigma}},
  \end{equation*}
where $C > 0$ is a constant.  Taking $\delta$ small enough, we then obtain that 
  \begin{equation*}
    \Vert  \varepsilon \nabla\varphi_{jk} \cdot  \nabla \tilde H^\varepsilon_{jk}(\varepsilon y, \xi_k')\Vert_{4+\sigma,k} < \frac{1}{4}.  
  \end{equation*}
  Thus, by taking $\delta<\bar C_1$ with $\bar C_1$ is a constant, then we have $\Vert II_3\Vert_{4+\sigma,k}<\frac{1}{2}.$
Combining this with \eqref{eq7.21}, \eqref{eq7.26} and \eqref{eq7.27}, one gets
 \begin{equation}\label{eq7.30}
     \|\mathcal{A}_{jk}(\vec{\varphi})\|_{2+ \sigma, k} \le  \varepsilon^{- \gamma_{1}}\|CF_{jk}(\boldsymbol{\varphi}, \mathbf{P})\|_{4 + \sigma, k} < 1,
  \end{equation}
which completes the proof of \eqref{eq7.20}.

We next prove that the outer operator satisfies
  \begin{equation}\label{eq7.31}
   \|\mathcal{A}_j^o(\vec{\varphi})\|_{b, o} < 1,~j=1,2,
  \end{equation}
  provided with $\Vert \vec{\varphi}\Vert_{X}<1.$
Invoking Lemma \ref{lemma33}, we have
  \begin{equation}\label{eq7.32}
     \|\mathcal{A}^o_j(\vec{\varphi})\|_{b, o} \le C \|\varepsilon^{-\gamma_{j2}} J_j(\boldsymbol{\varphi}, \mathbf P)\|_{b + 2, o}, 
  \end{equation}
where $C> 0$ given in (\ref{345}). Thus, it suffices to show that $C \|\varepsilon^{-\gamma_{j2}} J(\varphi_j, \mathbf P)\|_{b + 2, o} < 1$. Next, we are only concerned with the error terms generated by the inner solutions $(\varphi_{1k}, \varphi_{2k})$ since they are leading order ones.  In fact, we have for $j=1,2$, 
   \begin{equation*}
   \begin{split}
     \varepsilon^{\gamma_{1} - \gamma_{2}}|\nabla \varphi_{jk} \cdot \nabla \eta_k|&\le C \varepsilon^{\gamma_{1} - \gamma_{2}} \frac{1}{(1 + |y - \xi'_k|)^{4 + \sigma}}   \\
       &\le C \frac{\varepsilon^{\gamma_{2} - \gamma_{1}}}{\delta^{2(\gamma_{2}- \gamma_{1})}}\frac{1}{(1 + |y - \xi'_k|)^{4 + \sigma + 2\gamma_{1} - 2\gamma_{2}}}.
       \end{split}
   \end{equation*}
By choosing $\delta > \bar C_2 \sqrt{\varepsilon}$, $2\gamma_{1} - 2\gamma_{2} = -\frac{\sigma}{2}$ and $b = 2 + \frac{\sigma}{2}$, one gets
  \begin{equation*}
     \varepsilon^{\gamma_{1} - \gamma_{2}}\|\nabla \varphi_{jk} \cdot \nabla \eta_k\|_{b + 2, o} \le \sigma^*,  
  \end{equation*}
where $\sigma^*$ is a small constant and constant $\bar C_2=O(1)$ is chosen to guarantee the smallness of $\sigma^*$. 
 Proceeding with the other error terms generated by the inner solutions in a similar way, we can indeed show that $C \|\varepsilon^{-\gamma_{2}} J_j(\boldsymbol{\varphi}, \mathbf P)\|_{b + 2, o} < 1$, which implies \eqref{eq7.31} holds.

   In summary, we take $\sigma \in (0, 1)$ small enough, $\delta \in (\sqrt{\varepsilon}\bar C_2, \bar C_1)$, $b=2+\frac{\sigma}{2}$, $\gamma_{1}=\frac{\min\{m_1,m_2\} -2 - \sigma}{2}$ and $\gamma_{2} = \gamma_{1} + \frac{\sigma}{4}$ to get  \eqref{eq7.19} and \eqref{eq7.31} hold.

For the contraction properties of inner and outer operators, we perform the same arguments shown above to derive that there exist constants $\alpha_1, \alpha_2 \in (0, 1)$ such that 
\begin{equation}\label{inlight2025}
  \left\{
  \begin{aligned}
     \|\mathcal{A}_{jk}[\vec{\varphi}_1] - \mathcal{A}_{jk}[\vec{\varphi}_2]\|_{2 + \sigma, k} \le \alpha_1\|\vec{\varphi}_1 - \vec{\varphi}_2\|_X;   \\
     \|\mathcal{A}^o_{j}[\vec{\varphi}_1] - \mathcal{A}^o_{j}[\vec{\varphi}_2]\|_{2 + \sigma, k} \le  \alpha_2\|\vec{\varphi}_1 - \vec{\varphi}_2\|_X
  \end{aligned}
  \right.
\end{equation}
for any $\vec{\varphi}_1, \vec{\varphi}_2 \in X$ with $\|\vec{\varphi}_1\|_X, \|\vec{\varphi}_2\|_X < 1$, where $j =1, 2$ and $k =1, \cdots, m$ and the details are omitted.

It remains to study the position operator $\mathcal A_p.$ 
 As discussed in Section \ref{choiceansatz}, 
we have $\mathbf{P}_0$ is defined as 
    \begin{equation}
     \mathbf{P} = \mathbf{P}_0 + \mathbf{P}_1 \ \ \text{with} \ \ \mathbf{P}_0 = (c^0_{11}, \cdots, c^0_{1m}, c^0_{21}, \cdots , c^0_{2m}, \xi^0_1, \cdots, \xi^0_m), 
    \end{equation}
 where $(\xi^0_1, \cdots, \xi^0_m)$ is the critical point of $\mathcal{J}_m$.  Next, we adjust $\mathbf{P}_1$ to guarantee the orthogonality conditions shown in \eqref{eq7.12} hold, i.e., $m_{jlk}[h] = 0$ with $l=1,2,3$.  We shall show that $\mathbf{P}_1$ is $o(1)$, which immediately implies that $\|\mathcal{A}_p\|_p$ is a contraction mapping and satisfies $\|\mathcal{A}_p\|_p < 1$.


We first consider the multi-interior spots case.   Focusing on the mass condition and the $k$-th inner region, we find that the leading one in the right hand side of (\ref{eq7.10}) is $\int_{\Omega_{\va}}f(U_{jk})\eta_kdy$ for $j=1,2$ with $f_j(u)=u(\bar u_j-u)$.  Then we calculate to find 
\begin{align*}
\int_{\Omega_{\va}}f(U_{jk})\eta_k dy
=&\int_{\Omega_{\va}}c_{jk}e^{\Gamma_{j,\mu_{jk}}}\bigg(\bar u_j-c_{jk} e^{\Gamma_{j,\mu_{jk}}}\bigg)\eta_kdy+O(\va^2)\\
=&\int_{B_{2\delta/\va}(\xi_k')}(c^0_{jk}+ c^1_{jk})e^{\Gamma_{j,\mu_{jk}}}\bigg[\bar u_j-(c^0_{jk}+c^1_{jk})e^{\Gamma_{j,\mu_{jk}}}\bigg] dy+O(\va^2)\\
=&\int_{\R^2}c_{jk}^0e^{\Gamma_{j,\mu_{jk}}}\big(\bar u_j-c_{jk}^0e^{\Gamma_{j,\mu_{jk}}}\big) dy\\
&+\int_{B_{2\delta/\va}(\xi_k')}(c^0_{jk}+ c^1_{jk})e^{\Gamma_{j,\mu_{jk}}}\bigg[\bar u_j-(c^0_{jk}+c^1_{jk})e^{\Gamma_{j,\mu_{jk}}}\bigg] dy-\int_{\R^2}c_{jk}^0e^{\Gamma_{j,\mu_{jk}}}\big(\bar u_j-c_{jk}^0e^{\Gamma_{j,\mu_{jk}}}\big) dy\\
&+O(\va^2)\\
=&\int_{\R^2}c_{jk}^0e^{\Gamma_{j,\mu_{jk}}}\big(\bar u_j-c_{jk}^0e^{\Gamma_{j,\mu_{jk}}}\big) dy+O(1)c_{jk}^1+O(\va^{2}),
\end{align*}
where $c_{jk}^0$ is chosen such that
$$\int_{\R^2}c_{jk}^0e^{\Gamma_{j,\mu_{jk}}}\big(\bar u_j-c_{jk}^0e^{\Gamma_{j,\mu_{jk}}}\big) dy=0.$$
Now, we take $c_{jk}^0=O(\va^2)$ to guarantee that the mass condition holds.

Next, we verify the first-moment orthogonality condition.  In fact, the leading term is $\sum\limits_{n=1}^m\sum\limits_{k=1}^m\nabla_x\cdot(U_{jk} \nabla_x (\Gamma_{jn}+  H^{\va}_{jn}))-\nabla_x\cdot(U_{jk}\nabla\Gamma_{jk})$.  We expand it and obtain for $\iota=1,2$,
\begin{equation}\label{45777}
\begin{split}
    &\sum_{n =1}^m\sum_{k =1}^m \va\int_{\Omega_\va}\nabla_y\cdot\Big(U_{jk}(y - \xi'_k)  \nabla  H^{\va}_{jn}(\va y, \xi'_n)\Big)\big(y - \xi_k'\big)_{\iota} \eta_k(y)\d y\\
    &+\sum\limits_{n=1}\sum\limits_{k\not=n}\int_{\Omega_{\va}}\nabla_y\cdot(U_{jk}(y-\xi'_k) \nabla_y( \Gamma_{jn}(y-\xi'_n)+H_{jn}^{\va}))(y-\xi_k')_{\iota}\eta_k\d y \\
     =& \sum_{n =1}^m\sum_{k=1}^m \va\int_{\Omega_\va}U_{jk} \nabla  H^{\va}_{jn} \cdot \boldsymbol{e}_{\iota} \eta_k(y) \d y+\sum_{n =1}^m\sum_{k\not=n}^m \int_{\Omega_\va}U_{jk}\nabla (\Gamma_{jn}+H_{jn}^{\va})\cdot \boldsymbol{e}_{\iota} \eta_k(y) d y \\
      + \sum_{n =1}^m\sum_{k =1}^m& \va\int_{\Omega_\va}U_{jk} \nabla  H^{\va}_{jn} \cdot  (y - \xi_k')_{\iota} \nabla \eta_k(y)d y+ \sum_{n =1}^m\sum_{k \not=n}^m \int_{\Omega_\va}U_{jk} \nabla (\Gamma_{jn}+H_{jn}^{\va}) \cdot  (y - \xi_k')_{\iota} \nabla \eta_k(y) d y \\
     &:= \bar{III}_A + \bar{III}_B.
    \end{split}
\end{equation}
We estimate $\bar{III}_A$ and $\bar{III}_B$ term by term given in (\ref{45777}). 
For $\bar{III}_B$, we have from the decay property of $U_{jk}$ and $\nabla\Gamma_{jn}$ that $|\bar{III}_B| = O(\va^{m_j})$ with $m_j>2$. For $\bar{III}_A$, similarly as shown in (\ref{nablatildeHjva2025117}), we expand
\begin{align*}
    \nabla H^{\va}_{jn}(\va y, \xi_k') \cdot \boldsymbol{e}_{\iota} 
    &= \hat c_{jn}\p_{x_{\iota}}  H(\xi_k, \xi_k) + \va\hat c_{jn} (\p^2_{x_{\iota}}  H)(\xi_k, \xi_k)\cdot (y - \xi_k') + O(\va^{\alpha}),~\iota=1,2.
\end{align*}
where $\alpha\in(0,1).$  Similarly, we have
\begin{align*}
\nabla(\Gamma_{jn}+H_{jn}^{\va})(\va y,\xi_k')\cdot \boldsymbol{e}_{\iota}=\hat c_{jn}\partial_{x_{\iota}}G(\xi_k,\xi_k)+\va \hat c_{jn}\p^2_{x_{\iota}}G(\xi_k,\xi_k)(y-\xi'_k)+O(\va^{\alpha}),
\end{align*}
where $G$ satisfies (\ref{Greenintro}).
Upon substituting the two expansions into $\bar{III}_A$ in (\ref{45777}), we further obtain that
 \begin{equation}\label{461}
 \begin{split}
  \bar{III}_A 
    & = \va \pi^2\sigma^2_jc_{j0} \p_{x_{\iota}} \bigg(\sum\limits_{j=1}^m \bar c_{k}^2H(\xi_k,\xi_k)+\sum_{j\not=l} \bar c_k\bar c_lG(\xi_k,\xi_l)\bigg)  + O(\va^{\alpha+1})\\
    &= \va \pi^2\sigma^2_jc_{j0} \p_{x_{\iota}}\mathcal J_m+O(\va^{\alpha+1}),
  \end{split}
 \end{equation}
where $\mathcal J_m$ is defined by (\ref{jm}), $\sigma_j$ and $c_{j0}$ are given in (\ref{m1m2decayrate}) and (\ref{leadingci0}), respectively.  Noting that $(\xi_{1}^0,\cdots,\xi_{m}^0)$ is a $m$-tuple critical point of $\mathcal J_m$ and $\mathcal J_m$ has a non-degenerate property, we further expand 
$$\p_{x_{\iota}}  H(\xi_k, \xi_k)=\partial_{x_{\iota}} H(\xi^0_{k},\xi^0_{k})+\partial^2_{x_{\iota}} H(\xi_{k}^0,\xi_{k}^0)\xi^{\iota}_{k1}+O(\vert\xi_{k1}\vert^2), $$
and
$$\p_{x_{\iota}}  G(\xi_k, \xi_k)=\partial_{x_{\iota}} G(\xi^0_{k},\xi^0_{k})+\partial^2_{x_{\iota}} G(\xi_{k}^0,\xi_{k}^0)\xi^{\iota}_{k1}+O(\vert\xi_{k1}\vert^2).$$
Upon substituting them into \eqref{461}, we apply the fact that $(\xi_{1}^0,\cdots,\xi_{m}^0)$ is a critical point of $\mathcal J_m$ to conclude  
$$\xi_{k1}^{\iota}=O(\va^{\bar\alpha}),~~\iota=1,2,~k=1,\cdots,m,$$
where we have used that $|\bar{III}_B|=o(\va^{m_j})$ $0<\bar\alpha<1$ and $\bar \alpha\approx 1.$
Since $\xi_{k0}$ is the critical point of $\bar H_k$ and $\mathcal J_m$ has the non-degenerate property, $\bar{III}_A$ can be written as
\begin{align*}
\bar {III}_A=\va \xi^{\iota}_{k1} \partial^2_{x_{\iota}}\bar H_k(\xi_{k0}, \xi_{k0})  + O(\va^2)+O(\vert\xi_{k1}\vert^2) \ \ \text{for} \ \   \iota = 1,2.
\end{align*}
We remark that it is straightforward to verify the other terms, e.g. $\sum\limits_{n=1}\sum\limits_{k=1}\nabla_x\cdot (\phi_{jk} \nabla_x (\Gamma_{jn}+H_n))-\nabla_x\cdot(\phi_{jk}\nabla\Gamma_{jk})$, in the divergence form operator are negligible. 
Now, we complete the proof of our claim that when $\Vert \phi\Vert_{X}<1$, $\Vert \mathcal A(\phi)\Vert_{X}<1$ and $\mathcal A_p(\vec{\phi})$ is a contraction mapping.

Define $\mathcal B$ as 
$$\mathcal B=\{\varphi\in X:\Vert \varphi\Vert_{X}<1\}.$$
Thanks to \eqref{inlight2025}, we have
$$\mathcal A(\mathcal B)\subset \mathcal B,~\text{and }\Vert \mathcal A({\vec{\varphi_1}})-\mathcal A({\vec{\varphi_2}})\Vert_{X}\leq\frac{2}{3}\Vert \vec{\varphi_1}-\vec{\varphi_2}\Vert_{X},~~\forall \varphi_1,\varphi_2\in\mathcal B.$$
It follows that there exists a solution such that $\vec{\varphi}=\mathcal A{\vec{\varphi}}.$
Now, we constructed the multi-interior spots rigorously and next focus on the multi-boundary spots.

\subsection{Construction of Boundary Spots }
In this subsection, we are concerned with the existence of multi-boundary spots.  Firstly, we introduce the transformation to straighten the boundary. 
Define the graph $\rho(x_1)$ as $(x_1, x_2) = (x_1, \rho(x_1))$ with $\rho(0) = \rho'(x)$, then for $k =1, \cdots,  m$, we transform $(y_1, y_2)$ into 
    \begin{equation}
    z_{1, k} = y_1- \xi'_{k, 1}, \ \ z_{2, k} = y_2 - \xi'_{k, 2} - \frac{1}{\varepsilon}\rho(\varepsilon (y - \xi'_{k, 1})),
    \end{equation}
where $y_1 = x_1/\varepsilon$ and $y_2 = x_2/\varepsilon$.  Moreover, we denote operator $P_{\rho, \xi_k'}$ such that for any function $\omega(y_1, y_2)$, 
  \begin{equation}
      P_{\rho, \xi'_k} \omega(y_1, y_2) = \omega(z_{k, 1}, z_{k, 2}).
  \end{equation}
Under the transformation shown above, we have the Laplacian operator and Neumann boundary operator become 
\begin{equation}
   \left\{
   \begin{aligned}
      &\Delta_y \omega = \Delta_{z, k} \omega + (\rho'(\varepsilon z_{1,k }))^2\partial^2_{z_{2,k}}\omega - 2\rho'(\varepsilon z_{1,k})\partial_{z_{1,k}, z_{2,k}}\omega - \varepsilon \rho''(\varepsilon z_{1,k})\partial_{z_{2,k}}\omega, \\
     & \sqrt{1 + (\rho'(\varepsilon z_{1,k}))^2} \frac{\partial \omega}{\partial \boldsymbol{\nu}} = \rho'(\varepsilon z_{1,k}) \partial_{z_{1,k}}\omega - [1 + (\rho'(\varepsilon z_{1,k}))^2]\partial_{z_{2,k}}\omega.
   \end{aligned}
   \right.
\end{equation}
  Without confusing the reader, we replace $(z_{k, 1}, z_{k, 2})$ by $(z_1,z_2)$ for the simplicity of notations, then obtain
    \begin{equation}\label{53720251}
        \Delta_y \omega = \Delta_z \omega + (\rho''(0))^2\varepsilon^2 z_1^2 \partial_{z_2z_2}\omega - 2 \rho''(0)\varepsilon z_1 \partial_{z_1z_2}\omega - \varepsilon \rho''(0)\partial_{z_2}\omega +O(\varepsilon^2)
    \end{equation}
and 
\begin{equation}\label{53820251}
\begin{split}
    \nabla_y \omega_1 \cdot \nabla_y \omega_2 =& \nabla_z \omega_1 \cdot \nabla_z \omega_2 + \frac{\partial \omega_1}{\partial z_2} \cdot \frac{\p \omega_2}{\partial z_2}(\rho''(0))^2 \varepsilon^2 z_1^2  \\
       & \ \  - \Big(\frac{\partial \omega_1}{\partial z_1} \cdot\frac{\partial \omega_2}{\partial z_2}  + \frac{\partial \omega_1}{\partial z_2} \cdot\frac{\partial \omega_2}{\partial z_1}  \Big)\rho''(0)\varepsilon z_1 +O(\varepsilon^2),
       \end{split}
\end{equation}
where we have used the following expansions of $\rho$ and $\rho'$ 
  \begin{equation}
     \rho(\varepsilon z_1) = \frac{1}{2}\rho''(0)\varepsilon^2 z_1^2 + O(\varepsilon^2)  \ \ \text{and} \ \  \rho'(\varepsilon z_1) = 
     \rho''(0)\varepsilon z_1 + O(\varepsilon^3).
  \end{equation} 
Due to the presence of extra terms in (\ref{53720251}) and (\ref{53820251}), we expect that there exist many new terms in the error generated by the ansatz of boundary spots compared to interior ones. Whereas, we shall show they are all higher order terms and enjoy good decay estimates while performing the fixed point argument.

Before formulating the inner problem in the half space $\R^2_{+}$, we introduce the following cut-off function $\eta_H$ 
  \begin{equation}\label{etaHK2025118}
     \eta_{H, k}(z) = 1 \ \ \text{for} \ \  z \in {\bar{\mathbb{R}}}^2_+ \cap \bar B_{\delta/\varepsilon}(\xi_k')  \ \ \text{and} \ \   \eta_{H, k} = 0 \ \ \text{for} \ \  z \in \mathbb{R}^2 \cap B^c_{2\delta/\varepsilon}(\xi_k'). 
     \end{equation}
    Invoking (\ref{etaHK2025118}), (\ref{53720251}) and (\ref{53820251}), we define the new error function as
\begin{equation} \label{withaidnew1}
\begin{split}
   N^{\rho}_{jk} = & (\rho'(\varepsilon z_1))^2\Big[\frac{\partial^2(\varphi_{H, jk}\eta_{H, k})}{\partial z_2^2} - \Big(\frac{\partial U_{jk}}{\partial z_2} \cdot \frac{\partial \bar w_{H, jk}}{\partial z_2} + \frac{\partial \bar w_{H, jk}}{\partial z^2_2}\cdot U_{jk} \\
        &+ \frac{\partial(w_{H,jk}\eta_{H,k})}{\partial z_2}\cdot \frac{\partial \Gamma_{jk}}{\partial z_2} + \frac{\partial^2 \Gamma_{jk}}{\partial z_2^2} \cdot (\varphi_{H, jk}\eta_{H, k} ) \Big) \Big]  \\
     &  - (\rho'(\varepsilon z_1))\Big[\frac{\partial(\varphi_{H, jk} \eta_{H, k})}{\partial z_1\partial z_2} - \Big(\frac{\partial U_{jk}}{\partial z_1} \cdot \frac{\partial\bar  w_{H, jk}}{\partial z_2}  + \frac{\partial U_{jk}}{\partial z_2} \cdot \frac{\partial \bar w_{H, jk}}{\partial z_1}  \Big)
        - \frac{\partial^2 \bar w_{H, jk}}{\partial z_1 \partial z_2} \cdot U_{jk} \\
       &-\Big(\frac{\partial (\varphi_{H, jk} \eta_{H, k})}{\partial z_1}\cdot \frac{\partial \Gamma_{jk}}{\partial z_2}  + \frac{\partial \Gamma_{jk}}{\partial z_2} \cdot  \frac{\partial(\varphi_{H, jk}\eta_{H, k})}{\partial z_1}\Big) - \frac{\partial^2 \Gamma_{jk}}{\partial z_1 \partial z_2} (\varphi_{H, jk}\eta_{H, k}) \Big] \\
       & - \varepsilon \rho''(\varepsilon z_1) \Big[\frac{\partial (\varphi_{H, jk}\eta_{H, k})}{\partial z_2} - U_{jk}\frac{\partial \bar w_{H, jk}}{\partial z_2} - (\varphi_{H, jk}\eta_{H,k} )\frac{\p \Gamma_{jk}}{\partial z_2} \Big],
   \end{split}
\end{equation}
where $\bar w_{H, jk} = (-\Delta+\va^2)^{-1}[(a_{j1}\varphi_{H, 1k} + a_{j2}\varphi_{H,2k})\eta_{H,k}]$. Define $\mathbf{\hat P}_1$ and $\mathbf{\hat P}_2$ as the first and second coordinates of $\xi$, then we set the parameter vector $\mathbf{P}_H$ as 
  \begin{equation}\label{withaid2}
    \mathbf{P}_H = (\mathbf{c}, \mathbf{\hat P}_{H_1}, \mathbf{\hat P}_{H_2}) = (\mathbf{c}, \mathbf{\hat P}_1, \mathbf{\hat P}_2 - \rho(\mathbf{\hat P}_1)).
  \end{equation}
With the aid of (\ref{withaidnew1}) and (\ref{withaid2}), we find 
the inner equation given in (\ref{eq7.9}) becomes 
  \begin{equation}\label{eq543new2025}
      L^{\text{inn}}_{jk}[\varphi_{1k},\varphi_{2k}] = \Big(N^{\rho}_{jk} + \varepsilon^{- \gamma_{1}} F_{jk}(P_{\rho, \xi'_k}(\boldsymbol{\varphi}), \mathbf{P}_H )\Big)\eta_{H,k} 
        := F_{H, jk} (P_{\rho, \xi'_k}(\boldsymbol{\varphi}), \mathbf{P}_H)\eta_{H, k},
        \end{equation}
where $(z_1,z_2) \in \mathbb{R}^2_+$ and $F_{jk}(\boldsymbol{\varphi}, \mathbf{P})$ is given by (\ref{FJK2025116}). 

Before estimating $F_{H,jk}$ in (\ref{eq543new2025}), we define $\xi'_{H, k} = (\xi'_{k, 1}, \xi'_{k, 2} - \frac{1}{\varepsilon} \rho(\varepsilon \xi'_{k, 1}))$ and the inner norm in the half space as
  \begin{equation}
      \|h\|_{\nu, H, k}:= \sup_{z \in \mathbb{R}_+^2}|h|(1 + |z|)^\nu.   
  \end{equation}
Moreover, we denote space $X_{k, H}$, $X_{o, H}$ and $X_{p, H}$ the same as $X_k$, $X_o$ and $X_p$ except that $\mathbb{R}^2$ and $\|\cdot\|_{2 + \sigma, k}$ are replaced by $\R^2_{+}$ and $\Vert \cdot\Vert_{2+\sigma,H,k}.$ In addition, we define the norm and inner solution for boundary spots as 
$\|\cdot\|_X$ and $\vec{\varphi}_{H,k}$.

Next, we discuss the new error $N^{\rho}_{jk}$, where the straighten operator $P_{\rho,\xi_{k}'}$ is involved.      As shown in  (\ref{withaidnew1}), the worse term is
  \begin{equation}
  \begin{split}
   &(\rho'(\varepsilon z_1))^2 \frac{\partial^2(\varphi_{H, jk}\eta_{H, k})}{\partial z_2^2} -(\rho'(\varepsilon z_1))\frac{\partial^2(\varphi_{H, jk} \eta_{H, k})}{\p z_1\partial z_2} - \varepsilon \rho''(\varepsilon z_1) \frac{\partial (\varphi_{H, jk}\eta_{H, k})}{\partial z_2}\\
    & = (\rho''(0))^2(\varepsilon z_1)^2\frac{\partial^2(\varphi_{H, jk}\eta_{H, k})}{\partial z_2^2} - \rho''(0)\varepsilon z_1\frac{\partial^2(\varphi_{H, jk} \eta_{H, k})}{\p z_1\partial z_2}- \varepsilon \rho''(0)  \frac{\partial (\varphi_{H, ik}\eta_{H, k})}{\partial z_2}+O(\va^2).
    \end{split}
  \end{equation}
 Since $|z| < \delta$ for some constant $\delta > 0$, we can chose $\delta> 0$ small enough such that 
   \begin{equation}
      \Big\|(\rho'(\varepsilon z_1))^2 \frac{\partial^2(\varphi_{H, jk}\eta_{H, k})}{\partial z_2^2} -\rho'(\varepsilon z_1)\frac{\partial^2(\varphi_{H, jk} \eta_{H, k})}{\p z_1\partial z_2} - \varepsilon \rho''(\varepsilon z_1) \frac{\partial (\varphi_{H, jk}\eta_{H, k})}{\partial z_2} \Big\|_{4 + \sigma, H, k} < \sigma_1, 
   \end{equation}
where $\sigma_1>0$ is a small constant. For the other terms in $N^{\rho}_{jk}$, we analyze in a similar way and show that 
  \begin{equation}
     \|N^{\rho}_{jk} \eta_{H, k}\|_{4 + \sigma, H, k} \le \sigma_2,
  \end{equation}
where $\sigma_2>0$ is a small constant. 

 For the contraction property of $\mathcal A_{H}(\vec{\varphi_{H}})$, we follow the argument shown in Section 5 of \cite{kong2022existence} and obtain the desired conclusion. 
          Now, we shall check the orthogonality condition exhibited in Lemma \ref{lemma32}, which is equivalent to study $\mathcal A_{p,H}(\vec{\phi}_H)$. 
It remains to check orthogonality conditions shown in \eqref{331}. Noting that the leading term in $F_{H,{jk}}$ is $\int_{\R_2^+}f(U)\eta_{H,k}dz$, we perform the similar argument shown in Subsection \ref{sub51} to get
\begin{align}\label{height1bd}
\va^2\int_{\R^2_+}f(U)\eta_{H,k} dz=O(\va^2)c_{jk1}+O(\va^4),
\end{align}
where $c_{jk1}$ is the error of $c_{jk}=c_{jk}^0+c_{jk1}$ with $c_{jk}^0$ given by (\ref{leadingci0}). 

For the first-moment orthogonality given in (\ref{331}), we follow the same procedure shown in Section 5 of \cite{kong2022existence} to derive $\xi_{H,jk1}=O(\va^{\bar\alpha})$ with $\bar\alpha<1$ but close to $1$, where the details are omitted.  we point out that $\xi_{H,jk2}\equiv 0$ since the centre of boundary spot is located at the boundary.
Since $\mathbf {P}_H=o(1)$, we have $\mathcal A_p(\vec{\phi})$ is a contraction mapping.  Then by following the same argument shown in the end of Subsection \ref{sub51}, we find the existence of the remainder term $(\varphi_{1,H},\varphi_{2,H})$ to boundary spots.  This completes the proof of Theorem \ref{thm11}.

\section{Numerical Studies and Discussion}
This section is devoted to the numerical simulation for the emergence of spot patterns in system (\ref{timedependent}).  The error threshold is set as $\epsilon=0.01$ and the maximal time is $t=2000.$

Figure \ref{fig:singlespike} presents the numerical profile of single boundary spot solution we constructed in Theorem \ref{thm11}, which is a snapshot with $t\approx 2000$ in the spatial-temporal dynamics of (\ref{timedependent}).  Here the chemotactic coefficients $\chi_1=\chi_2$ are large enough.  As shown in the figures, $(u_1,u_2)$, located at a corner of the rectangle $\Omega,$ enjoys the fast decay property in the ``far-from" corner region; while $(v_1,v_2)$ follows the shape of $(u_1,u_2)$ mildly but has a global structure, also is positively bounded from below.  These characteristics well match our theoretical analysis shown above.  We remark that the profiles of $u_1$ and $u_2$ are distinct due to the variation of matrix coefficients given in $A=(a_{ij})_{i,j=1,2}.$   In addition, the locations of $u_1$ and $u_2$ are the same since all coefficients in $A$ are positive.
\begin{figure}[h!]
\centering
\begin{subfigure}[t]{0.5\textwidth}
    \includegraphics[width=\linewidth]{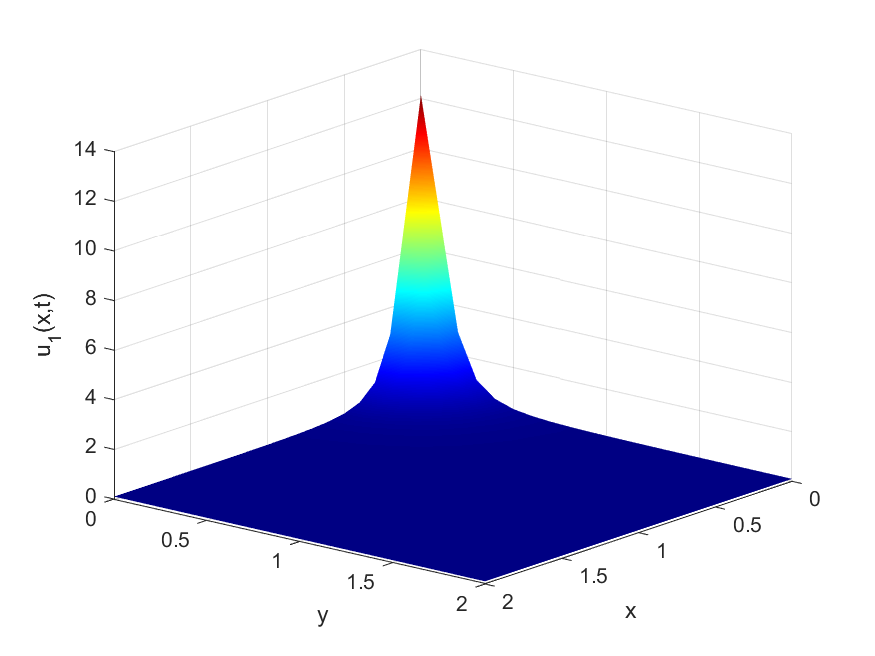}
\end{subfigure}\hspace{-0.25in}
\begin{subfigure}[t]{0.5\textwidth}
  \includegraphics[width=\linewidth]{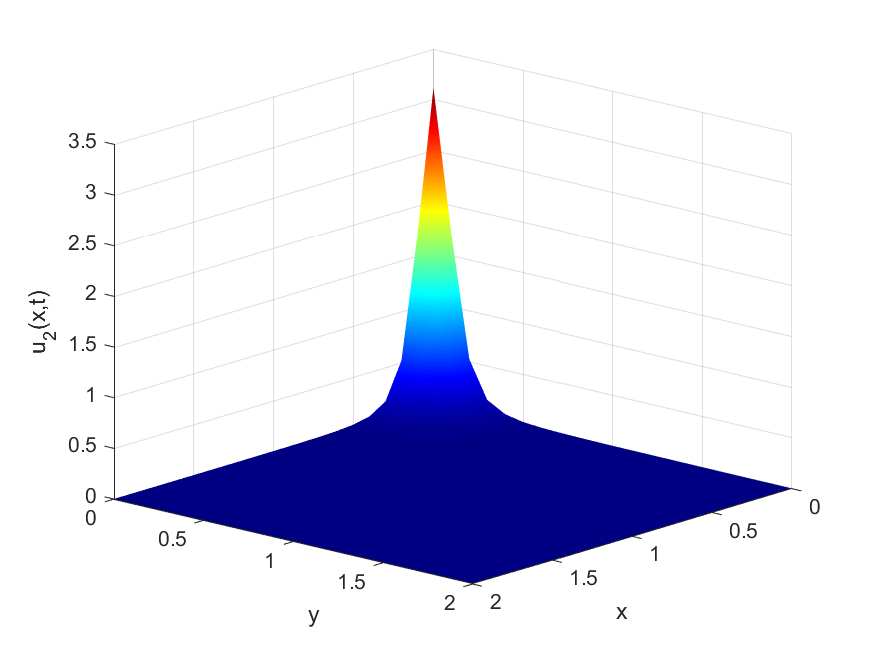}
\end{subfigure}

\begin{subfigure}[t]{0.5\textwidth}
    \includegraphics[width=\linewidth]{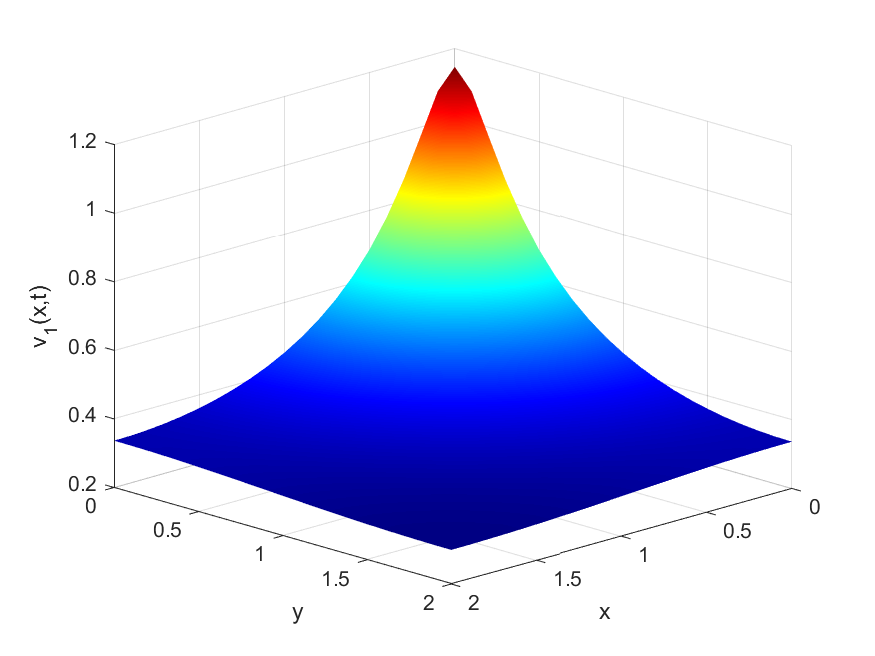}
\end{subfigure}\hspace{-0.25in}
\begin{subfigure}[t]{0.5\textwidth}
  \includegraphics[width=\linewidth]{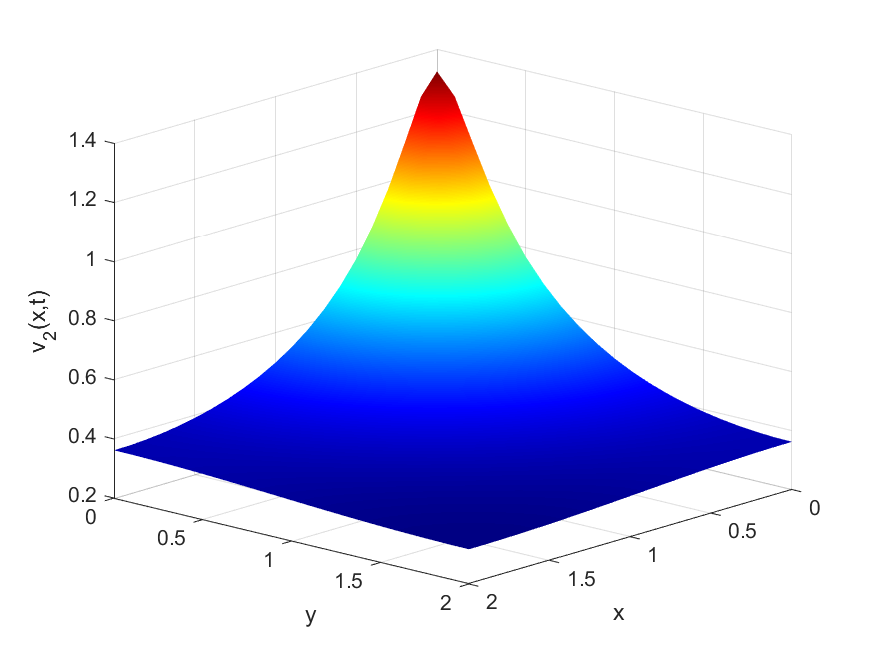}
\end{subfigure}
\\ 
\caption{The numerical profile of a single boundary spot steady state obtained by using FLEXPDE7\cite{flex2021} to (\ref{timedependent}) with $\Omega=(0,2)\times (0,2),$ where the rest parameters are set as $\chi_1=\chi_2=8.5$, $\lambda_1=\lambda_2=0.5$, $\bar u_1=2$, $\bar u_2=1,$  $a_{11}=2$, $a_{12}=1$, $a_{21}=2$ and $a_{22}=3.$  Here the initial data are chosen as $u_{10}=u_{20}=6e^{-10(x^2+y^2)}+0.1$ and $v_{10}=v_{20}=2e^{-10(x^2+y^2)}+0.1$.  The numerical solution is captured by approximating the time-dependent system (\ref{timedependent}) with \( t = 2000\).}
\label{fig:singlespike}
\vspace{-0.0in}\end{figure}

Figure \ref{fig:singlespike2} illustrates that the general form of (\ref{timedependent}) may admit the stable single interior spots in other regimes, which are different from the large chemotactic movement.  In addition, the half profiles of solutions shown in Figure \ref{fig:singlespike2} are non-monotone with respect to radius $r.$

\begin{figure}[h!]
\centering
\begin{subfigure}[t]{0.5\textwidth}
    \includegraphics[width=\linewidth]{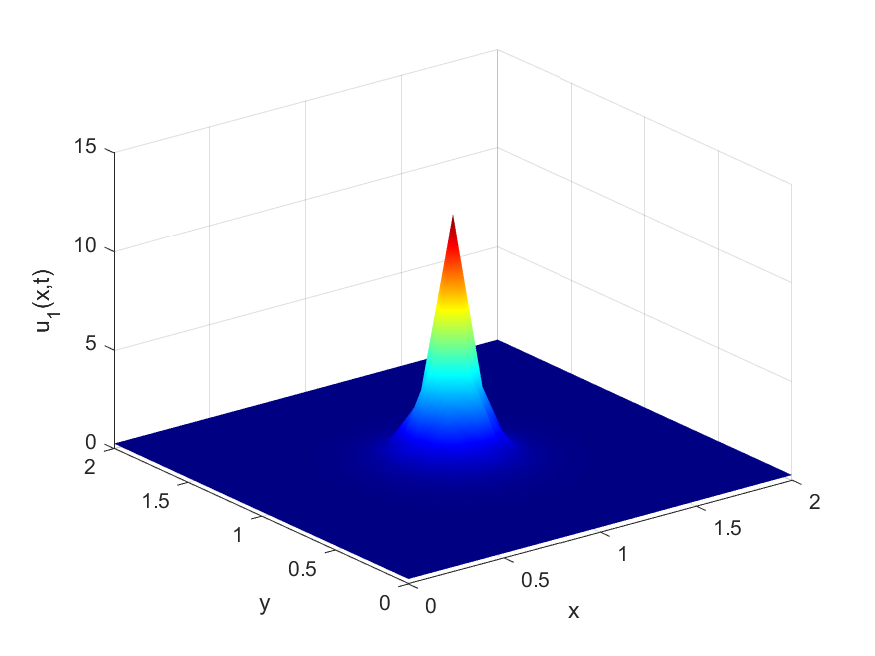}
\end{subfigure}\hspace{-0.25in}
\begin{subfigure}[t]{0.5\textwidth}
  \includegraphics[width=\linewidth]{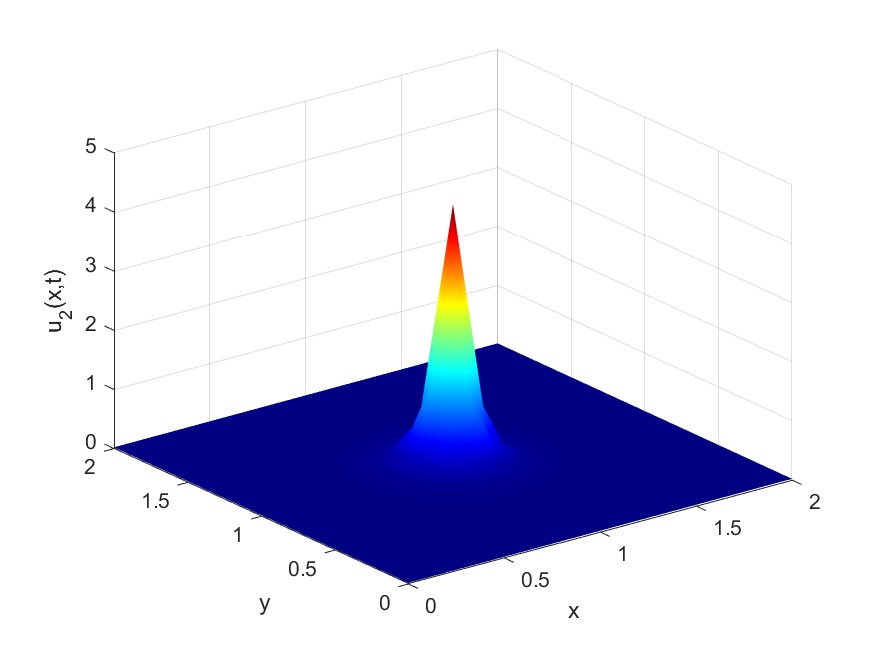}
\end{subfigure}

\begin{subfigure}[t]{0.5\textwidth}
    \includegraphics[width=\linewidth]{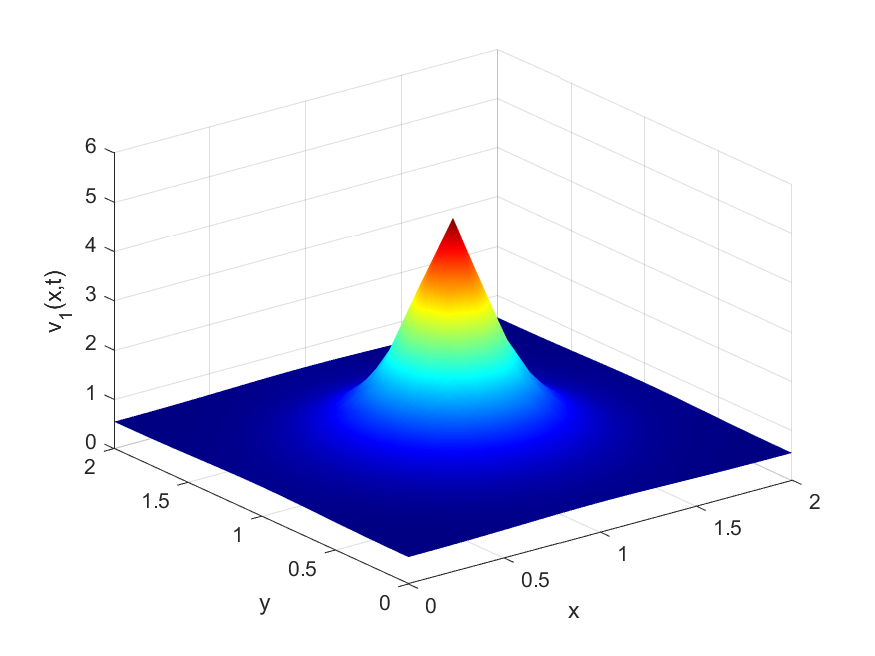}
\end{subfigure}\hspace{-0.25in}
\begin{subfigure}[t]{0.5\textwidth}
  \includegraphics[width=\linewidth]{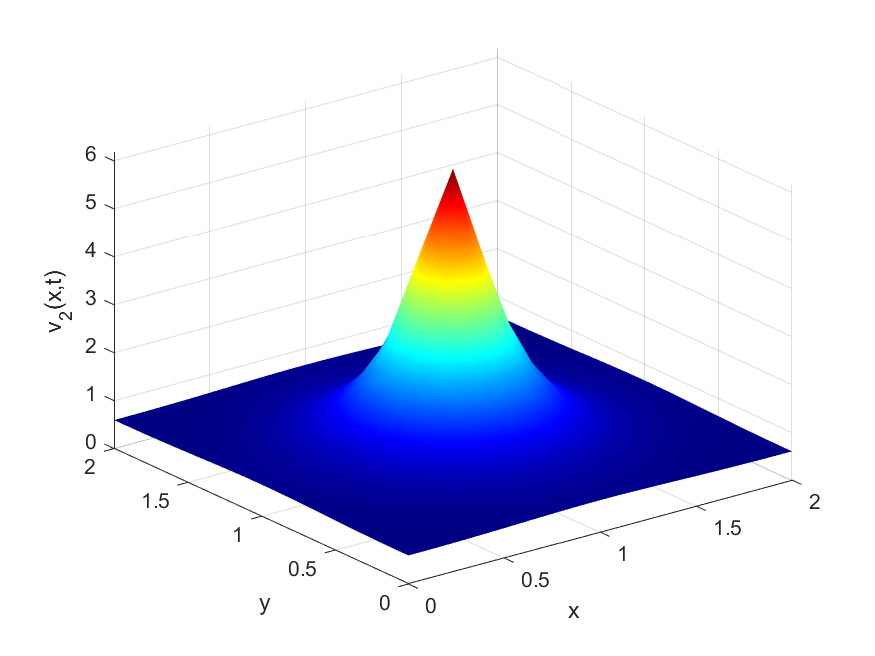}
\end{subfigure}
\\ 
\caption{The numerical profile of a single interior spot to (\ref{ss}) with $\Omega=(0,2)\times (0,2),$ where the rest parameters are the same except $\chi_1=\chi_2=1$ and $d_{v1}=d_{v2}=0.05$.  Here we incorporate chemical self-diffusion rates $d_{v1}$ and $d_{v2}$ in the $v_1$-equation and $v_2$-equation of (\ref{ss}).  In particular, the initial data are chosen as $u_{10}=u_{20}=6e^{-10[(x-1)^2+(y-1)^2]}+0.1$ and $v_{10}=v_{20}=2e^{-10[(x-1)^2+(y-1)^2]}+0.1$.  The numerical results suggest that the single interior spot in (\ref{ss}) is locally stable.}
\label{fig:singlespike2}
\vspace{-0.0in}\end{figure}

Figure \ref{fig:repulsive} demonstrates that when the coefficient matrix $A$ in (\ref{ss}) is not positive, (\ref{ss}) admits the spot steady states for sufficient large $\chi_1$ and $\chi_2$, in which cellular densities $u_1$ and $u_2$ are located at different points in $\bar\Omega.$  Also, the facts $a_{11}\not=a_{21}$ and $a_{12}\not=a_{22}$ in matrix $A$ trigger the formation of spots, where $u_1$ and $u_2$ do not share the same localized structure.

\begin{figure}[h!]
\centering
\begin{subfigure}[t]{0.5\textwidth}
    \includegraphics[width=\linewidth]{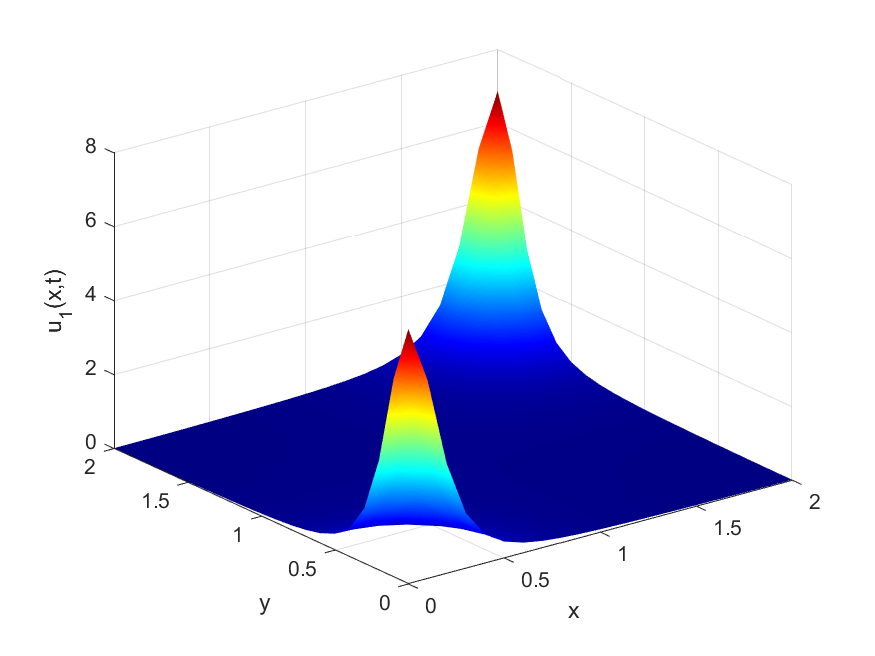}
\end{subfigure}\hspace{-0.25in}
\begin{subfigure}[t]{0.5\textwidth}
  \includegraphics[width=\linewidth]{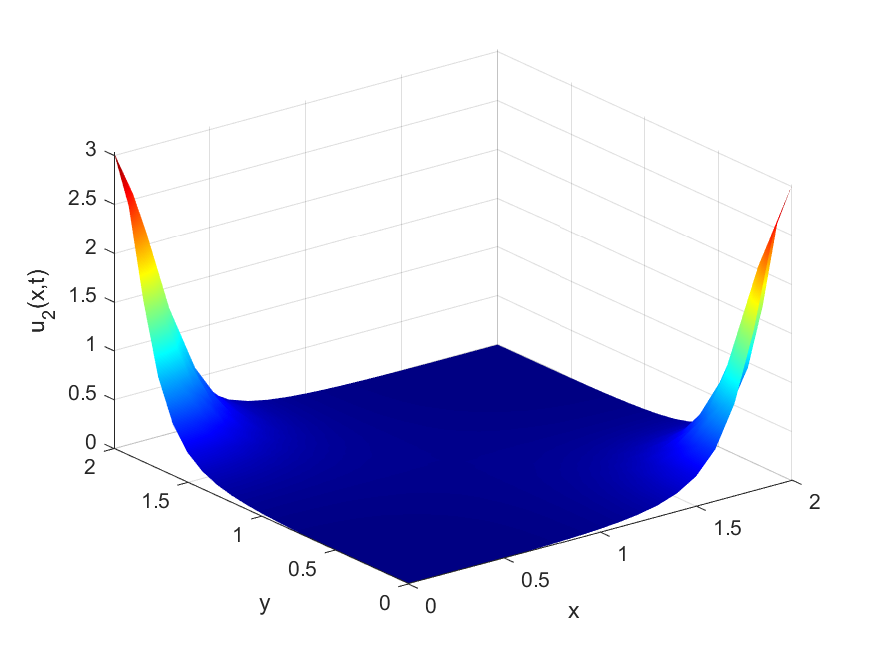}
\end{subfigure}

\begin{subfigure}[t]{0.5\textwidth}
    \includegraphics[width=\linewidth]{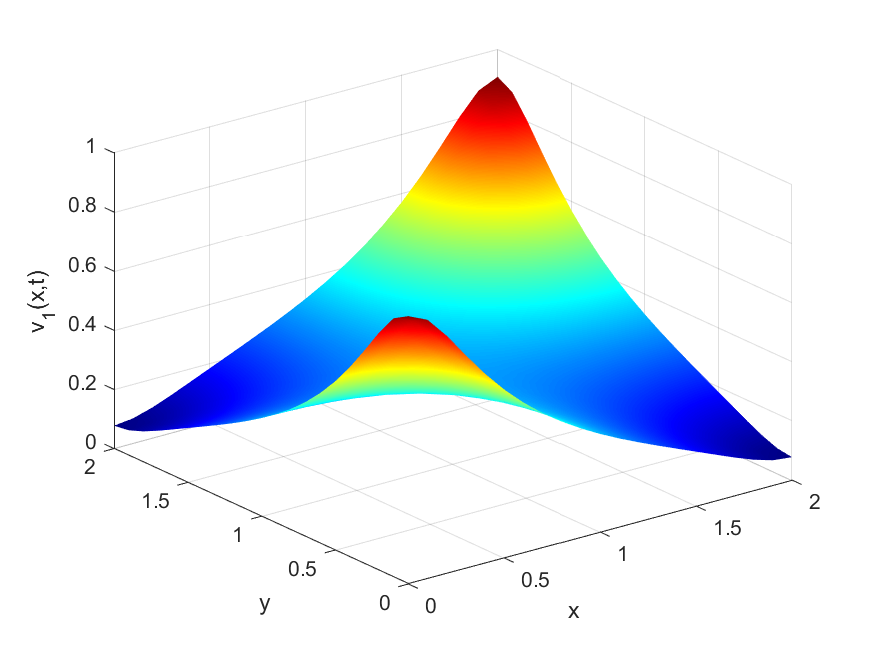}
\end{subfigure}\hspace{-0.25in}
\begin{subfigure}[t]{0.5\textwidth}
  \includegraphics[width=\linewidth]{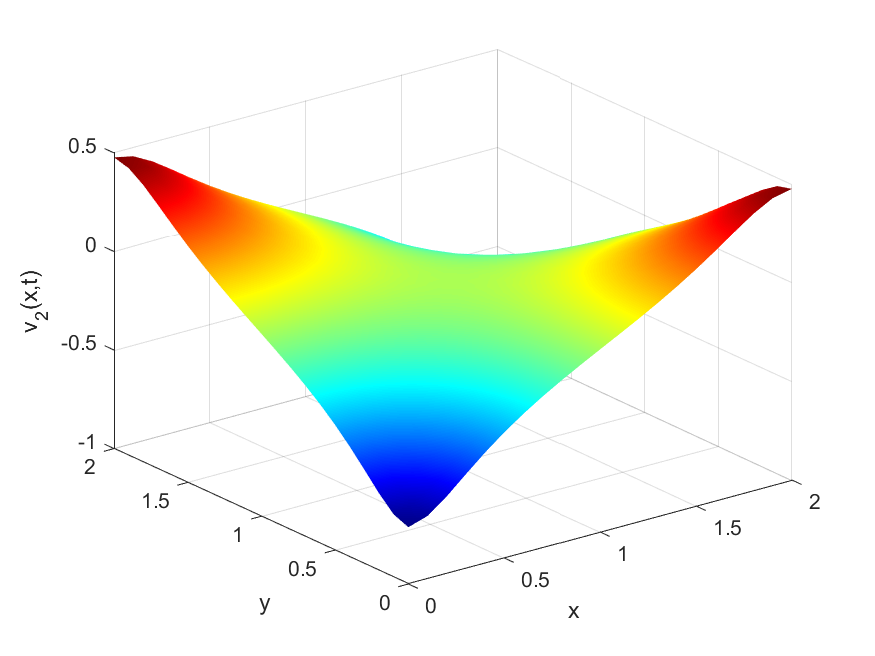}
\end{subfigure}
\\ 
\caption{The profiles of two double boundary spots to (\ref{ss}) with $\Omega=(0,2)\times (0,2).$  Here the rest parameters are the same as those stated in Figure \ref{fig:singlespike} except that $a_{12}=-1$ and $a_{21}=-2$.  These numerical findings suggest that the locations of concentrated cellular densities $u_1$ and $u_2$ are different when the interaction coefficients are negative.   }
\label{fig:repulsive}
\vspace{-0.0in}\end{figure}

\subsection{Discussion}
In this paper, we study the localized pattern formation in (\ref{timedependent}) under the asymptotical limits of $\chi_1,$ $\chi_2\rightarrow+\infty.$   Our main goal is to extend the results shown in \cite{kong2022existence} to the multi-species Keller-Segel model counterpart by employing the inner-outer gluing method.  Imposing some assumptions on the coefficient matrix $A$ such that $A$ is positive and irreducible, i.e. the interactions are both attractive, we show (\ref{ss}) admits the multi-spots given by (\ref{mainu}) and (\ref{mainv}). Compared to the core problem given in \cite{kong2022existence}, the core equation governing the profile of solution to (\ref{ss}) is still strongly coupled, which may cause the difficulty while establishing the linear theories especially in the inner region.  To overcome this, we borrow the ideas shown in \cite{lin2010profile,lin2013liouville,lin2013liouville1} and develop the inner linear theories stated in Section \ref{sect3}, where the extensive analysis of bounded kernels to the linearized Liouville system is crucial.  The main restriction in their results is the assumption that the coefficient matrix is positive.  As stated in Theorem \ref{thm11}, the profiles of cellular densities $u_1$ and $u_2$ are determined by the entire solutions $e^{\Gamma_1}$ and $e^{\Gamma_2}$ solving  (\ref{gamma}). 
 Unlike the governing profile shown in \cite{kong2022existence}, we do not have the explicit forms of $\Gamma_1$ and $\Gamma_2$ defined by (\ref{gamma}).  Whereas, we can still establish linear theories and perform the fixed point argument since the relations between algebraic decay rates of $e^{\Gamma_1}$, $e^{\Gamma_2}$  and their total mass are well understood.  

We would like to point out some intriguing research directions that deserve exploration in the future.  As discussed above, we only consider that the coefficient matrix is positive.  As shown in Figure \ref{fig:repulsive}, when this assumption is not satisfied, system (\ref{ss}) admit new types of concentrated patterns when $\chi_1$ and $\chi_2$ are large enough, where the locations of are $u_1$ and $u_2$ are at alternative corners.  The theoretical analysis for the existence of these stationary solutions is challenging but worthwhile.  Figure \ref{fig:singlespike2} reveals that the logistic multi-species Keller-Segel model under the small chemical diffusivity regime admits stable interior spots, which can not be detected under the large chemotactic movement regime.  Investigating the relevant pattern formation is open and presents an intriguing direction for future research. 
 It seems some ideas shown in \cite{kong2024existence} devoted to the single-species Keller-Segel model with logistic growth are beneficial.  

\section{Acknowledgments}
   Liangshun Xu is supported by NSFC (Grant No. 12301243) and Guangxi Young Talent  Project (Grant No. ZX02080031124004).
 J. Wei is partially supported  by   Hong Kong General Research Fund "New frontiers in singular limits of nonlinear partial differential equations.

\bibliographystyle{plain}
\bibliography{ref}

\begin{thebibliography}{10}

\bibitem{ao2016non}
Weiwei Ao, Chang-Shou Lin, and Juncheng Wei.
\newblock {\em On Non-Topological Solutions of the $A_2$ and {$B_2$} Chern-Simons System}.
\newblock Memoirs of the American Mathematical Society, 2016.

\bibitem{childress1981}
Stephen Childress and Jerome~K Percus.
\newblock Nonlinear aspects of chemotaxis.
\newblock {\em Mathematical Biosciences}, 56(3-4):217--237, 1981.

\bibitem{cortazar2020green}
Carmen Cort{\'a}zar, Manuel Del~Pino, and Monica Musso.
\newblock Green’s function and infinite-time bubbling in the critical nonlinear heat equation.
\newblock {\em Journal of the European Mathematical Society}, 22(1):283--344, 2020.

\bibitem{davila2024existence}
Juan Davila, Manuel del Pino, Jean Dolbeault, Monica Musso, and Juncheng Wei.
\newblock Existence and stability of infinite time blow-up in the {K}eller--{S}egel system.
\newblock {\em Archive for Rational Mechanics and Analysis}, 248(4):61, 2024.

\bibitem{davila2020singularity}
Juan D{\'a}vila, Manuel Del~Pino, and Juncheng Wei.
\newblock Singularity formation for the two-dimensional harmonic map flow into {$S^2$}.
\newblock {\em Inventiones mathematicae}, 219(2):345--466, 2020.

\bibitem{del2006}
Manuel del Pino and Juncheng Wei.
\newblock Collapsing steady states of the {K}eller--{S}egel system.
\newblock {\em Nonlinearity}, 19(3):661, 2006.

\bibitem{flex2021}
FlexPDE.
\newblock Solutions inc.
\newblock \textit{https://www.pdesolutions.com}, 2021.

\bibitem{herrero1996}
Miguel~A Herrero and Juan~JL Vel{\'a}zquez.
\newblock Chemotactic collapse for the {K}eller-{S}egel model.
\newblock {\em Journal of Mathematical Biology}, 35(2):177--194, 1996.

\bibitem{hillen2009}
Thomas Hillen and Kevin~J Painter.
\newblock A user’s guide to {PDE} models for chemotaxis.
\newblock {\em Journal of mathematical biology}, 58(1):183--217, 2009.

\bibitem{horstmann2004}
Dirk Horstmann.
\newblock From 1970 until present: the {K}eller-{S}egel model in chemotaxis and its consequences. ii.
\newblock {\em Jahresbericht der Deutschen Mathematiker-Vereinigung}, 106:51--69, 2004.

\bibitem{horstmann2005}
Dirk Horstmann and Michael Winkler.
\newblock Boundedness vs. blow-up in a chemotaxis system.
\newblock {\em Journal of Differential Equations}, 215(1):52--107, 2005.

\bibitem{huang2019existence}
Hsin-Yuan Huang.
\newblock Existence of bubbling solutions for the {L}iouville system in a torus.
\newblock {\em Calculus of Variations and Partial Differential Equations}, 58:1--26, 2019.

\bibitem{kong2024existence}
Fanze Kong, Michael~J. Ward, and Juncheng Wei.
\newblock {Existence, Stability and Slow Dynamics of Spikes in a 1D Minimal Keller--Segel Model with Logistic Growth}.
\newblock {\em Journal of Nonlinear Science}, 34(3):51, 2024.

\bibitem{kong2022existence}
Fanze Kong, Juncheng Wei, and Liangshun Xu.
\newblock Existence of multi-spikes in the {K}eller-{S}egel model with logistic growth.
\newblock {\em Mathematical Models Methods Applied Science}, 33(11):2227--2270, 2023.

\bibitem{lin2010profile}
Chang-Shou Lin and Lei Zhang.
\newblock Profile of bubbling solutions to a {L}iouville system.
\newblock {\em Annales de l'IHP Analyse non lin{\'e}aire}, 27(1):117--143, 2010.

\bibitem{lin2013liouville1}
Chang-shou Lin and Lei Zhang.
\newblock On {L}iouville systems at critical parameters, part 1: One bubble.
\newblock {\em Journal of Functional Analysis}, 264(11):2584--2636, 2013.

\bibitem{lin2013liouville}
Chang-shou Lin and Lei Zhang.
\newblock On {L}iouville systems at critical parameters, part 1: One bubble.
\newblock {\em Journal of Functional Analysis}, 264(11):2584--2636, 2013.

\bibitem{nanjundiah1973}
Vidyanand Nanjundiah.
\newblock Chemotaxis, signal relaying and aggregation morphology.
\newblock {\em Journal of Theoretical Biology}, 42(1):63--105, 1973.

\bibitem{painter2019mathematical}
Kevin~J Painter.
\newblock Mathematical models for chemotaxis and their applications in self-organisation phenomena.
\newblock {\em Journal of theoretical biology}, 481:162--182, 2019.

\bibitem{senba2000}
Takasi Senba and Takashi Suzuki.
\newblock Some structures of the solution set for a stationary system of chemotaxis.
\newblock {\em Advances in Mathematical Sciences and Applications}, 10(1):191--224, 2000.

\bibitem{senba2002}
Takasi Senba and Takashi Suzuki.
\newblock Weak solutions to a parabolic-elliptic system of chemotaxis.
\newblock {\em Journal of Functional Analysis}, 191(1):17--51, 2002.

\bibitem{turing1990chemical}
Alan~Mathison Turing.
\newblock The chemical basis of morphogenesis.
\newblock {\em Bulletin of mathematical biology}, 52:153--197, 1990.

\bibitem{wang2015time}
Qi~Wang, Jingyue Yang, and Lu~Zhang.
\newblock {Time-periodic and stable patterns of a two-competing-species {K}eller-{S}egel chemotaxis model: Effect of cellular growth}.
\newblock {\em Discrete and Continuous Dynamical Systems - B}, 2015.

\bibitem{wang2013mathematics}
Zhi-An Wang.
\newblock Mathematics of traveling waves in chemotaxis.
\newblock {\em Discrete \& Continuous Dynamical Systems-Series B}, 18(3), 2013.

\bibitem{wolansky1997critical}
G~Wolansky.
\newblock A critical parabolic estimate and application to nonlocal equations arising in chemotaxis.
\newblock {\em Applicable Analysis}, 66(3-4):291--321, 1997.

\bibitem{wolansky2002multi}
Gershon Wolansky.
\newblock Multi-components chemotactic system in the absence of conflicts.
\newblock {\em European Journal of Applied Mathematics}, 13(6):641--661, 2002.

\end{thebibliography}

\begin{appendices}
\section{Formal Construction of Spots}\label{appena}
In Appendix \ref{appena}, we shall employ the matched asymptotic analysis to reconstruct the multi-spots (\ref{mainu}) and (\ref{mainv}), which are complementary of our rigorous analysis.  Without loss of generality, we only consider the single interior spot case.

First of all, let $\bar v_1=\chi_1 v_1$ and $\bar v_2=\chi_2 v_2$ in (\ref{ss}), we have
\begin{align}\label{timeschi}
\left\{\begin{array}{ll}
0=\Delta u_1-\nabla\cdot(u_1\nabla \bar v_1)+ \lambda_1 u_1(\bar u_1-u_1),&x\in\Omega,\\
0=\Delta u_2-\nabla\cdot(u_2  \nabla\bar v_2)+ \lambda_2 u_2(\bar u_2-u_2),&x\in\Omega,\\
0=\Delta \bar v_1-\bar v_1+a_{11}\chi_1 u_1+a_{12}\chi_1 u_2,&x\in\Omega,\\
0=\Delta \bar v_2-\bar v_2+a_{21}\chi_2 u_1+a_{22}\chi_2 u_2,&x\in\Omega,\\
\partial_{\textbf{n}}u_1=\partial_{\textbf{n}}u_2=\partial_{\textbf{n}}v_1=\partial_{\textbf{n}}v_2=0,&x\in\partial\Omega.
\end{array}
\right.
\end{align}
Let $\chi_1=\frac{1}{\varepsilon^2}\gg 1$ with $\chi_2=\gamma\chi_1$ and in the inner region, we introduce
$$y=\frac{x-\xi}{\varepsilon},~U_i(y)=\bar u_i(x),~\bar V_i(y)=\bar v_i(x),~i=1,2,$$
where $\xi$ denotes the center. 

Then, we expand 
$$U_{i}(y)=U_{i0}+\varepsilon^2 U_{i1}+\cdots,~\bar V_{i}(y)=\bar V_{i0}+\varepsilon^2 \bar V_{i1}+\cdots,$$
and obtain from (\ref{timeschi}) that the leading order equation is 
\begin{align}\label{timeschi1}
\left\{\begin{array}{ll}
0=\Delta U_{10}-\nabla\cdot(U_{10}\nabla \bar V_{10}),&y\in\mathbb R^2,\\
0=\Delta U_{20}-\nabla\cdot(U_{20}  \nabla\bar V_{20}),&y\in\mathbb R^2,\\
0=\Delta \bar V_{10}+a_{11} U_{10}+a_{12} U_{20},&y\in\mathbb R^2,\\
0=\Delta \bar V_{20}+a_{21}\gamma U_{10}+a_{22}\gamma U_{20},&y\in\mathbb R^2.
\end{array}
\right.
\end{align}
Then we let $\nu:=-\frac{1}{\log\varepsilon},$ further expand for $i=1,2,$
$$U_{i0}=U_{i0,0}+\nu U_{i0,1}+\cdots,~~\bar V_{i0}=\frac{1}{\nu}\bar V_{i0,0}+\bar  V_{i0,1}+\cdots.$$
Upon substituting the expansion into (\ref{timeschi1}), one has $\bar V_{i0,0}=D_{i0}$ with constant $D_{i0}$ determined later on and
\begin{align}\label{timeschi11}
\left\{\begin{array}{ll}
0=\Delta U_{10,0}-\nabla\cdot(U_{10,0}\nabla \bar V_{10,1}),&y\in\mathbb R^2,\\
0=\Delta U_{20,0}-\nabla\cdot(U_{20,0}  \nabla\bar V_{20,1}),&y\in\mathbb R^2,\\
0=\Delta \bar V_{10,1}+a_{11} U_{10}+a_{12} U_{20},&y\in\mathbb R^2,\\
0=\Delta \bar V_{20,1}+a_{21}\gamma U_{10}+a_{22}\gamma U_{20},&y\in\mathbb R^2.
\end{array}
\right.
\end{align}
Then, we find $U_{i0,0}=C_{i,0}e^{\bar V_{i0,1}}$ with $i=1,2$ and
\begin{align}\label{timeschi111}
\left\{\begin{array}{ll}
0=\Delta \bar V_{10,1}+a_{11} C_{1,0}e^{\bar V_{10,1}}+a_{12}C_{2,0}e^{\bar V_{20,1}} ,&y\in\mathbb R^2,\\
0=\Delta \bar V_{20,1}+a_{21}\gamma C_{1,0}e^{\bar V_{10,1}}+a_{22}\gamma C_{2,0}e^{\bar V_{20,1}},&y\in\mathbb R^2.
\end{array}
\right.
\end{align}
Let $\tilde y=\sqrt{C_{1,0}}y$, then (\ref{timeschi111}) becomes
\begin{align}\label{timeschi1111}
\left\{\begin{array}{ll}
0=\Delta \bar V_{10,1}+{a_{11}}e^{\bar V_{10,1}}+\frac{a_{12}C_{2,0}}{C_{1,0}}e^{\bar V_{20,1}} ,&\tilde y\in\mathbb R^2,\\
0=\Delta \bar V_{20,1}+a_{21}\gamma e^{\bar V_{10,1}}+a_{22}\gamma \frac{C_{2,0}}{C_{1,0}}e^{\bar V_{20,1}},&\tilde y\in\mathbb R^2.
\end{array}
\right.
\end{align}
Assume that
\begin{align}\label{assumeappendix}
a_{12}C_{20}=a_{21}\gamma C_{10},
\end{align}
then as shown in \cite{lin2010profile}, we have (\ref{timeschi1111}) admits a family of solution pair ($\bar V_{10,1}$,$\bar V_{20,1}$) depending on ($\mu_1$,$\mu_2$).  It follows that 
\begin{align}\label{ui002025213}
U_{i0,0}(y)=C_{i,0}e^{\bar V_{i0,1}(\tilde y)},~i=1,2.
\end{align}
Moreover, the integral constraints imply for $i=1,2,$ 
$$C_{i,0}=\frac{\bar u_i\int_{\mathbb R^2} e^{\bar V_{i0,1}}d\tilde y}{\int_{\mathbb R^2} e^{2\bar V_{i0,1}}d\tilde y}.$$
Of concern (\ref{assumeappendix}), we find the following condition is assume to hold
$$\frac{\int_{\mathbb R^2}e^{ \bar V_{10,1}}\,d\tilde y\bar u_1}{\int_{\mathbb R^2}e^{ \bar V_{20,1}}\,d\tilde y\bar u_2}=\frac{a_{12}}{a_{21}}\frac{\chi_1}{\chi_2}\frac{\int_{\mathbb R^2}e^{2 \bar V_{10,1}}\,d\tilde y}{\int_{\mathbb R^2}e^{2 \bar V_{20,1}}\,d\tilde y}.$$
In addition, we have by Pohozaev identity that 
$$4(\sigma_1+\sigma_2) =b_{11}\sigma_1^2+2b_{12}\sigma_1\sigma_2+b_{22}\sigma_2^2,$$
where $b_{ij}$, $i,j=1,2$ are given by \eqref{bvaluenew2024104} and
$$\sigma_i=\frac{1}{2\pi}\int e^{\bar V_{i0,1}}d\tilde y,~i=1,2.$$

{In the outer region}, we approximate $u_1$ and $u_2$ as Dirac-delta functions to obtain for $i=1,2$,
\begin{align*}
\Delta \bar v_1-\bar v_1\sim - \bigg(a_{11}\int_{\mathbb R^2} U_{10,0}d\tilde y+a_{12}\int_{\mathbb R^2} U_{20,0}d\tilde y\bigg)\delta(x-\xi),
\end{align*}
and
\begin{align*}
\Delta \bar v_2-\bar v_2\sim - \bigg(a_{21}\gamma\int_{\mathbb R^2} U_{10,0}d\tilde y+a_{22}\gamma\int_{\mathbb R^2} U_{20,0}d\tilde y\bigg)\delta(x-\xi),
\end{align*}
where $U_{10,0}$ and $U_{20,0}$ are defined by \eqref{ui002025213}.
Hence,
\begin{align}
\bar v_i=2\pi m_i G(x;\xi),
\end{align}
where $m_i$ are given in \eqref{m1m2decayrate} and $G$ is the Neumann reduced-wave Green's function satisfying
\begin{equation}
   \left\{\begin{array}{ll}
   \Delta G-G=-\delta(x-\xi),&x\in\Omega,\\
   \partial_{\textbf{n}}G=0,&x\in\partial\Omega.
   \end{array}
   \right.
\end{equation}


Finally, we match the inner and outer solutions to determine parameters $D_{i0}$, $\mu_i$ with $i=1,2$.  Recall that as $\tilde y\rightarrow+\infty,$ the inner solution satisfies 
$$\bar V_{i0}\sim \frac{1}{\nu}D_{i0}+\bar \mu_i+m_i\log |\tilde y|.$$
On the other hand, the far-field behavior of the outer solution is
$$\bar v_{i}\sim 2\pi m_i\bigg[-\frac{1}{2\pi}\log|x-\xi|+H(\xi,\xi)+\nabla H(\xi,\xi)\cdot(x-\xi)\bigg].$$
Since $\frac{\tilde y}{\sqrt{C_{1,0}}}=\frac{x-\xi}{\varepsilon}$, we match the inner and outer solutions implies
$$ D_{i0}= m_i,~~\mu_i=2\pi m_iH(\xi,\xi),~~\nabla H(\xi,\xi)=0,$$
where $\mu_i:=\bar \mu_i-m_i\log\sqrt{C_{1,0}}.$


\end{appendices}

\end{document}